\documentclass[a4paper]{amsart}
\usepackage[utf8]{inputenc}
\usepackage[T1]{fontenc}

\usepackage{geometry}
\usepackage{amssymb}
\usepackage{amsfonts}
\usepackage{amsthm}
\usepackage{calrsfs}
\usepackage{array}
\usepackage{bbm}
\usepackage{stmaryrd}
\usepackage{hyperref}
\usepackage{enumerate}
\usepackage{mathabx}
\usepackage{enumitem}
\usepackage{graphicx}
\usepackage{caption}
\usepackage{subcaption}
\usepackage{mathtools}
\mathtoolsset{showonlyrefs}

\usepackage{csquotes}

\usepackage{mathtools}
\mathtoolsset{showonlyrefs=true}
\usepackage[svgnames]{xcolor}
\usepackage{hyperref}

\usepackage{tikz}
\usetikzlibrary{positioning}

\tikzset{
	dotA/.style={
		transform shape, fill,circle,inner sep=1.5pt, label distance=-1pt,
		font={\normalsize }
	},
	>=stealth,
}

\usepackage{anysize}

\marginsize{2cm}{2cm}{1.5cm}{1.5cm}

\newcommand{\R}{\mathbb{R}}

\newcommand{\complex}{\mathbb{C}}

\newcommand{\jap}[1]{\langle #1 \rangle}

\newcommand{\supp}{supp}

\newcommand{\dem}{\noindent \textit{Proof: }}
\renewcommand{\k}{\kappa}
\newcommand{\E}{\mathcal{E}}

\newcommand{\K}{\mathcal{K}}

\newcommand{\M}{\mathcal{M}}

\newcommand{\T}{\mathcal{T}}

\newcommand{\norma}[1]{{\left\vert\kern-0.25ex\left\vert\kern-0.25ex\left\vert #1 
    \right\vert\kern-0.25ex\right\vert\kern-0.25ex\right\vert}}


\renewcommand{\k}{\kappa}

\DeclareMathOperator{\sgn}{\mathrm{sgn}}

\newtheorem{theorem}{Theorem}[section]

\newtheorem{lemma}[theorem]{Lemma}

\theoremstyle{definition}
\newtheorem{definition}[theorem]{Definition}
\newtheorem{remark}[theorem]{Remark}
\newtheorem{example}[theorem]{Example}

\newtheorem*{theorem*}{Theorem}

\theoremstyle{plain}
\newtheorem{thm}{Theorem}[section]
\newtheorem*{thm*}{Theorem}
\newtheorem{prop}[thm]{Proposition}

\theoremstyle{definition}

\theoremstyle{remark}

\numberwithin{equation}{section}
\newtheoremstyle{mytheoremstyle} 
{\topsep}                    
{\topsep}                    
{}                   
{}                           
{\scshape}                   
{.}                          
{.5em}                       
{}  

\theoremstyle{mytheoremstyle} 
\theoremstyle{mytheoremstyle} 

\makeatletter
\@namedef{subjclassname@2020}{%
	\textup{2020} Mathematics Subject Classification}
\makeatother

\date{}
\author{Simão Correia, Gonçalo Pereira and Thyago S.R. Santos}

\title{Small-Amplitude Self-Similar Solutions for one-dimensional Nonlinear Dispersive Equations}



\thanks{S. C. was partially supported by Funda\c{c}\~ao para a Ci\^encia e Tecnologia, through CAMGSD, IST-ID (project UID/04459/2025) and T.S.R.S was partially supported by FAPESP Grant No. 2024/15587-1.}

\allowdisplaybreaks

\begin{document}

\begin{abstract}
Given a nonlinear dispersive equation which admits a scaling invariance, there may exist self-similar solutions. In this work, we present a systematic approach for the construction of small-amplitude self-similar solutions, together with precise asymptotic descriptions at both small and large frequency scales. These ideas are then applied to three classic dispersive models: the modified Benjamin-Ono, the quartic Korteweg-de Vries and the cubic nonlinear Schrödinger equations.
\end{abstract}

\maketitle


\section{{Introduction}}

\subsection{Setting of the problem}

In mathematical problems which present a scale-invariance, self-similarity refers to the property of an object which remains unchanged as one moves through different scales. These objects connect microscopic and macroscopic scales and thus have been the pursuit of numerous mathematical and numerical studies.
In the present article, we aim to provide a roadmap to prove the existence of small-amplitude self-similar solutions to one-dimensional dispersive equations, together with precise asymptotic descriptions of both small and large (self-similar) scales. To that end, we will concentrate in three dispersive models: 
\begin{itemize}
    \item the cubic nonlinear Schrödinger equation,
 \begin{equation}\label{3NLS}
iq_t + q_{xx} + |q|^2q= 0, 
\tag{NLS}
\end{equation}
{one of the most fundamental nonlinear dispersive equations, with applications ranging from nonlinear optics, water waves to plasma physics and Bose--Einstein condensation \cite{Hasegawa_Tappert_1973,Zakharov_1968,Karpman_1975,Gross_1961, Sulem_Sulem_1999}. A particularly interesting application links \eqref{3NLS} with the dynamics of vortex filaments, i.e., flows with vorticity concentrated over a three-dimensional curve \cite{DaRios1906}, as discovered by Hasimoto \cite{Hasimoto} (see also \cite{PinkallGross_hasimoto}).  This connection has enabled the analysis of self-similar filaments, such as two half-lines meeting at a vertex, through the nonlinear Schrödinger equation \cite{BV13, Banica_Vega_2007,Banica_Vega_2008, GVV_binormal}.
}
\item  the modified Benjamin-Ono equation, 
\begin{equation}\label{dnls}
w_t +  Hw_{xx} + (w^3)_x= 0, 
\tag{mBO}
\end{equation}
  where $Hf(x) := (-i \sgn(\xi) \widehat{f}(\xi))^\vee(x)$ represents the Hilbert transform, is a fundamental dispersive model that describes one-dimensional internal waves with strong nonlinearity in deep stratified fluids, especially in the ocean. This model also plays a central role in analyzing the propagation and interaction of solitary waves in plasma and unconventional optical media, capturing phenomena such as modulational instabilities and rogue wave formation. Recent studies highlight its importance in describing abundant wave structures, multi-soliton solutions, and nonclassical symmetries, demonstrating its relevance for both physical applications and advances in applied mathematics and analytical methods (\cite{Yu_Song_2024}, \cite{Khater_2022}, \cite{Apeanti_2025}, \cite{IP_AP741}, \cite{KenigTakaoka2006}).
\item the quartic Korteweg-de Vries equation,
\begin{equation}\label{4Kdv}
v_t +  v_{xxx} + (v^4)_x= 0,
\tag{4KdV}
\end{equation}
which is a higher-order version of the classical modified Korteweg-de Vries equation
\begin{equation}\label{mKdV}
    u_t +  u_{xxx} \pm (u^3)_x= 0.\tag{mKdV}
\end{equation}
This last model is known to describe 
acoustic waves in certain anharmonic lattices \cite{ZABUSKY1967223} and Alfvén waves in cold, collisionless plasmas \cite{KakutaniOno}. Similar to the \eqref{NLS} case, the modified KdV equation can also be linked to a geometric problem in planar fluid mechanics. More precisely, Goldstein and Petrich \cite{GoldsteinPetrich} proposed a model for the time evolution  of the boundary of a two-dimensional vortex patch under Euler's equations. Under this model, it can be seen that the curvature of the boundary satisfies \eqref{mKdV}. As such, self-similar solutions to \eqref{mKdV} relate directly to self-similar boundaries, such as two half-lines meeting at a corner or logarithmic spirals.
\end{itemize}
A common feature among all four dispersive models is the existence of a scaling invariance, namely
\begin{align*}
q_\lambda(t,x)&:=\lambda q(\lambda^2 t, \lambda x) \mbox{ for \eqref{3NLS}},\qquad &w_\lambda(t,x):=\lambda^{\frac12}w(\lambda^2t, \lambda x) \mbox{ for \eqref{dnls}},
\\
v_\lambda(t,x)&:=\lambda^{\frac23} v(\lambda^3 t, \lambda x) \mbox{ for \eqref{4Kdv}},\qquad &u_\lambda(t,x):=\lambda u(\lambda^3t, \lambda x) \mbox{ for \eqref{mKdV}}.
\end{align*}
As such, all four models may have \textit{self-similar solutions}, i.e., solutions which remain invariant by the scaling of the equation. In particular, these solutions must be of the form
\begin{align*}
q(t,x)&=\frac{1}{t^{\frac12}} Q\left(\frac{x}{t^{\frac12}}\right) \mbox{ for \eqref{3NLS}},\qquad &w(t,x)=\frac{1}{t^{\frac14}} W\left(\frac{x}{t^{\frac12}}\right) \mbox{ for \eqref{dnls}},
\\
v(t,x)&=\frac{1}{t^{\frac29}} V\left(\frac{x}{t^{\frac13}}\right) \mbox{ for \eqref{4Kdv}},\qquad &u(t,x)=\frac{1}{t^{\frac13}} U\left(\frac{x}{t^{\frac13}}\right) \mbox{ for \eqref{mKdV}},
\end{align*}
where the profile functions $Q, W, V$ and $U$ must satisfy, respectively,

\begin{equation}\label{edo associada 3NLS}
Q''(y) + |Q(y)|^2 Q(y) - \frac{i}{2}\big(Q(y) + y Q'(y)\big) = 0,
\end{equation}
\begin{equation}\label{edo associada DNLS}
HW''(y) - \frac{1}{2}\,y\, W'(y) - \frac{1}{4}\, W(y) + 3 W(y)^2W^\prime(y) = 0,
\end{equation}
\begin{equation}\label{edo associada 4kdV}
 V'''(y) + 4 V(y)^3 V'(y) - \frac{1}{3} y V'(y) - \frac{2}{9} V(y) = 0
\end{equation}
and
\begin{equation}\label{edo associada mkdV}
 U'''(y) \pm 3 U(y)^2 U'(y) - \frac{1}{3} y U'(y) - \frac{1}{3} U(y) = 0.
\end{equation}
Self-similar solutions have been shown to influence the global dynamics of the dispersive model: they appear in the asymptotic behavior of small solutions for large times \cite{HayashiNaumkin_mkdv2, GPR_mkdv, DeiftZhou}, are essential in the construction of finite-time blow-up solutions \cite{MerleRaphael} and have a stable blow-up behavior, either at initial time or at blow-up time \cite{lan, CollotRaphaelSzeftel,CC24, cote2024sharp} (see also the numerical experiments in \cite{Chapman_et_al}). Particularly essential to the proofs of these qualitative results is the precise description of the self-similar profile, either in physical space (via spectral properties) or in frequency space (via dispersive estimates). In particular, even if the existence of self-similar solutions is already known, a great deal of effort should be made in providing the tools to handle the interplay between these critical objects and the nonlinear dynamics.

\medskip

Even though all four models present a scale invariance, each equation possesses its own particular features: 
\begin{itemize}

    \item \textbf{Scattering properties.} Given a nonlinear dispersive equation, it may occur that globally defined solutions scatter at $t=+\infty$, i.e., that they behave as linear solutions. In other words, the total nonlinear effects over $t\in \R^+$ can be shown to have a finite influence in the dynamics. In some particular cases, (linear) scattering is impossible and must be replaced with \emph{modified} scattering: at $t=+\infty$, solutions behave linearly \emph{after} a phase correction.

In the cases of \eqref{3NLS}, \eqref{dnls} and \eqref{mKdV}, Hayashi and Naumkin \cite{MR1491867, HayashiNaumkin_mkdv1, HayashiNaumkin_mkdv2, HayashiNaumkin_nls, MR1491867} proved the existence of modified scattering behavior for small solutions. These results have been revisited in \cite{KatoPusateri} (for \eqref{3NLS}) and in \cite{CCV21, GPR_mkdv} (for \eqref{mKdV}). In the \eqref{4Kdv} case, linear scattering was proved by Tao \cite{Tao_4kdv} as a direct consequence of the Cauchy theory for small initial data over the scaling-critical space $\dot{H}^{-\frac16}(\R)$.
    
    \item \textbf{Critical regularity scale.} Given a scaling invariance, one may look for scaling-invariant spaces of initial data. These spaces lie in the so-called critical scale, which combines integrability, differentiability and summability in an optimal balance. A common collection of spaces to consider is the $H^s(\R)$ class, where criticality is determined completely through the regularity parameter $s$. A simple computation shows that the critical regularity $s_c$ for each of the four equations is
    $$
    s_c\eqref{3NLS}=s_c\eqref{mKdV}=-\frac12,\quad s_c\eqref{4Kdv}=-\frac16,\quad s_c\eqref{dnls}=0.
    $$
    In our analysis, a more relevant scale is given by the Fourier-Lebesgue class
    $$
    \mathcal{F}L^{s,\infty}(\R):=\{u\in \mathcal{S'}(\R): \|\jap{\xi}^s\widehat{u}\|_{L^\infty_\xi} <\infty\}, 
    $$
    in which the critical regularities are
      $$
    s_{c,\infty}\eqref{3NLS}=s_{c,\infty}\eqref{mKdV}=0,\quad s_{c,\infty}\eqref{4Kdv}=\frac13,\quad s_{c,\infty}\eqref{dnls}=\frac12.
    $$
As self-similar solutions are, by definition, scaling-invariant, their profiles will lie in critical spaces. As such, these critical Fourier-Lebesgue exponents  will appear later on in our proofs.

\item \textbf{Derivative loss in the nonlinearity.} The presence of a derivative in the nonlinearities of \eqref{dnls}, \eqref{mKdV} and \eqref{4Kdv} leads to several difficulties in the qualitative analysis of the dynamics: in the process of deriving suitable \textit{a priori} bounds, one must find a way to recover from such a derivative loss in other to have any hope of bootstrapping the nonlinear estimates.

In our context, this derivative loss is not a hindrance at all, as we will work with nonlinear operators that already incorporate the derivative loss. If anything, this loss will actually be \emph{helpful}, since it implies that, formally, the equation preserves the \emph{average} of the solution,
$$
\int_\R w(t,x)dx=\int_\R w(0,x)dx,\quad \int_\R v(t,x)dx=\int_\R v(0,x)dx,\quad \int_\R u(t,x)dx=\int_\R u(0,x)dx.
$$
\end{itemize}

As we will see below, the construction of small-amplitude self-similar solutions, together with their asymptotic behavior at small and large scales, is directly connected to these three features. Our rationale for the choice of the nonlinear dispersive models was driven by looking for different combinations of such features:
\begin{itemize}
    \item \eqref{3NLS}: modified scattering, $s_{c,\infty}=0$, no average conservation.
    \item \eqref{dnls}: modified scattering, $s_{c,\infty}=1/2$, preserves averages.
    \item \eqref{4Kdv}: linear scattering, $s_{c,\infty}=1/3$, preserves averages.
    \item \eqref{mKdV}: modified scattering, $s_{c,\infty}=0$, preserves averages.
\end{itemize}
Other important properties which greatly impact the dynamics of each equation include complete integrability and focusing/defocusing nonlinearities. In our context, these properties will have no impact in the construction of (small) self-similar solutions.

\subsection{Previous works} The pursuit of self-similar solutions in nonlinear dispersive models has attracted many experts in the field and lead to the development of different techniques, ranging from ODE techniques \cite{DonningerSchorkhuber, PerelmanVega, Troy, RottschaferKaper, BonaWeissler, FGO20, KavianWeissler}, bifurcation arguments \cite{BahriMartelRaphael, Koch_gkdv} and complete integrability methods \cite{DunstKokocki}, to local well-posedness for scaling-invariant initial data \cite{Planchon_wave, KatoOzawa_wave, Germain_wave_map, CazWei_asympt_ss, CazWei_more_ss, CDEW, GuoCui_beam}.

\begin{itemize}
    \item In the context of the \eqref{3NLS} equation, the results on self-similar solutions are essentially nonexistent: the small-data techniques of \cite{CazWei_asympt_ss, CazWei_more_ss} only cover large nonlinear powers, while the ODE techniques of \cite{RottschaferKaper, DonningerSchorkhuber, Troy, KavianWeissler} fail to cover the one-dimensional cubic case.

    \item As in equation \eqref{3NLS}, there are no known results in the literature concerning self-similar solutions for \eqref{dnls}. Notice that \eqref{edo associada DNLS} is a nonlocal equation, making the application of ODE techniques substantially more difficult.
    
\item For \eqref{4Kdv}, the only work proving the existence of (small) self-similar solutions is \cite{MVW25} by Molinet, Vento and Weissler. Their remarkable argument refines the aforementioned small data arguments (which rely only the existence of a suitable well-posedness theory at critical regularity) and uses the additional structure of a self-similar solution to propose a well-behaved integral operator whose fixed points are precisely the self-similar profiles. Their proof holds for lower dispersions (strictly above two derivatives and below three) and for larger exponents, but fails to produce any information on the asymptotic properties of the self-similar profile. Our Theorem \ref{teorema Self Similar 4KdV} will greatly complement their results, providing the asymptotic behavior of the profile.

\item For the modified KdV equation, there have been several works. When \eqref{edo associada mkdV} is reducible to a Painlevé II equation, strong results can be achieved \cite{DeiftZhou,HastingsMcLeod}. In \cite{PerelmanVega}, the authors treated the focusing case (with a $+$ sign in the nonlinearity). Particularly relevant to the present article is the work of the first author, Côte and Vega \cite{CCV20}, where a new approach, based on pure Fourier analysis and stationary phase arguments, was developed. More precisely, the authors prove the following result\footnote{The result is presented in the defocusing case. However, the sign of the nonlinear term is irrelevant and impacts only the definition of the ansatz \cite{CCV20}}: 

\begin{thm}[Self-similar solutions for \eqref{mKdV}, \cite{CCV20}]\label{teorema Self Similar mKdV}
Let $ \kappa \in \left(\frac{1}{2}, \frac{4}{7}\right) $ and $ c \in \mathbb{C} $ with $ |c| \ll 1 $. Then, there exist a constant $ A = A(c) \in \mathbb{C} $ and a tempered distribution $ U_c \in \mathcal{S}^\prime$ such that the function
$$
u(t,x) := t^{-\frac{1}{3}} U_c\left(t^{-\frac{1}{3}} x\right)
$$
is a (real-valued) self-similar solution to  \eqref{mKdV}. Furthermore, the function $ U_c $ is characterized by the relation
$$
e^{-i \xi^3}\widehat{U}_c(\xi) = S_{A(c)}(\xi)+ z(\xi),
$$
where 
\begin{equation}\label{eq:SA_mkdv}
    S_A(\xi)= e^{-\frac{3i}{4\pi}|A|^2\log |\xi|}\left(A + \frac{3i|A|^2Ae^{-i\frac{3}{4\pi}|A|^2\log 3}}{16\sqrt{2}\pi\xi^3} e^{-\frac{3i}{2\pi}|A|^2\log |\xi|-\frac{8i\xi^3}{9}}\right), \quad \xi\gg 1, \qquad S_A(\xi)=\overline{S_A(-\xi)},
\end{equation}
and the remainder term $ z(\xi)= \overline{z(-\xi)} $ satisfies
$$
z(0^+)=c\quad \text{and} \quad \| \langle \xi \rangle^{\kappa} z\|_{L_\xi^\infty} + \| \langle \xi \rangle^{\kappa+1} z^\prime \|_{L_\xi^\infty} \lesssim |c|.
$$
\end{thm}
The precise knowledge of the asymptotics in frequency space can then be used to recover bounds in physical space. We point out that these asymptotics were crucial in the subsequent works of the first author and Côte \cite{CC24, cote2024sharp}, where the blow-up stability of self-similar solutions was proved.
\end{itemize}

\subsection{Main results}

Our goal is to extend the methodology presented in \cite{CCV20} to derive existence results for small-amplitude self-similar solutions in the lines of Theorem \ref{teorema Self Similar mKdV}.

To that end, let 
$\chi \in C^\infty(\mathbb{R})$ be a smooth cut-off function defined by\footnote{The choice of cut-off impacts only the smallness of the parameters in the main theorems.}
\begin{equation}\label{cutoff function}
\chi(\xi) =
\begin{cases}
0, & \text{if } \xi < 1/2, \\
1, & \text{if } \xi > 1.
\end{cases}    
\end{equation}

\begin{thm}[Self-similar solutions for \eqref{4Kdv}]\label{teorema Self Similar 4KdV}
Let $ \kappa \in \left(\frac{5}{8}, \frac{2}{3}\right) $ and $ c \in \mathbb{C} $ with $ |c| \ll 1 $. Then, there exist a constant $ A = A(c) \in \mathbb{C} $ and a tempered distribution $ V_c \in \mathcal{S}^\prime$ such that the function
$$
v(t,x) := t^{-\frac{2}{9}} V_c\left(t^{-\frac{1}{3}} x\right)
$$
is a (real-valued) self-similar solution to  \eqref{4Kdv}. Furthermore, the function $ V_c $ is characterized by the relation
$$
e^{-i \xi^3} \xi^{1/3} \widehat{V}_c(\xi) = S_{A(c)}(\xi)+ z(\xi),
$$
where
$$
S_A(\xi):= A \cdot \chi(\xi) + \overline{A} \cdot \chi(-\xi),
$$
and the remainder term $ z(\xi)= \overline{z(-\xi)} $ satisfies
$$
z(0^+)=c\quad \text{and} \quad \| \langle \xi \rangle^{\kappa} z\|_{L_\xi^\infty} + \| \langle \xi \rangle^{\kappa+1} z^\prime \|_{L_\xi^\infty} \lesssim |c|.
$$
\end{thm}

{
\begin{thm}[Self-similar solutions for \eqref{dnls}]\label{teorema Self Similar_dnls}
Let $ \kappa \in \left(0, \frac{1}{4}\right) $ and $ c \in \mathbb{R} $ with $ |c| \ll 1 $. Then, there exist a constant $ A = A(c) \in \mathbb{C} $ and a tempered distribution $ W_c \in \mathcal{S}^\prime$ such that the function
$$
w(t,x) := t^{-\frac{1}{4}} W_c\left(t^{-\frac{1}{2}} x\right)
$$
is a self-similar solution to \eqref{dnls}. Furthermore, the function $ W_c $ is characterized by the relation
$$
e^{-i \xi|\xi|} |\xi|^{1/2} \widehat{W}_c(\xi) = S_{A(c)}(\xi)+ z(\xi),
$$
where
\begin{equation}\label{eq:ansatz_mbo}
    S_A(\xi):= \left(Ae^{ia\log \eta} + B \frac{e^{2i\eta^2/3+3ia\log\eta}}{\eta^2}\right)\chi(\eta),\quad \xi>0,\qquad S_A(-\xi)=\overline{S_A(\xi)},
\end{equation}
and the parameters $a,B$ are defined as
\begin{equation}
    a=\frac{12|A|^2+3|c|^2}{8\pi}.
\end{equation}
\begin{equation}
    B=\frac{3\sqrt{3} iA^3}{8\pi}.
\end{equation}
 The remainder term $ z(\xi)= \overline{z(-\xi)} $ satisfies
$$
z(0)=c\quad \text{and} \quad \| \langle \xi \rangle^{\kappa} z\|_{L_\xi^\infty} + \| \langle \xi \rangle^{\kappa+1} z^\prime \|_{L_\xi^\infty} \lesssim |c|.
$$
\end{thm}

\label{remark:dnls}}

\begin{thm}[Self-similar solutions for \eqref{3NLS}]\label{teorema Self Similar nls}
Let $ \kappa \in \left(0, \frac{1}{2}\right) $ and $ A \in \mathbb{C} $ with $|A| \ll 1 $. Then, there exists a tempered distribution $ Q_A \in \mathcal{S}^\prime$ such that the function
$$
q(t,x) := t^{-\frac{1}{2}} Q_A\left(t^{-\frac{1}{2}} x\right)
$$
is a self-similar solution to \eqref{3NLS}. Furthermore, the function $ Q_A $ is characterized by the relation
$$
e^{-i \xi^2} \widehat{Q}_A(\xi) = S_A(\xi)+ z(\xi),
$$
where
\begin{equation}\label{eq:ansatz_nls}
    S_A(\xi):= Ae^{-\frac{i|A|^2}{2\pi}\log|\xi|} \cdot \chi(\xi) + A e^{\frac{i|A|^2}{2\pi}\log|\xi|}\cdot \chi(-\xi)
\end{equation}
and the remainder term satisfies
$$
\|(\log |\xi|)^{-1}z\|_{L^\infty(|\xi|<1)} +\||\xi|z'\|_{L^\infty(|\xi|<1)} +\| |\xi|^{\kappa} z\|_{L^\infty(|\xi|>1)} + \| |\xi|^{\kappa+1} z^\prime \|_{L^\infty(|\xi|>1)} \lesssim |A|.
$$
\end{thm}
\begin{remark}
    In Theorems \ref{teorema Self Similar mKdV}, \ref{teorema Self Similar 4KdV} and \ref{teorema Self Similar_dnls}, the self-similar solution is parametrized by $c\in \complex$, which describes the behavior near the zero frequency. We chose to do so as the parameter $c$ relates with the average-preservation property of each equation (and thus of physical significance). In the \eqref{3NLS} case, since the flow does not preserve the zero frequency, the parameter $c$ is of little relevance. As such, in  Theorem \ref{teorema Self Similar nls}, we parametrized instead self-similar solutions by their high-frequency behavior.
\end{remark}

\subsection{Overview of the method}

\subsubsection{{Review of the \eqref{mKdV} case.}} As shown in \cite{CCV20}, the self-similar profile $U$, solution to \eqref{edo associada mkdV}, is such that
$$
\tilde{U}(\xi):=e^{-i\xi^3}\widehat{U}(\xi)
$$
satisfies the fixed-point equation
\begin{equation}\label{eq:ss_mkdv}
\tilde{U}(\xi)=\tilde{U}(0^+) + \frac{3i}{4\pi^2}\int_0^\xi \iint\limits_{\eta=\eta_1+\eta_2+\eta_3} e^{-i\Phi}\tilde{U}(\eta_1){\tilde{U}(\eta_2)}\tilde{U}(\eta_3)d\eta_1 d\eta_2 d\eta,\qquad \Phi=\eta^3-\eta_1^3-\eta_2^3-\eta_3^3.
\end{equation} 
There exists another way to reach \eqref{eq:ss_mkdv}, which is perhaps a bit more enlightening: going back to \eqref{mKdV}, define the interaction representation $f$ of the solution $u$ as
$$
f(t,x)=e^{-t\partial_x^3}u(t,x),
$$
where $\{e^{-t\partial_x^3}\}_{t \in \R} $ is the linear group associated with \eqref{mKdV}. 
Using Duhamel's formula, the equation for $f$ reads
$$
f(t,x) = f(0,x) + \int_0^t e^{t'\partial_x^3}\left((e^{-t'\partial_x^3}f)^3\right)_x dt'.
$$
In Fourier variables, $\tilde{u}:=\widehat{f}$ satisfies
\begin{equation}\label{eq:interaction}
    \tilde{u}(t,\xi) = \tilde{u}(0,\xi) -\frac{3i\xi}{4\pi^2} \int_0^t \iint\limits_{\xi=\xi_1+\xi_2+\xi_3} e^{-it'\Phi}\tilde{u}(t',\xi_1) {\tilde{u}(t',\xi_2) }\tilde{u}(t',\xi_3)d\xi_1 d\xi_2 dt'.
\end{equation}
The representation of a dispersive equation in terms of $\tilde{u}$ has many advantages. All the nonlinear effects are concentrated in a single convolution integral and the interaction between different wave packets in the nonlinear term is perfectly encoded in the resonance function $\Phi$ (for an example of application, see the space-time resonance method developed by Germain, Masmoudi and Shatah \cite{GMS}).
Assuming now that the solution is self-similar, that is,
$$
\tilde{u}(t,\xi)=\tilde{U}(t^{\frac13}\xi),
$$
equation \eqref{eq:interaction} reduces to \eqref{eq:ss_mkdv}. 

To construct self-similar solutions, we must now find fixed points of \eqref{eq:ss_mkdv} in a suitable functional space.
\begin{itemize}

    \item In order to explore the oscillatory effects in the integral term, we must employ stationary-phase arguments, which require some information on the derivatives of $\tilde{U}$. As $s_{c,\infty}=0$, $L^\infty_\xi$ is a scaling-invariant space and we aim to derive uniform bounds on $\tilde{U}$ and $\partial_\xi\tilde{U}$.
    
    \item For low frequencies, since the \eqref{mKdV} flow preserves the zero frequency, we do not expect any sort of singular behavior near $\xi=0$.

    \item For large frequencies, one must understand precisely the behavior of
  \begin{equation}\label{eq:integral}
         \int_0^\infty \iint\limits_{\eta=\eta_1+\eta_2+\eta_3} e^{-i\Phi}\tilde{U}(\eta_1){\tilde{U}(\eta_2)}\tilde{U}(\eta_3)d\eta_1 d\eta_2 d\eta,
  \end{equation}
    which can morally be interpreted as the limit $t\to \infty$ in \eqref{eq:interaction}. In other words, we must comprehend the full nonlinear effects over $t\in \R^+$ - this is the scattering problem! As soon as this connection is made, one must look for what is known concerning scattering for \eqref{mKdV} \cite{HayashiNaumkin_mkdv1, HayashiNaumkin_mkdv2, GPR_mkdv}: the integral fails to converge (absolutely) at $\xi=\infty$ due to a logarithmic divergence $\sim \frac{1}{\xi}$, caused by the stationary points of $\Phi$,
    $$
    (\xi_1,\xi_2,\xi_3)=(\xi,\xi,-\xi), \ (\xi,-\xi,\xi), \ (-\xi,\xi,\xi) \ \mbox{ and } \left(\frac{\xi}{3},\frac{\xi}{3},\frac{\xi}{3}\right). 
    $$
    At the first three points, the resonance function $\Phi$ vanishes\footnote{These are the so-called space-time resonances, see \cite{GMS}.} and there is no oscillation available to make the integral simply convergent. At the last point, the resonance function is exactly $\Phi=-8\xi^3/9$ and therefore induces rapid oscillations near $\xi=\infty$, strong enough to make the integral converge.

    At this point, split $\tilde{U}=S_A+z$, where $S_A$ is meant to collect the contributions of the stationary points and $z$ is a remainder term. Plugging in \eqref{eq:ss_mkdv}, the term $S_A$ on the l.h.s. must compensate exactly the logarithmic divergences caused by the stationary points. To find such an ansatz $S_A$, one may use stationary phase arguments near each of the space-resonances in the interaction representation equation \eqref{eq:interaction} to write
    \begin{equation}\label{eq:stat}
        \tilde{u}_t(t,\xi) \simeq -\frac{i}{4\pi t}\left(|\tilde{u}(t, \xi)|^2\tilde{u}(t, \xi) + e^{-\frac{8i}{9}t\xi^3}\tilde{u}^3\left(t, \frac{\xi}{3}\right)\right),\quad  |t^{1/3}\xi|\gg 1.
    \end{equation}
    Assuming that the solution is self-similar, the first term on the l.h.s. leads to the first-order approximation
    $$
    Ae^{-\frac{3i}{4\pi}|A|^2\log|t^{1/3}\xi|},\quad |t^{1/3}\xi|\gg 1.
    $$
    However, as we also need to control derivatives in order to perform rigorously the stationary-phase arguments, we need to insert the above approximation in \eqref{eq:stat} - this produces the second-order ansatz \eqref{eq:SA_mkdv}.

    This term is the main asymptotic of $\tilde{U}$ at large frequencies and thus $z$ is expected to decay at least like $\jap{\xi}^{-\kappa}$ for some $\kappa>0$.

\item After the ansatz $S_A$ takes care of the divergent terms in the nonlinearity, \eqref{eq:ss_mkdv} becomes 
\begin{align*}
    z(\xi) &= \tilde{U}(0^+) +\int_0^\xi(\text{absolutely convergent})d\eta\\&= \tilde{U}(0^+) +\int_0^\infty(\text{absolutely convergent})d\eta - \int_\xi^\infty(\text{absolutely convergent})d\eta, \quad \xi\gg 1.
\end{align*}
Taking the limit $\xi\to+\infty$, we find
\begin{equation}\label{eq:datascat}
    \tilde{U}(0^+) =-\int_0^\infty(\text{absolutely convergent})d\eta.
\end{equation}
This equation relates the value at the zero frequency with the total nonlinear evolution over all scales. Since, at $t=0$, $\tilde{U}(t^{\frac13}\xi)\equiv \tilde{U}(0^+)$ for $\xi>0$, this equation corresponds to the \emph{initial data-to-scattering state} map for the self-similar solution. As a consequence, \eqref{eq:datascat} relates the small scale $\tilde{U}(0^+)$ with the large scale $A$.

\item Given $A\in \complex$ small, the considerations in the two points above lead us to conclude that \eqref{eq:ss_mkdv} allows for a fixed-point argument for the remainder $z$ in the space
$$
Z^\kappa = \{z\in L^\infty_{loc}(\R): \|\jap{\xi}^\kappa z\|_{L^\infty_\xi} + \|\jap{\xi}^{\kappa+1} z'\|_{L^\infty_\xi} <\infty\}.
$$
To implement the fixed-point argument, precise multilinear estimates over $Z^\kappa$ must be derived. In \cite{CCV20}, it is shown that the necessary estimates hold for $\kappa\in(\frac12, \frac47).$ 
\item The final step consists in inverting the relation between $\tilde{U}(0^+)$ and $A$, so that the self-similar solutions are parametrized by $\tilde{U}(0^+)=c\in \complex$ small.

\end{itemize}

In \cite{CCV20}, the parallel between \eqref{eq:ss_mkdv} and \eqref{eq:interaction} was not made and the generalization of these arguments to other dispersive equations was quite unclear. We will show how the scattering theory helps to make the proof more robust and therefore applicable to \eqref{3NLS}, \eqref{dnls} and \eqref{4Kdv}.

\subsubsection{{The \eqref{4Kdv} case}} Following the interaction representation approach, suppose that $v$ is a solution of \eqref{4Kdv} and define the interaction representation (in frequency variables) as
\begin{equation}\label{expressão do perfil 4KdV}
 \tilde{v}(t,\xi):=\left(e^{-t\partial_x^3}v(t)\right)^\wedge(\xi).   
\end{equation}
Substituting the expression \eqref{expressão do perfil 4KdV} into \eqref{4Kdv}, we obtain the \textit{interaction representation equation}
\begin{equation}\label{eq.profile}
\tilde{v}_t(t,\xi)= - \frac{i \xi}{8\pi^3} \iiint e^{-it\Phi} \tilde{v}(t,\xi_1)\tilde{v}(t,\xi_2)\tilde{v}(t,\xi_3)\tilde{v}(t,\xi_4)\, d\xi_1d\xi_2d\xi_3,
\end{equation}
where $\xi_4= \xi - \xi_1-\xi_2-\xi_3$ and
$$
\quad \Phi:= \xi^3 - \xi_1^3-\xi_2^3-\xi_3^3-\xi_4^3.
$$
Assuming that the solution $v$ is self-similar,
$$
v(t,x)=\frac{1}{t^{\frac29}}V\left(\frac{x}{t^{\frac13}}\right),
$$
the interaction representation must satisfy
\begin{equation}\label{eq:perfil_4kdv}
    \tilde{v}(t,\xi)=(t^{\frac13}\xi)^{\frac{1}{3}}\tilde{V}(t^{\frac13}\xi),\qquad \tilde{V}=\xi^{\frac13}e^{i\xi^3}\widehat{V}.
\end{equation}
Since we are looking for real-valued self-similar solutions, $\tilde{V}(-\xi)=\overline{\tilde{V}(\xi)}$, which means that we may restrict our attention to $\xi>0$.

Substituting \eqref{eq:perfil_4kdv} into equation \eqref{eq.profile}, we find
\begin{equation}
    \tilde{V}(\xi) = \tilde{V}(0^+) - \frac{3i}{8\pi^3}\int_0^\xi \eta^{1/3}\left(\iiint\limits_{\eta=\eta_1+\dots+\eta_4} e^{-i\Phi}\frac{\tilde{V}(\eta_1)\tilde{V}(\eta_2)\tilde{V}(\eta_3)\tilde{V}(\eta_4)}{(\eta_1\eta_2\eta_3\eta_4)^{\frac13}} d\eta_1d\eta_2d\eta_3\right) d\eta.
\end{equation}
Defining the multilinear operator
$$
\mathcal{M}[f_1,f_2,f_3,f_4](\eta)=\iiint\limits_{\eta=\eta_1+\dots+\eta_4} e^{-i\Phi} \frac{f_1(\eta_1)f_2(\eta_2)f_3(\eta_3)f_4(\eta_4)}{(\eta_1\eta_2\eta_3\eta_4)^{\frac{1}{3}}}\,\, d\eta_1d\eta_2d\eta_3,
$$
and $\mathcal{M}[f]:=\mathcal{M}[f,f,f,f]$, the equation for the profile can be written as
\begin{equation}\label{eq:ss_4kdv}
    \tilde{V}(\xi) = c - \frac{3i}{8\pi^3}\int_0^\xi \eta^{\frac13}\mathcal{M}[\tilde{V}](\eta)d\eta,\qquad c:= \tilde{V}(0^+).
\end{equation}
We stress that \eqref{eq:ss_4kdv} can also be deduced from the corresponding ODE \eqref{edo associada 4kdV}. In particular, $\tilde{V}$ is a solution of \eqref{eq:ss_4kdv} iff $V$ is a solution to \eqref{edo associada 4kdV}.

Let us now follow the ideas applied in the \eqref{mKdV} case:
\begin{itemize}
\item Since $s_{c,\infty}=1/3$ for \eqref{4Kdv} and $\tilde{V}=\xi^{\frac13}e^{i\xi^3}\hat{V}$, we aim to derive uniform bounds for $\tilde{V}$ (and $\tilde{V}^\prime$) in $L^\infty_\xi$.
    \item For small frequencies, since \eqref{4Kdv} preserves the zero frequency, we do not expect any singular behavior.
    \item For large frequencies, one must understand the scattering problem for \eqref{4Kdv}. In \cite{Tao_4kdv}, Tao showed the (linear) scattering for small data in $\dot{H}^{-\frac16}$ and thus we expect the nonlinear effects to remain bounded as one integrates up to $\xi=\infty$. As such, the high-frequency ansatz is simply
    $$
    S_A(\xi)=A,\quad \xi\gg 1.
    $$
    \item Writing $\tilde{V}=S_A + z$,  $z$ decaying at $\xi=+\infty$, if one takes the limit $\xi\to +\infty$,\eqref{eq:ss_4kdv} gives
    \begin{equation}\label{eq:datascat_4kdv}
        A=c-\frac{3i}{8\pi^3} \int_0^\infty \eta^{\frac{1}{3}}\mathcal{M}[S_A+z](\eta)d\eta.
    \end{equation}
    This is the initial data-to-scattering state map. Using \eqref{eq:datascat_4kdv}, we can define $c$ as a function of $A$ and $z$.
    For the remainder $z$, this choice of $c$ implies that
    \begin{equation}
        z(\xi) = \frac{3i}{8\pi^3} \int_0^\xi \eta^{\frac13}\mathcal{M}[\tilde{V}](\eta)d\eta
    \end{equation}
    which can be shown to decay polynomially.
    \item Given $A\in \complex$ small, these considerations allow us to perform a fixed-point argument for $z$ in the space
    $$
Z^\kappa = \{z\in L^\infty_{loc}(\R): \|\jap{\xi}^\kappa z\|_{L^\infty_\xi} + \|\jap{\xi}^{\kappa+1} z'\|_{L^\infty_\xi} <\infty\}.
$$
as long as $\kappa\in (\frac58, \frac 23)$. The fixed-point will follow from multilinear estimates over $Z^\kappa$ as in the \eqref{mKdV} case. However, the argument is a bit more involved for two reasons: first, the extra integration requires an additional step in the estimates (as these are done one integral at a time); second, the presence of singular weights in the expression for $\mathcal{M}$ requires a delicate control of the integral at low frequencies.
\item Finally, we invert the relation between $A$ and $c$ by proving that \eqref{eq:datascat_4kdv} is a perturbation of the identity map near 0.
\end{itemize}

\subsubsection{{The \eqref{dnls} case}.}
 As before, we start by rewriting \eqref{dnls} in the interaction representation variable. Setting
\begin{equation}\label{expressão do perfil}
 \tilde{w}(t,\xi):=\left(e^{tH\partial_x^2}w(t)\right)^\wedge(\xi),   
\end{equation}
equation \eqref{dnls} becomes
\begin{equation}\label{eq.profile_dnls}
\tilde{w}_t(t,\xi)= \frac{i \xi}{4\pi^2} \iint \limits_{\xi=\xi_1+\xi_2+\xi_3} e^{it\Phi(\xi)} \tilde{w}(t,\xi_1){\tilde{w}(t,\xi_2)}\tilde{w}(t,\xi_3)\, d\xi_1d\xi_2,\qquad \Phi= \xi|\xi|  -\xi_1|\xi_1| - \xi_2|\xi_2| - \xi_3|\xi_3|.
\end{equation}
Assuming now that the solution to \eqref{dnls} is self-similar,
$$
w(t,x)=\frac{1}{t^{\frac14}}W\left(\frac{x}{t^\frac12}\right),
$$
this implies that
$$
\tilde{w}(t,\xi)=|t^{\frac12}\xi|^{\frac12}\tilde{W}(t^{\frac12}\xi),\quad \tilde{W}:=|\xi|^{\frac12}e^{-i\xi|\xi|}\hat{W}.
$$
Plugging into \eqref{eq.profile_dnls},
\begin{equation}
    \tilde{W}(\xi)=\tilde{W}(0) +\frac{2i}{4\pi^2} \int_0^\xi \eta^{\frac12}\left(\iint_{\eta_1+\eta_2+\eta_2=\eta}e^{i\Phi}\frac{\tilde{W}(\eta_1){\tilde{W}(\eta_2)}\tilde{W}(\eta_3)}{|\eta_1\eta_2\eta_3|^{\frac12}}d\eta_1d\eta_2\right)d\eta,\quad \xi\in \R.
\end{equation}

Introducing the multilinear operator
\begin{equation}
    \mathcal{I}[g_1,g_2,g_3](\eta):= \iint_{\eta_1+\eta_2+\eta_3=\eta} e^{i\Phi} \frac{g_1(\eta_1){g_2(\xi_2)}g_3(\eta_3)}{(\eta_1\eta_2\eta_3)^{\frac{1}{2}}}\,\, d\eta_1d\eta_2 
\end{equation}
and setting $\mathcal{I}[g]:=\mathcal{I}[g,g,g]$, we arrive at the profile equation
\begin{equation}\label{eq:ss_dnls}
    \tilde{W}(\xi)=c + \frac{i}{2\pi^2} \int_0^\xi \eta^{\frac12} \mathcal{I}[\tilde{W}](\eta)d\eta,\quad \xi\in \R, \qquad c:= \tilde{W}(0).
\end{equation}
We once again look for real-valued self-similar solutions, which imposes $\tilde{W}(-\xi)=\overline{\tilde{W}(\xi)}$.

Following the ideas discussed for \eqref{mKdV} and \eqref{4Kdv}, let us present the main ingredients to tackle \eqref{eq:ss_dnls}.
 {
\begin{itemize}
    \item As $s_{c,\infty}=1/2$ and $\tilde{W}=|\xi|^{\frac{1}{2}}e^{i\xi|\xi|}\widehat{W}$, we once again look for $\tilde{W}, \tilde{W}'\in L^\infty_\xi$.
    \item Since \eqref{dnls} preserves the average, the profile should remain well-behaved for small frequencies.
    \item Concerning the scattering theory for \eqref{dnls}, it has been shown in \cite{MR1491867, MR1877838} that \eqref{dnls} presents modified scattering behavior. This is caused by the resonances
    $$
    (\xi_1,\xi_2,\xi_3)\in\left\{(\xi,\xi,-\xi).(\xi,-\xi,\xi),(-\xi,\xi,\xi),\left(\frac{\xi}{3}, \frac{\xi}{3}, \frac{\xi}{3}\right)\right\}
    $$
    and by the high$\times$low$\times$low$\to$high interactions\footnote{These interactions are usually unobservable in the usual scattering theory, but appear here as a consequence of the singular behavior $|\xi|^{-\frac12}$ near the zero frequency.}
    Away from this point, the nonlinear effects can be shown to be finite when integrated up to $t=+\infty$.
    As in the \eqref{mKdV} case, we handle this problematic resonance via the ansatz \eqref{eq:ansatz_mbo}. Notice that, unlike the \eqref{mKdV} case, due to the  high$\times$low$\times$low$\to$high interactions, the phase parameter $a$ also depends on $c$.
    \item Taking $\tilde{W}=S_A+z$, for some decaying $z$, \eqref{eq:ss_dnls} implies that
\begin{align*}
        z(\xi)&=c+\int_0^\xi \text{(absolutely convergent)}d\eta \\&=c+\int_0^\infty \text{(absolutely convergent)}d\eta - \int_\xi^\infty \text{(absolutely convergent)}d\eta,\quad \xi\gg 1.
\end{align*}
as thus, taking the limit $\xi \to +\infty$, we find the \textit{data-to-scattering map}
\begin{equation}
    c=-\int_0^\infty \text{(absolutely convergent)}d\eta,
\end{equation}
which can be used to determine $A$ in terms of  $c$ in terms of $A$. Notice that, contrary to the \eqref{4Kdv} case, the integral on the r.h.s. is highly nonlinear in $A$ (see the ansatz \ref{eq:ansatz_mbo}).
\item For $c\in \R$ fixed, we perform a fixed-point argument for $z$ over
 $$
Z^\kappa = \{z\in L^\infty_{loc}(\R): \|\jap{\xi}^\kappa z\|_{L^\infty_\xi} + \|\jap{\xi}^{\kappa+1} z'\|_{L^\infty_\xi} <\infty\}.
$$
As we will see, the required multilinear estimates hold for $\kappa\in (0,\frac14)$. As in the \eqref{4Kdv} case, the presence of singularities at low frequencies in the convolution operator $\mathcal{I}$ is a nontrivial difficulty for the derivation of multilinear bounds.
\end{itemize}
}

\subsubsection{{The \eqref{3NLS} case}} Finally, we discuss the methodology to construct self-similar solutions to \eqref{3NLS}. For the sake of simplicity, since the sign in the nonlinear term has no impact in our construction, we focus on the focusing case. Consider the interaction representation
$$
\tilde{q}(t,\xi)=\left(e^{it\partial_x^2}q(t)\right)^\wedge(\xi),
$$
which satisfies
\begin{equation}\label{eq:profile}
\tilde{q}_t(t,\xi)= -\frac{i}{4\pi^2}\iint_{\xi=\xi_1+\xi_2+\xi_3} e^{-it\Phi} \tilde{q}(t,\xi_1)\overline{\tilde{q}(t,\xi_2)}\tilde{q}(t,\xi_3)\, d\xi_1d\xi_2,\qquad \Phi =\xi^2-\xi_1^2 + \xi_2^2 -\xi_3^2.
\end{equation}
Assuming that $u$ is a self-similar solution, 
$$
\tilde{q}(t,\xi)=\tilde{Q}(t^{\frac12}\xi),\quad \tilde{Q}:=e^{i\xi^2}\widehat{Q}.
$$
Replacing in \eqref{eq:profile}, we find, for $\xi>0$,
\begin{equation}\label{eq1}
\tilde{Q}(\xi)= \tilde{Q}(1) -\frac{i}{2\pi^2}\int_1^\xi \frac{1}{\eta}\iint_{\eta=\eta_1+\eta_2+\eta_3} e^{-i\Phi} {\tilde{Q}(\eta_1)\overline{\tilde{Q}(\eta_2)}\tilde{Q}(\eta_3)}\,\, d\eta_1d\eta_2 d\eta
\end{equation}
and, for $\xi<0$,
\begin{equation}\label{eq2}
\tilde{Q}(\xi)= \tilde{Q}(-1) -\frac{i}{2\pi^2}\int_{-1}^\xi \frac{1}{\eta}\iint e^{-i\Phi} {\tilde{Q}(\eta_1)\overline{\tilde{Q}(\eta_2)}\tilde{Q}(\eta_3)}\,\, d\eta_1d\eta_2 d\eta
\end{equation}
If we define
$$
\mathcal{T}[h_1,h_2,h_3](\eta):= \iint_{\eta=\eta_1+\eta_2+\eta_3}e^{-i\Phi} {h_1(\eta_1)\overline{h_2(\eta_2)}h_3(\eta_3)}\,\, d\eta_1d\eta_2
$$
and $\mathcal{T}[\tilde{Q}]:=\mathcal{T}[\tilde{Q}, \tilde{Q}, \tilde{Q}]$, \eqref{eq1}-\eqref{eq2} may be rewritten as
\begin{equation}\label{eq:ss_nls}
    \tilde{Q}(\xi)=\tilde{Q}(\pm 1) - \frac{i}{2\pi^2}\int_{\pm 1}^\xi \frac{1}{\eta}\mathcal{T}[\tilde{Q}](\eta)d\eta, \qquad \xi\in \R^\pm.
\end{equation}
Let us proceed with an overview of the proof of Theorem \ref{teorema Self Similar nls}.
\begin{itemize}
    \item Since $s_{c,\infty}=0$, we look for $\tilde{Q}, \tilde{Q}'\in L^\infty_\xi$.
    \item For low frequencies, due to the lack of derivative loss in the nonlinearity, the integral in $\eta$ in \eqref{eq:ss_nls} has a logarithmic divergence at the zero frequency\footnote{This is the reason why the integration in \eqref{eq:ss_nls} starts from $\xi=\pm 1$.}. As such, we expect that $\tilde{Q}(\xi)\sim \log |\xi|$ for $|\xi|\ll 1$.
    \item For large frequencies, as in the \eqref{dnls} case, \eqref{NLS} presents modified scattering behavior at $t=+\infty$ \cite{HayashiNaumkin_nls, KatoPusateri}, caused  by the space-time resonance
    $$
    (\xi_1,\xi_2,\xi_2)=(\xi,-\xi,\xi).
    $$
    To take care of this problematic divergence, we introduce the ansatz $S_A$ as in \eqref{eq:ansatz_nls}.
    \item After the removal of the divergent interaction and arguing as for \eqref{mKdV} and \eqref{dnls},
    $$
    \tilde{Q}(\pm 1) = -\int_{\pm 1}^{\pm \infty}(\text{absolutely convergent})d\eta.
    $$
    This allows us to determine the value at $\xi=\pm 1$ from the value at $\xi=\pm \infty$.
    \item The fixed point for the remainder $z$ must now be performed in a space which takes into account the possible singular behavior near $\xi=0$. To that end, define
    $$
\|z\|_{Y^\kappa} := \| (\log |\xi|)^{-1} z\|_{L^\infty({|\xi|\lesssim 1})} 
+ \| |\xi| z^\prime\|_{L^\infty({|\xi|\lesssim 1})} 
+ \| \langle \xi \rangle^\kappa z\|_{L^\infty({|\xi|\gg 1})} 
+ \| \langle \xi \rangle^{\kappa +1} z^\prime\|_{L^\infty({|\xi|\gg 1})} .
$$
and $Y^\kappa=\{z\in \mathcal{S}'(\R): \|z\|_{Y^\kappa}<\infty\}$.
For $A\in \complex$ fixed, we will derive the necessary multilinear estimates for $\mathcal{T}$ over $Y^\kappa$ in order to perform a fixed-point argument in $z$.

\end{itemize}

\subsection{Future perspectives} To conclude, let us now discuss some new research directions which arise from the analysis performed in this work.

\subsubsection{Small self-similar solutions for 1d dispersive equations.} Consider a generic 1d-nonlinear dispersive equation
\begin{equation}\label{eq dispersiva}
 u_t + i P(D_x)u =N(u) 
\end{equation}
where $P(D_x)$, is a linear (pseudo-)differential operator in the spatial variables given by a real Fourier symbol $P(\xi)$
$$
[{P(D_x) f}](x):= \left(P(\xi) \widehat{f}(\xi)\right)^\vee(x)
$$
and $N(u)$ is a general nonlinearity defined by the multilinear form $N(u)= N(u,...,u)$ where
$$
\mathcal{F}_x\left[N(f_1,...,f_k)\right] (\xi) = \idotsint\limits_{\xi = \xi_1 + \cdots + \xi_k} m(\xi_1, ...,\xi_k) \widehat{f_1}(\xi_1) \cdots \widehat{f_k}(\xi_k) \, d\xi_1 \cdots d\xi_{k-1},
$$
where $m=m(\xi_1,...,\xi_k) \in \mathbb{C}$ is a Fourier multiplier that represents the possible loss of derivatives occurring  in the nonlinear term. 

From this point on, we assume that our dispersive equation \eqref{eq dispersiva} has dispersion of order $n$ and a nonlinearity of order $k \in \mathbb{Z}_+ \setminus \{0,1\}$, with a loss of $m \in \mathbb{Z}_+$ derivatives, thus encompassing all the concrete models treated above. Therefore, \eqref{eq dispersiva} admits an invariance under scaling, that is, if $u=u(t,x) \in \mathbb{K}$ (with $\mathbb{K} = \mathbb{R}$ or $\mathbb{K} = \mathbb{C}$) is a solution of \eqref{eq dispersiva}, then
\begin{equation}\label{scaling generico}
u_\lambda(t,x) := \lambda^{\frac{n - m}{k - 1}} u(\lambda^n t, \lambda x),    
\end{equation}
is also a solution for every $\lambda > 0$.
\begin{example}[Generalized Korteweg–de Vries equation]
For $P(\xi)=\xi ^3$ and 
\begin{equation}\label{nonlinearidade da gkdv}
N(u)= \mp \frac{\partial_x(u^{k+1})}{k+1}\,, \qquad k \in \mathbb{Z}_+,  
\end{equation}
we obtain the generalized Korteweg–de Vries (gKdV) equation
$$
u_t+ u_{xxx} \pm u^k u_x=0,
$$
where $u: \R \times \R \to \R$.  In particular, for $k=1$ we recover the classical KdV equation, for $k=2$ we obtain \eqref{mKdV} and $k=4$ we recover the equation \eqref{4Kdv}.
\end{example}

\begin{example}[Nonlinear Schrödinger equation]
Choosing $P(\xi)= \xi^2$ and $N(u)=\pm i|u|^\alpha u$ with $\alpha>0$, we obtain the nonlinear Schrödinger (NLS) equation
\begin{equation}\label{NLS}
    iu_t + u_{xx} \pm |u|^\alpha u=0, 
\end{equation}
with $u:\R \times \R \to \mathbb{C}$.
\end{example}

\begin{example}[Generalized Benjamin–Ono equation]
For $P(\xi)=-|\xi|\xi$ and $N(u)$ as in \eqref{nonlinearidade da gkdv}, we obtain the generalized Benjamin–Ono equation
\begin{equation}\label{BO}
    u_t + Hu_{xx} \pm u^ku_x=0.  
\end{equation}
 In particular, for $k=1$ we recover the classical Benjamin-Ono equation and for $k=2$ we obtain \eqref{dnls}.

\end{example}

Since the equation \eqref{eq dispersiva} is invariant under scaling \eqref{scaling generico}, we can formally define the notion of a self-similar solution.

\begin{definition}[Self-similar solution for \eqref{eq dispersiva}]
A solution $ u=u(t,x)$ of \eqref{eq dispersiva} is said to be \textit{self-similar solution} if it satisfies
$$
u_\lambda \equiv u \quad \text{for all $\lambda >0$}.
$$
More precisely, they assume the form
 \begin{equation}
u(t,x):= t^{\frac{m-n}{(k-1)n}}V(t^{-\frac{1}{n}}x) 
 \end{equation}
for $ t > 0 $ and the self-similar profile $ V: \mathbb{R} \to \mathbb{K} $ satisfies (in self-similar variables $\zeta= t^{-\frac{1}{n}}x$) the equation
\begin{equation}\label{edo associada}
\frac{1}{k - 1} V(\zeta) - \frac{1}{n} \, \zeta V'(\zeta) + i\, P(D_\zeta) V'(\zeta) = N\big(V(\zeta)\big).
\end{equation}
\end{definition}
Within this framework, we propose a roadmap to construct small self-similar solutions of general 1d-dispersive equations:

\bigskip

\noindent \textbf{\underline{Roadmap for construction of self-similar solutions.}}
\begin{itemize}
    \item The self-similar profile is sought in the scaling-critical $\mathcal{F}L^{s,\infty}$ space (together with its derivative).
    \item For small frequencies, either there is no singularity (usually if the equation preserves the average) or there may exists a logarithmic (or possibly stronger) singularity. In the first case, one should work over $Z^\kappa$, while the second case requires the incorporation of the singular behavior as in $Y^\kappa$.
    \item For large frequencies, if the nonlinearity presents resonant interactions which are not integrable in time (leading to modified scattering behavior), one must introduce an ansatz that absorbs the divergent terms. If the nonlinearity is absolutely integrable over all times, one should take the constant ansatz.
    \item Once the ansatz is derived, one writes the data-to-scattering map, relating low and large frequencies.
    \item For a fixed asymptotic behavior at large frequencies, perform a fixed-point argument for the remainder.
\end{itemize}

\subsubsection{Finite-energy self-similar solutions} The methodology presented in this paper works for small amplitude self-similar solutions, since in this case it is easy to prove that the profile equation is a perturbation of the identity map.
One may attempt to further extend the self-similar curve - this requires some \textit{spectral information} on the profile equation. This is far from trivial, as our analysis is performed in the Fourier-Lebesgue class $\mathcal{F}L^{s,\infty}$ instead of the hilbertian class $H^s$.

Assuming that the self-similar curve can be extended for large values of the parameter $c$, it may occur that the corresponding high-frequency parameter $A$, related to $c$ via the data-to-scattering map, \emph{vanishes}. When this occurs, the solution in physical space will suddenly present a much weaker oscillatory behavior. As a consequence, this special self-similar solution can become of \emph{finite energy}. Finite-energy self-similar solutions are among the most desired and elusive objects in dispersive equations, as they are physically relevant (stable) blow-up solutions. In this direction, we refer the works \cite{DonningerSchorkhuber, Troy, RottschaferKaper, Koch_gkdv, BahriMartelRaphael} for different approaches on the subject.



\bigskip \noindent \textbf{Organization of the paper.} In Sections \ref{sec:4kdv_fixo}-\ref{sec:nls_fixo}, assuming the validity of the multilinear estimates on the operators $\mathcal{M}, \mathcal{I}$ and $\mathcal{T}$, we construct the self-similar solutions to \eqref{4Kdv}, \eqref{dnls} and \eqref{3NLS}, respectively. The last four sections (Sections \ref{sec:4kdv}-\ref{sec:nls}) are then dedicated to the precise technical derivation of the necessary multilinear bounds.

\bigskip \noindent \textbf{Notation.} Throughout this work, given $a\in \mathbb{R}$, we use $a^{+}$ (resp. $a^{-}$) to denote a number slightly larger (resp. smaller) than $a$.  
Given $a,b\in \mathbb{R}$, we write $a \lesssim b$ when there exists a universal constant $C>0$ such that $a \leq Cb$.  If $C$ can be chosen sufficiently small, we write $a \ll b$.  If $a \lesssim b$ and $b \lesssim a$, we write $a \sim b$.  Moreover, if $|b-a| \ll |a|$, we say that $a \simeq b$.

\section{{Self-similar solutions for 4KdV}}\label{sec:4kdv_fixo}

The key ingredient is a quadrilinear multilinear estimate, whose proof is postponed to Section \ref{sec:4kdv}.

\begin{thm}[Multilinear Estimates]\label{Multilinear estimates for M}
Let $\kappa \in  (\frac{5}{8},\frac{2}{3})$  and take, for $j=1,2,3,4$, $\kappa_j\in \{0,\kappa\}$ and $f_j\in Z^{\kappa_j}$. If $f_j\notin Z^{\kappa}$, suppose also that $f_j(\xi)\equiv \chi(\xi)$ or $f_j(\xi)\equiv \chi (-\xi)$. Then
\begin{equation}
    |\mathcal{M}[f_1,f_2,f_3,f_4](\eta)|\lesssim \jap{\eta} ^{-4/3-k}\prod_{j=1}^4\|f_j\|_{Z^{\kappa_j}}.
\end{equation}
\end{thm}

To proceed with the proof of Theorem \ref{teorema Self Similar 4KdV}, following \ref{eq:datascat_4kdv}, it is necessary to define the map $ c: \mathbb{C} \times Z^\k \to \mathbb{C} $ as 
$$
c(A,z):= A+ \frac{3i}{8\pi^3} \int_0^\infty \,\mathcal{M}[S_A+z](\eta)\,\eta^{1/3}\, d\eta.
$$
Observe that, if $ z\in Z^\k $, then by Theorem \ref{Multilinear estimates for M},
\begin{align*}
|c(A, z)| &\lesssim |A| + \int_0^\infty  + |\mathcal{M}[S_A + z](\eta) ||\eta|^{1/3} \, d\eta \\
&\lesssim |A| + (|A|+ \|z\|_{Z^\kappa}^4)\int_0^\infty \frac{1}{\jap{\eta}^{1+\kappa}}d\eta \lesssim |A| + \left(|A| +\|z\|_{Z^\k} \right)^4.
\end{align*} 
Next, for $\xi>0$, consider the operator  
\begin{equation}
\Gamma_A [z](\xi) := c(A, z) - S_A(\xi) - \frac{3i}{8\pi^3}  \int_0^\xi \mathcal{M}[S_A + z](\eta) \, \eta^{1/3} \, d\eta,
\end{equation}  
 with $A \in \mathbb{C}$ and $z \in Z^\k$. We extend $\Gamma_A$ for $\xi < 0$ via the relation $\Gamma_A [z](\xi) = \overline{\Gamma_A [z](-\xi)}$.

\begin{theorem}\label{ponto fixo Gamma_A}
Fix $ \kappa \in \left(\frac58, \frac23 \right) $ and $ A \in \mathbb{C}$ with $ |A| \ll 1 $. Then the operator $ \Gamma_A $ admits a unique fixed point $ z_A \in Z^\kappa(\mathbb{R}) $. Furthermore, $z_A$ satisfies
\begin{equation}\label{estimativa Z<A}
 \|z_A\|_{Z^\k} \lesssim |A|.   
\end{equation}
\end{theorem}

\dem  First, let $ z \in Z^\kappa(\mathbb{R}) $. Consider firstly $ 0<\xi \lesssim 1 $. Then, we have  
\begin{align}
\Gamma_A [z](\xi) &= c(A, z) - S_A(\xi) - \frac{3i}{8\pi^3}  \int_0^\xi \mathcal{M}[S_A + z](\eta) \eta^{1/3} \, d\eta \\
& = c(A, z) - \frac{3i}{8\pi^3}  \int_0^\xi \mathcal{M}[S_A + z](\eta) \eta^{1/3} \, d\eta,
\end{align}
where the second equality follows because $ S_A(\xi) $ is a cut-off function that cancels its contribution for small $ \xi$.  Now, by Theorem \ref{Multilinear estimates for M}, it follows that 
\begin{align}
|\Gamma_A [z](\xi)| &\leq |c(A, z)| + \int_0^\xi |\mathcal{M}[S_A + z](\eta)| |\eta|^{1/3} \, d\eta \\
& \lesssim \! |A|+  \left(|A| +\|z\|_{Z^\k} \right)^4  \left[1+\int_0^\xi |\eta|^{1/3} \, d\eta \right]\lesssim |A| + \left(|A| +\|z\|_{Z^\k} \right)^4 ,
\end{align}
since $ |\xi| \lesssim 1 $. Therefore, for $ \kappa \in  \left(\frac58, \frac23 \right) $, we conclude that
\begin{equation}
\langle \xi \rangle^{\kappa} |\Gamma_A [z](\xi)| \lesssim \left(|A| +\|z\|_{Z^\k} \right)^4,\qquad |\xi|\lesssim 1.
\end{equation}

Similarly, for the derivative,
\begin{equation}\label{expressão derivada de Gamma}
 (\Gamma_A [z]^\prime (\xi) = \frac{3i}{8\pi^3} \mathcal{M}[S_A + z](\xi) \xi^{1/3},   
\end{equation}
it follows from Theorem \ref{Multilinear estimates for M} that
\begin{align}
 \langle \xi \rangle^{\kappa+1} |(\Gamma_A [z])^\prime (\xi)| \lesssim \left(|A| +\|z\|_{Z^\k} \right)^4, \qquad |\xi|\lesssim 1.  
\end{align}
Now consider the case $ |\xi| \gg 1 $. In this scenario, observe that
\begin{align*}
\Gamma_A [z](\xi) &= c(A,z)- S_A(\xi) - \frac{3i}{8\pi^3}  \int_0^\xi \,\mathcal{M}[S_A+z](\eta)\,\eta^{1/3}\, d\eta \\
&= c(A,z)- A - \frac{3i}{8\pi^3}  \int_0^\infty \,\mathcal{M}[S_A+z](\eta)\,\eta^{1/3}\, d\eta + \frac{3i}{8\pi^3} \int_\xi^\infty \,\mathcal{M}[S_A+z](\eta)\,\eta^{1/3}\, d\eta\\
&= \frac{3i}{8\pi^3} \int_\xi^\infty \,\mathcal{M}[S_A+z](\eta)\,\eta^{1/3}\, d\eta,
\end{align*}
by the definition of $c(A,z)$.
Again, by Theorem \ref{Multilinear estimates for M},
\begin{align}
|\Gamma_A[z](\xi)|\lesssim \int_\xi^\infty \,|\mathcal{M}[S_A+z](\eta)||\eta|^{1/3}\, d\eta\lesssim \left(|A| +\|z\|_{Z^\k} \right)^4  \int_\xi^\infty \,|\eta|^{-k-1}\, d\eta\lesssim \left(|A| +\|z\|_{Z^\k} \right)^4 \jap{\xi}^{-\k}
\end{align}
since $ |\xi| \gg 1 $. Therefore  
$$  
\langle \xi \rangle^{\k} |\Gamma_A [z](\xi)| \lesssim \left(|A| +\|z\|_{Z^\k} \right)^4 <\infty.  
$$  
Similarly, using the expression \eqref{expressão derivada de Gamma}, we obtain  
$$  
\langle \xi \rangle^{\k+1} |(\Gamma_A [z])^\prime(\xi)| < \infty. 
$$  
Thus, $ \Gamma_A [z] \in Z^\k $. Since $ \mathcal{M}[S_A + z] $ is a quartic term in both $ A $ and $ z $, using arguments analogous to the previous estimates and Theorem \ref{Multilinear estimates for M}, it can be shown that for $ |A| \ll 1 $ sufficiently small, $ \Gamma_A $ is a contraction in a suitable closed ball of $Z^\k$. By Banach's fixed-point theorem, there exists a unique $ z_A \in Z^\k $ with $ \|z_A\|_{Z^\k} \ll 1 $ such that 
$$
\Gamma_A [z_A] \equiv z_A,
$$  
with
$$
\|z_A\|_{Z^\k} \lesssim |A|.
$$
\qed

To complete the proof of Theorem \ref{teorema Self Similar 4KdV}, it remains to invert the relation  
\begin{equation}\label{equação para c(A)}
 c(A) := c(A, z_A)=A+ \frac{3i}{8\pi^3}  \int_0^\infty \,\mathcal{M}[S_A+z_A](\eta)\,\eta^{1/3}\, d\eta. 
\end{equation}
where $ z_A \in Z^\kappa$ is given by Theorem \ref{ponto fixo Gamma_A}. To achieve this, we prove that \eqref{equação para c(A)} is a perturbation of the identity map near $0\in \complex$. 

For $ |A_1|, |A_2| \leq \varepsilon $ sufficiently small, Theorem \ref{Multilinear estimates for M} implies that  
$$
\left| \int_0^\infty \left[\mathcal{M}[S_{A_1} + z_{A_1}](\eta) - \mathcal{M}[S_{A_2} + z_{A_2}](\eta) \right]\, \eta^{1/3} \, d\eta \right| \lesssim \varepsilon^3 \left( |A_1 - A_2| + \|z_{A_1} - z_{A_2}\|_{Z^\k} \right).
$$
Arguing as in the proof of Theorem \ref{ponto fixo Gamma_A},
$$
\|z_{A_1} - z_{A_2}\|_{Z^\k} \lesssim \epsilon^3|A_1-A_2|,
$$
which implies that
\begin{equation}\label{eq:lips_A}
    \left| \int_0^\infty \mathcal{M}[S_{A_1} + z_{A_1}](\eta) \, \eta^{1/3} \, d\eta - \int_0^\infty \mathcal{M}[S_{A_2} + z_{A_2}](\eta) \, \eta^{1/3} \, d\eta \right| \lesssim \varepsilon^3 |A_1 - A_2|
\end{equation}

Define the function  
$$
F(c, A) := c - A - \int_0^\infty \mathcal{M}[S_A + z_A](\eta) \, \eta^{1/3} \, d\eta,
$$  
where $ c \in \mathbb{C} $ and $ A \in \mathbb{C} $ with $ |A| \ll 1 $. Note that $F$ is a Lipschitz function in variable $A$ and 
$$
F(c(A), A) = 0.
$$ 
Moreover, by \eqref{eq:lips_A},
$$
|F(c, A_1) - F(c, A_2)| \gtrsim (1 - \varepsilon^3) |A_1 - A_2|,
$$
for all $c\in \mathbb{C}$. To proceed, we need the following Implicit Function Theorem for Lipschitz functions proposed by Michael Wuertz in \cite{Wuertz2008}.

\begin{thm}[See \cite{Wuertz2008}, Theorem 5.1]\label{teo funçao implícia}
Let $U_m \subseteq \mathbb{R}^m$ and $U_n \subseteq \mathbb{R}^n$ be open. Fix $a \in U_m$ and $b \in U_n$. Consider 
\begin{equation} \label{eq:5.1}
    F : U_m \times U_n \to \mathbb{R}^n
\end{equation}
a Lipschitz function such that 
\begin{equation} \label{eq:5.2}
    F(a, b) = 0
\end{equation}
and with the property that there exists a constant $K > 0$ for which 
\begin{equation} \label{eq:5.3}
    |F(x, y_2) - F(x, y_3)| \geq K |y_2 - y_3| \quad \text{for all } (x, y_j) \in U, \, j = 1, 2.
\end{equation}
Then there exists an open set $V_m \subseteq \mathbb{R}^m$, with $a \in V_m$, and a Lipschitz function $\varphi : V_m \to U_n$ such that $\varphi(a) = b$ and 
\begin{equation} \label{eq:5.4}
    \{(x, y) \in V_m \times U_n : F(x, y) = 0\} = \{(x, \varphi(x)) : x \in V_m\}.
\end{equation}
In particular, 
$$
F(x, \varphi(x)) = 0, \quad \text{for all } x \in V_m.
$$
\end{thm}

Thus, applying the Theorem \ref{teo funçao implícia} to the function $F$, for $ |c| \ll 1 $ small enough (identifying $\mathbb{C}$ with $\mathbb{R}^2$ in the standard way) we can invert the relation \eqref{equação para c(A)}, thereby completing the proof of Theorem \ref{teorema Self Similar 4KdV}.
\qed

\section{{Self-similar solutions for mBO}}\label{sec:dnls_fixo}

First, recall that
\begin{align}
\mathcal{I}[g_1,g_2,g_3](\eta)&=\iint e^{i\Phi} \frac{g_1(\eta_1){g_2(\eta_2)}g_3(\eta_3)}{|\eta_1\eta_2\eta_3|^{\frac{1}{2}}}\,\, d\eta_1d\eta_2\\
&= \iint\limits_{\eta = \eta_1+\eta_2+\eta_3}\frac{1}{|\eta_1\eta_2\eta_3|^\frac{1}{2}}e^{i (|\eta|\eta-|\eta_1|\eta_1-|\eta_2|\eta_2-|\eta_3|\eta_3)} g_1(\eta_1){g_2(\eta_2)}g_3(\eta_3) d\eta_1d\eta_2.
\end{align}

Take $\phi\in C_c^\infty(\R)$ a cut-off function over the interval $[-1,1]$ and split $\mathcal{I}$ according to the size of $\eta_1,\eta_2$ and $\eta_3$:
\begin{align}
    \mathcal{I}[g_1,g_2,g_3](\eta) &= \mathcal{I}_{\text{h}\times \text{h}}[g_1,g_2,g_3](\eta) + \mathcal{I}_{\text{l}\times \text{l}\times \text{h}}[g_1,g_2,g_3](\eta) + \mathcal{I}_{\text{h}\times \text{l}\times \text{l}}[g_1,g_2,g_3](\eta) \\&\qquad + \mathcal{I}_{\text{l}\times \text{h}\times \text{l}}[g_1,g_2,g_3](\eta) + \mathcal{I}_{\text{l}\times \text{l}\times \text{l}}[g_1,g_2,g_3](\eta)\label{eq:splitI}
\end{align}
where
\begin{equation}
    \mathcal{I}_{\text{h}\times \text{l}\times \text{l}}[g_1,g_2,g_3](\eta) := \iint\limits_{\eta = \eta_1+\eta_2+\eta_3}\frac{1}{|\eta_1\eta_2\eta_3|^\frac{1}{2}}e^{i \Phi} g_1(\eta_1){g_2(\eta_2)}g_3(\eta_3) \phi(\eta_2)\phi(\eta_3)(1-\phi(\eta_1))d\eta_1d\eta_2.
\end{equation}
\begin{align*}
    \mathcal{I}_{\text{l}\times \text{l}\times \text{h}}[g_1,g_2,g_3](\eta) &:= \iint\limits_{\eta = \eta_1+\eta_2+\eta_3}\frac{1}{|\eta_1\eta_2\eta_3|^\frac{1}{2}}e^{i \Phi} g_1(\eta_1){g_2(\eta_2)}g_3(\eta_3) \phi(\eta_2)\phi(\eta_3) (1-\phi(\eta_1))d\eta_1d\eta_2 \\&=  \mathcal{I}_{\text{h}\times \text{l}\times \text{l}}[g_3,g_2,g_1](\eta).
\end{align*}
\begin{align*}
    \mathcal{I}_{\text{l}\times \text{h}\times \text{l}}[g_1,g_2,g_3](\eta) &:= \iint\limits_{\eta = \eta_1+\eta_2+\eta_3}\frac{1}{|\eta_1\eta_2\eta_3|^\frac{1}{2}}e^{i \Phi} g_1(\eta_1){g_2(\eta_2)}g_3(\eta_3) \phi(\eta_1)\phi(\eta_3) (1-\phi(\eta_2))d\eta_1d\eta_2\\&=\mathcal{I}_{\text{h}\times \text{l}\times \text{l}}[g_2,g_1,g_3](\eta).
\end{align*}
and
\begin{equation}
    \mathcal{I}_{\text{l}\times \text{l}\times \text{l}}[g_1,g_2,g_3](\eta) := \iint\limits_{\eta = \eta_1+\eta_2+\eta_3}\frac{1}{|\eta_1\eta_2\eta_3|^\frac{1}{2}}e^{i \Phi} g_1(\eta_1){g_2(\eta_2)}g_3(\eta_3) \phi(\eta_1)\phi(\eta_2)\phi(\eta_3)d\eta_1d\eta_2.
\end{equation}

As in the \eqref{4Kdv} case, we need some multilinear estimates. For the proof of Theorems \ref{thm:multi_mbo_1} and \ref{thm:multi_mbo_2}, see Sections \ref{sec:est_mbo1} and \ref{sec:est_mbo2}, respectively.

\begin{theorem}[Multilinear Estimates I]\label{thm:multi_mbo_1}
    Let $\k_j \in \left( -\frac{1}{4},\frac{1}{4}\right)$ and $g_j \in Z^{\k_j}$, for $j=1,2,3$. Then, for $0<\eta\lesssim1$.
    $$\left|\mathcal I[g_1,g_2,g_3](\eta)\right| \lesssim \prod_{j=1}^3\|g_j\|_{Z^{\kappa_j}}.$$
    Moreover, for $\eta\gg1$,
    $$\left|\mathcal I_{\text{h}\times \text{h}}[g_1,g_2,g_3](\eta)\right| \lesssim \jap{\eta}^{-\frac32 - \k_1-\k_2-\k_3}\prod_{j=1}^3\|g_j\|_{Z^{\kappa_j}}, $$
    and
     \begin{equation}
         \left|\mathcal I_{\text{h}\times \text{l}\times \text{l}}[g_1,g_2,g_3](\eta)-\frac{\pi}{4\eta^{\frac32}}g_1(\eta)g_2(0)g_3(0)\right| \lesssim \jap{\eta}^{-\frac52}\log^2\eta\prod_{j=1}^3\|g_j\|_{Z^{\kappa_j}},
     \end{equation}
\end{theorem}

When $g_j\in Z^0$, $j=1,2,3,$ the decay ensured by Theorem \ref{thm:multi_mbo_2} is not enough to ensure integrability in\footnote{The resulting logarithmic divergence is to be expected, as \eqref{dnls} presents modified scattering with logarithmic phase corrections.} $\eta$. At this level, we need another estimate, which holds only for the ansatz in \eqref{eq:ansatz_mbo}.

\begin{theorem}[Multilinear Estimates II]\label{thm:multi_mbo_2}
    For every $A, B \in \mathbb C$ and $a\in \R$ with $|a|+|B|\lesssim |A|$, define
\begin{equation}\label{eq:defi_Saab}
        S_{A,a,B}(\eta):=\left(Ae^{ia\log \eta} + B \frac{e^{2i\eta^2/3+3ia\log\eta}}{\eta^2}\right)\chi(\eta),\quad \eta>0,\qquad S_{A,a,B}(-\eta)=\overline{S_{A,a,B}(\eta).}
\end{equation}
    Then, for $\eta\ge 1$,
    $$|\mathcal{R}[S_{A,a,B}]|:=\left|\mathcal{I}[S_{A,a,B}] - \frac{3|A|^2A\pi e^{ia\log \eta}}{\eta^\frac32} - e^{2i\eta^2/3+3ia\log\eta}\frac{i\pi A^3\sqrt{3}e^{-3ia\log 3}}{\eta^\frac32}\right|\lesssim |A|^3 \langle\eta\rangle^{-\frac{7}{2}}.$$
    Moreover, for $A_j, B_j \in \mathbb C$ and $a_j\in\R$ with $|a_j|+|B_j|\lesssim |A_j|$, $j=1,2$,
    \begin{align}
        \left|\mathcal R[S_{A_1,a_1,B_1}](\eta)-\mathcal R[S_{A_2,a_2,B_2}](\eta)\right| \lesssim  (|A_1|^2+|A_2|^2)(|A_1-A_2| + |B_1-B_2| + |a_1-a_2|\log \eta)\langle\eta\rangle^{-\frac{7}{2}},\label{eq:IA_contr}
    \end{align}

\end{theorem}



We are now in shape to construct our solution. First, we have to match the asymptotic behavior at large frequencies to find the correct ansatz. This will determine a system of four equations in five unknowns, $c, A, a, B$ and $z\in Z^\k$. For each $c\in \complex$ small, we then perform a fixed-point argument to solve the system and build the self-similar solution.

\begin{remark} The main difference when compared with Section \ref{sec:4kdv_fixo} is the construction of the ansatz. Therein, since the phase of the ansatz depends only on the high-frequency parameter $A$, one may perform the whole construction using $A$ as a parameter and then invert the relation between $A$ and $c$. In the \eqref{dnls} case, the ansatz also depends on the low-frequency parameter $c$, which forces us to construct the solution directly in terms of $c$. 
We point out that, while the argument is slightly different, the required estimates are essentially the same. This underlines once again the robustness of the method.

\end{remark}

We focus on the argument for $\xi>0$, as we are assuming that $\tilde{W}(-\xi)=\overline{\tilde{W}(\xi)}$. Plugging
\begin{equation}\label{eq:ansatz_W}
    \tilde{W}(\xi)=S_{A,a,B}(\xi) + z(\xi)
\end{equation}
into \eqref{eq:ss_dnls},
\begin{align}
    z(\xi) =\nonumber &\ -\left(Ae^{ia\log \xi} + B \frac{e^{2i\xi^2/3+3ia\log\xi}}{\xi^2}\right)\chi(\xi)+c+\frac{i}{2\pi^2}\int_0^\xi |\eta|^\frac12\mathcal{I}[\tilde{W}](\eta)d\eta\\=&\ -\left(Ae^{ia\log \xi} + B \frac{e^{2i\xi^2/3+3ia\log\xi}}{\xi^2}\right)\chi(\xi)+c+\frac{i}{2\pi^2}\int_0^1 |\eta|^\frac12\mathcal{I}[\tilde{W}](\eta)d\eta \\  & + \frac{i}{2\pi^2}\Bigg[\int_1^\xi \frac{\pi (12|A|^2 + 3|c|^2)Ae^{ia\log \eta}}{4\eta} d\eta + \int_1^\xi \frac{e^{2i\eta^2/3+3ia\log\eta}i\pi A^3\sqrt{3}e^{-3ia\log 3}}{\eta}d\eta\Bigg]\\ & + \frac{i}{2\pi^2}\int_1^\xi \left(|\eta|^\frac12\mathcal{I}[\tilde{W}](\eta) - \frac{\pi (12|A|^2 + 3|c|^2)Ae^{ia\log \eta}}{4\eta}-\frac{e^{2i\eta^2/3+3ia\log\eta}i\pi A^3\sqrt{3}e^{-3ia\log 3}}{\eta}\right)d\eta.
\end{align}
Since
\begin{align*}
    \int_1^\xi \frac{e^{2i\eta^2/3+3ia\log \eta}}{\eta}d\eta = \left[\frac{3e^{2i\eta^2/3+3ia\log \eta}}{4i\eta^2}\right]_1^\xi
 - \int_1^\xi\left(-\frac i2+\frac{3a}{4}\right)\frac{e^{2i\eta^2/3+3ia\log \eta} }{\eta^3}d\eta,\end{align*}
we find, for $\xi >1$
\begin{align}
    z(\xi) = &\ -\left(Ae^{ia\log \xi} + B \frac{e^{2i\xi^2/3+3ia\log\xi}}{\xi^2}\right)+c+\frac{i}{2\pi^2}\int_0^1 |\eta|^\frac12\mathcal{I}[\tilde{W}](\eta)d\eta\nonumber \\  & + \frac{i}{2\pi^2}\Bigg[\frac{\pi(12|A|^2+3|c|^2)Ae^{ia\log \eta}}{4ia} + \frac{3\sqrt{3}i\pi A^3e^{-3ia\log    3}e^{2i\eta^2/3+3ia\log \eta}}{4i\eta^2}\Bigg]_1^\xi\label{eq:defi_z0}\\  
    & + \frac{i}{2\pi^2}\int_1^\xi \left(|\eta|^\frac12\mathcal{I}[\tilde{W}](\eta) - \frac{\pi (12|A|^2 + 3|c|^2)Ae^{ia\log \eta}}{4\eta}-\frac{e^{2i\eta^2/3+3ia\log\eta}i\pi A^3\sqrt{3}e^{-3ia\log 3}}{\eta}\right)d\eta\nonumber\\ & - \frac{i}{2\pi^2}\int_1^\xi \sqrt{3}i\pi A^3e^{-3ia\log 3}\left(-\frac i2+\frac{3a}{4}\right)\frac{e^{2i\eta^2/3+3ia\log \eta} }{\eta^3}d\eta.
\end{align}

We equate the terms in $e^{ia\log\xi}$ to 0, which gives the relation
\begin{equation}\label{eq:defi_a}
    a=\frac{12|A|^2+3|c|^2}{8\pi}.
\end{equation}
Doing the same for the $e^{2i\xi^2/3+3ia\log\xi}$ terms yields
\begin{equation}\label{eq:defi_B}
    B=\frac{3\sqrt{3} iA^3}{8\pi}.
\end{equation}
Under these choices of $a$ and $B$, \eqref{eq:defi_z0} reduces to
\begin{align}
  z(\xi)& =  c+\frac{i}{2\pi^2}\int_0^1 |\eta|^\frac12\mathcal{I}[\tilde{W}](\eta)d\eta\label{eq:defi_z2}\\& - \frac{i}{2\pi^2}\Bigg(\frac{\pi(12|A|^2+3|c|^2)A}{4ia} + \frac{3\sqrt{3}\pi A^3e^{-3ia\log 3}e^{2i/3}}{4}\Bigg) \\ & + \frac{i}{2\pi^2}\int_1^\xi \left(|\eta|^\frac12\mathcal{I}[\tilde{W}](\eta) - \frac{\pi (12|A|^2 + 3|c|^2)Ae^{ia\log \eta}}{4\eta}-\frac{e^{2i\eta^2/3+3ia\log\eta}i\pi A^3\sqrt{3}e^{-3ia\log 3}}{\eta}\right)d\eta\\ & - \frac{i}{2\pi^2}\int_1^\xi \sqrt{3}i\pi A^3e^{-3ia\log 3}\left(-\frac i2+\frac{3a}{4}\right)\frac{e^{2i\eta^2/3+3ia\log \eta} }{\eta^3}d\eta,\qquad \xi>1.
\end{align}
Assuming that $z(\xi)\to 0$ as $\xi\to \infty$ and the integrals are absolutely convergent over $\R^+$, the limit $\xi \to \infty$ gives
\begin{align}
    A&= c+\frac{i}{2\pi^2}\int_0^1 |\eta|^\frac12\mathcal{I}[\tilde{W}](\eta)d\eta\label{eq:defi_A} -\frac{3\sqrt{3}\pi A^3e^{-3ia\log 3}e^{2i/3}}{4}\\ & + \frac{i}{2\pi^2}\Bigg[\int_1^\infty \left(|\eta|^\frac12\mathcal{I}[\tilde{W}](\eta) - \frac{\pi (12|A|^2 + 3|c|^2)Ae^{ia\log \eta}}{4\eta}-\frac{e^{2i\eta^2/3+3ia\log\eta}i\pi A^3\sqrt{3}e^{-3ia\log 3}}{\eta}\right)d\eta\Bigg]\\ & - \frac{i}{2\pi^2}\int_1^\infty \sqrt{3}i\pi A^3e^{-3ia\log 3}\left(-\frac i2+\frac{3a}{4}\right)\frac{e^{2i\eta^2/3+3ia\log \eta} }{\eta^3}d\eta.
\end{align}
\begin{prop}\label{prop:construcao_A}
Fix $\k\in \left(0,\frac14\right)$ and $\varepsilon>0$. For $z\in Z^\k$ and $c\in \R$ with $|c| + \|z\|_{Z^\k}<\varepsilon$,  there exists a unique $A=A(c,z)\in \complex$ satisfying \eqref{eq:defi_A} (with the parameter $a$ given by \eqref{eq:defi_a}). Moreover,
\begin{equation}
    |A(c,z_1)- A(c,z_2)|\lesssim \varepsilon^2|z_1-z_2|,\quad\mbox{for } \|z_1\|_{Z^k},\|z_2\|_{Z^k}<\varepsilon. \label{eq:contr_A}
\end{equation}
\end{prop}
\begin{proof}
    The proof follows from a fixed-point argument for the operator
\begin{align}
    \Theta[A]&= c+\frac{i}{2\pi^2}\int_0^1 |\eta|^\frac12\mathcal{I}[\tilde{W}](\eta)d\eta\nonumber -\frac{3\sqrt{3}\pi A^3e^{-3ia\log 3}e^{2i/3}}{4}\\ \label{eq:defi_A_prova}& + \frac{i}{2\pi^2}\int_1^\infty \left(|\eta|^\frac12\mathcal{I}[\tilde{W}](\eta) - \frac{\pi (12|A|^2 + 3|c|^2)Ae^{ia\log \eta}}{4\eta}-\frac{e^{2i\eta^2/3+3ia\log\eta}i\pi A^3\sqrt{3}e^{-3ia\log 3}}{\eta}\right)d\eta\\ & - \frac{i}{2\pi^2}\int_1^\infty \sqrt{3}i\pi A^3e^{-3ia\log 3}\left(-\frac i2+\frac{3a}{4}\right)\frac{e^{2i\eta^2/3+3ia\log \eta} }{\eta^3}d\eta,
\end{align}
over the space $E=\{A\in \complex: |A|\le 2\varepsilon\}$.\\

\textit{Step 1. Boundedness of $\Theta$.}  We estimate each term in \eqref{eq:defi_A_prova}, where we recall that $\tilde{W}$ is given by \eqref{eq:ansatz_W},
 $$
  \tilde{W}=S_{A,a,B} + z\in Z^0.
 $$
 Theorem \ref{thm:multi_mbo_1} gives directly
\begin{equation}\label{eq:controlo_etapeq}
        \left|\int_0^1 |\eta|^\frac12\mathcal{I}[\tilde{W}](\eta)d\eta \right| \lesssim \varepsilon^3.
\end{equation}
    For the integral involving $\mathcal{I}[\tilde{W}]$ over $(1,\infty)$, we must decompose $\mathcal{I}[\tilde{W}]$ as in \eqref{eq:splitI},
    $$
     \mathcal{I}[\tilde{W}] = \mathcal{I}_{\text{h}\times \text{h}}[\tilde{W}] + \mathcal{I}_{\text{l}\times \text{l}\times \text{h}}[\tilde{W}] + \mathcal{I}_{\text{h}\times \text{l}\times \text{l}}[\tilde{W}] + \mathcal{I}_{\text{l}\times \text{h}\times \text{l}}[\tilde{W}] + \mathcal{I}_{\text{l}\times \text{l}\times \text{l}}[\tilde{W}]$$
The last term vanishes for $\eta>1$ (since, in the corresponding frequency domain, $|\eta|\le |\eta_1|+ |\eta_2|+ |\eta_3|\ll 1$). By Theorem \ref{thm:multi_mbo_1},
    \begin{equation}\label{eq:llh_prova}
\left|\mathcal{I}_{\text{h}\times \text{l}\times \text{l}}[\tilde{W}](\eta)+\mathcal{I}_{\text{l}\times \text{h}\times \text{l}}[\tilde{W}]+\mathcal{I}_{\text{l}\times \text{l}\times \text{h}}[\tilde{W}](\eta)-\frac{3\pi}{4\eta^{\frac32}}\tilde{W}(0){\tilde{W}(0)}\tilde{W}(\eta)\right|\lesssim \varepsilon^3 \jap{\eta}^{-\frac{5}{2}}\log^2 |\eta|
    \end{equation}
By definition of $\tilde{W}(0)$, for $\eta>1$,
$$
\tilde{W}(0){\tilde{W}(0)}\tilde{W}(\eta) = |c|^2A e^{ia\log\eta}+O(|\eta|^{-\k}).
$$
Together with \eqref{eq:llh_prova}, 
  \begin{equation}\label{eq:asympt_llh}
\left|\mathcal{I}_{\text{h}\times \text{l}\times \text{l}}[\tilde{W}](\eta)+\mathcal{I}_{\text{l}\times \text{l}\times \text{h}}[\tilde{W}](\eta)-\frac{3\pi}{4\eta^{\frac32}}|c|^2Ae^{ia\log \eta}\right|\lesssim \varepsilon^3 \jap{\eta}^{-\k-\frac32}.
    \end{equation}
It remains to control $\mathcal{I}_{\text{h}\times \text{h}}[\tilde{W}]$. On the one hand, Theorem \ref{thm:multi_mbo_1} implies that
$$
\left|\mathcal{I}[\tilde{W}](\eta) - \mathcal{I}[S_{A,a,B}](\eta)\right|\lesssim \varepsilon^3\jap{\eta}^{-\frac32 - \k}, \qquad S_{A,a,B}\text{ as in \eqref{eq:defi_Saab}},
$$
since the difference can be written as a sum of $\mathcal{I}$'s involving at least one $z\in Z^\k$. On the other hand, Theorem \ref{thm:multi_mbo_2} gives
$$
\left|\mathcal{I}[S_{A,a,B}] - \frac{3|A|^2A\pi e^{ia\log \eta}}{\eta^\frac32} - e^{2i\eta^2/3+3ia\log\eta}\frac{i\pi A^3\sqrt{3}e^{-3ia\log 3}}{\eta^\frac32}\right|\lesssim \varepsilon^3 \langle\eta\rangle^{-\frac{7}{2}}.
$$
Together with \eqref{eq:asympt_llh}, this implies that
\begin{equation}\label{eq:I_sem_div}
    \left||\eta|^\frac12\mathcal{I}[\tilde{W}](\eta) - \frac{\pi (12|A|^2 + 3|c|^2)Ae^{ia\log \eta}}{4\eta}-\frac{e^{2i\eta^2/3+3ia\log\eta}i\pi A^3\sqrt{3}e^{-3ia\log 3}}{\eta}\right|\lesssim \frac{\varepsilon^3}{\jap{\eta}^{1+\k}}
\end{equation}
and thus
\begin{equation}\label{eq:controlo_asympt}
   \Bigg|\int_1^\infty \left(|\eta|^\frac12\mathcal{I}[\tilde{W}](\eta) - \frac{\pi (12|A|^2 + 3|c|^2)Ae^{ia\log \eta}}{4\eta}-\frac{e^{2i\eta^2/3+3ia\log\eta}i\pi A^3\sqrt{3}e^{-3ia\log 3}}{\eta}\right)d\eta\Bigg|\lesssim  \varepsilon^3.
\end{equation}
Estimates \eqref{eq:controlo_etapeq} and \eqref{eq:controlo_asympt} imply that
\begin{align}
    |\Theta[z,a,B,c]-c|&\lesssim \left|\int_0^1 |\eta|^\frac12\mathcal{I}[\tilde{W}](\eta)d\eta\right|\label{eq:est_A} +|A|^3\\ & + \Bigg|\int_1^\infty \left(|\eta|^\frac12\mathcal{I}[\tilde{W}](\eta) - \frac{\pi (12|A|^2 + 3|c|^2)Ae^{ia\log \eta}}{4\eta}-\frac{e^{2i\eta^2/3+3ia\log\eta}i\pi A^3\sqrt{3}e^{-3ia\log 3}}{\eta}\right)d\eta\Bigg|\\ & -\left|\int_1^\infty \sqrt{3}i\pi A^3e^{-3ia\log 3}\left(-\frac i2+\frac{3a}{4}\right)\frac{e^{2i\eta^2/3+3ia\log \eta} }{\eta^3}d\eta\right| \lesssim \varepsilon^3.
\end{align}

\textit{Step 2. Bounds on the difference of two ansatz.} Given $A_1,A_2\in \complex$, define $a_j, B_j$, $j=1,2$, via \eqref{eq:defi_a} and \eqref{eq:defi_B} and set
$$
S_j:= S_{A_j,a_j,B_j},\quad j=1,2.
$$
For $\eta \ll1 $, $S_1(\eta)-S_2(\eta) = 0$. For $\eta\gtrsim 1$, we have
    \begin{align*}
        |(S_1-S_2)(\eta)| & \leq \left|\left(e^{ia_1\log |\eta|}-e^{ia_2\log |\eta|}\right)\left(A_1  + B_1 \frac{e^{2i\xi^2}}{\xi^2}\right)\right|+\left|A_1-A_2 \right| + \frac{1}{\jap{\eta}^2}|B_1-B_2|\\
        &  \lesssim (1+\varepsilon^2 \log |\eta|) \left(|a_1-a_2|+|A_1-A_2|+|B_1-B_2| \right).
    \end{align*}
    Analogously,
    $$|(S_{1}-S_{2})'(\eta)| \lesssim \frac{1}{|\eta|}(1+\varepsilon^2\log |\eta|)\left(|a_1-a_2|+|A_1-A_2|+|B_1-B_2| \right).$$
    In particular, it follows that $S_{1}-S_{2} \in Z^{-\delta}$, for every $\delta > 0$, with $\|S_1-S_2\|_{Z^{-\delta}}\lesssim|A_1-A_2|$.

\textit{Step 3. Contraction estimate.} Let
$$
\tilde{W}_j=S_j+ z,\quad j=1,2.
$$
Then
\begin{align}
    &\qquad |\Theta[A_1]-\Theta[A_2]| \\&\lesssim \int_0^1 |\eta|^{\frac12}\left|\mathcal{I}[\tilde{W}_1]- \mathcal{I}[\tilde{W}_2]\right|d\eta + \left|\frac{3\sqrt{3}\pi A_1^3e^{-3ia_!\log 3}e^{2i/3}}{4}-\frac{3\sqrt{3}\pi A_2^3e^{-3ia_2\log 3}e^{2i/3}}{4}\right|\\& + \Bigg|\int_1^\infty \left(|\eta|^\frac12\mathcal{I}[\tilde{W}_1](\eta) - \frac{\pi (12|A_1|^2 + 3|c|^2)A_1e^{ia_1\log \eta}}{4\eta}-\frac{e^{2i\eta^2/3+3ia\log\eta}i\pi A_2^3\sqrt{3}e^{-3ia_2\log 3}}{\eta}\right)d\eta\\&\qquad -\int_1^\infty \left(|\eta|^\frac12\mathcal{I}[\tilde{W}_2](\eta) - \frac{\pi (12|A_2|^2 + 3|c|^2)A_1e^{ia_2\log \eta}}{4\eta}-\frac{e^{2i\eta^2/3+3ia\log\eta}i\pi A_2^3\sqrt{3}e^{-3ia_2\log 3}}{\eta}\right)d\eta  \Bigg|\\& + \Bigg|\int_1^\infty  A_1^3e^{-3ia_1\log 3}\left(-\frac i2+\frac{3a_1}{4}\right)\frac{e^{2i\eta^2/3+3ia_1\log \eta} }{\eta^3}- A_2^3e^{-3ia_2\log 3}\left(-\frac i2+\frac{3a_2}{4}\right)\frac{e^{2i\eta^2/3+3ia_2\log \eta} }{\eta^3}d\eta\Bigg|\\&\lesssim I + II + III + IV.
\end{align}
The terms $II$ and $IV$ are easily bounded by $\varepsilon^2|A_1-A_2|$. For the term $I$, using the multilinear structure of $\mathcal{I}$ and Theorem \ref{thm:multi_mbo_1},
$$
\int_0^1 |\eta|^{\frac12}\left|\mathcal{I}[\tilde{W}_1]- \mathcal{I}[\tilde{W}_2]\right|d\eta \lesssim \int_0^1|\eta|^{\frac12}\left(\|\tilde{W}_1\|_{Z^0}^2 + \|\tilde{W}_2\|_{Z^0}^2\right)\|\tilde{W}_1 - \tilde{W}_2\|_{Z^{-\delta}}d\eta \lesssim \varepsilon^2|A_1-A_2|.
$$
Finally, for the term $III$, we write
\begin{equation}
    \mathcal{I}[\tilde{W_1}] - \mathcal{I}[\tilde{W_1}] = \left((\mathcal{I}[\tilde{W_1}]-\mathcal{I}[S_1]) - (\mathcal{I}[\tilde{W_2}]-\mathcal{I}[S_2])\right) + (\mathcal{I}[S_1]-\mathcal{I}[S_2])\label{eq:I_difere}
\end{equation}
and decompose $III$ accordingly. Again by the multilinear structure of $\mathcal{I}$, the first term can be written as a sum of
$$
\mathcal{I}[z,S_1-S_2,g_3], \quad \text{for }g_3\in\{z, S_1, S_2\} \subset Z^0.
$$
By Theorem \ref{thm:multi_mbo_1},
$$
\int_1^\infty |\eta|^{\frac12}|\mathcal{I}[z,S_1-S_2,g_3]|(\eta)d\eta\lesssim \int_1^\infty|\eta|^{-\frac32 -k + \delta} d\eta\|z\|_{Z^\k}\|S_1-S_2\|_{Z^{-\delta}}\|g_3\|_{Z^0}\lesssim \varepsilon^2|A_1-A_2|.
$$
As for the second term in \eqref{eq:I_difere}, we must use Theorem \ref{thm:multi_mbo_2} to find
\begin{align*}
    &\Bigg|\int_1^\infty \left(|\eta|^\frac12\mathcal{I}[\tilde{W}_1](\eta) - \frac{\pi (12|A_1|^2 + 3|c|^2)A_1e^{ia_1\log \eta}}{4\eta}-\frac{e^{2i\eta^2/3+3ia\log\eta}i\pi A_2^3\sqrt{3}e^{-3ia_2\log 3}}{\eta}\right)d\eta\\&\qquad -\int_1^\infty \left(|\eta|^\frac12\mathcal{I}[\tilde{W}_2](\eta) - \frac{\pi (12|A_2|^2 + 3|c|^2)A_1e^{ia_2\log \eta}}{4\eta}-\frac{e^{2i\eta^2/3+3ia\log\eta}i\pi A_2^3\sqrt{3}e^{-3ia_2\log 3}}{\eta}\right)d\eta  \Bigg|\\ & \lesssim \varepsilon^2|A_1-A_2|\int_1^\infty \jap{\eta}^{-3}\log \eta d\eta \lesssim  \varepsilon^2|A_1-A_2|.
\end{align*}
We finally conclude that
\begin{equation}
    |\Theta[A_1]-\Theta[A_2]|\lesssim \varepsilon^2|A_1-A_2|
\end{equation}
which proves that $\Theta:E\to E$ is a contraction. By Banach's fixed point theorem, there exists a unique $A=A(c,z)$ with $|A|\le 2\varepsilon$ satisfying \eqref{eq:defi_A}.

\textit{Step 4. Proof of \eqref{eq:contr_A}.}  Arguing as in Step 3 and using the multilinear structure of $\mathcal{I}$,
$$
|A_1-A_2|\lesssim \varepsilon^2\left(\|z_1-z_2\|_{Z^\k} + |A_1-A_2|\right),
$$
from which we immediately conclude \eqref{eq:contr_A}.
\end{proof}

\begin{proof}[Proof of Theorem \ref{teorema Self Similar_dnls}]
Recalling \eqref{eq:defi_z}, the proof follows once again by a fixed-point argument for 
\begin{align}
    \Gamma[z]=\nonumber &\ -\left(Ae^{ia\log \xi} + B \frac{e^{2i\xi^2/3+3ia\log\xi}}{\xi^2}\right)\chi(\xi)+c+\frac{i}{2\pi^2}\int_0^\xi |\eta|^\frac12\mathcal{I}[\tilde{W}](\eta)d\eta\\=&\ -\left(Ae^{ia\log \xi} + B \frac{e^{2i\xi^2/3+3ia\log\xi}}{\xi^2}\right)\chi(\xi)+c+\frac{i}{2\pi^2}\int_0^1 |\eta|^\frac12\mathcal{I}[\tilde{W}](\eta)d\eta\label{eq:defi_z} \\  & + \frac{i}{2\pi^2}\Bigg[\int_1^\xi \frac{\pi (12|A|^2 + 3|c|^2)Ae^{ia\log \eta}}{4\eta} d\eta + \int_1^\xi \frac{e^{2i\eta^2/3+3ia\log\eta}i\pi A^3\sqrt{3}e^{-3ia\log 3}}{\eta}d\eta\Bigg]\\ & + \frac{i}{2\pi^2}\Bigg[\int_1^\xi \left(|\eta|^\frac12\mathcal{I}[\tilde{W}](\eta) - \frac{\pi (12|A|^2 + 3|c|^2)Ae^{ia\log \eta}}{4\eta}-\frac{e^{2i\eta^2/3+3ia\log\eta}i\pi A^3\sqrt{3}e^{-3ia\log 3}}{\eta}\right)d\eta\Bigg],
\end{align}
where $A=A(c,z)$ is given by Proposition \ref{prop:construcao_A} and $a, B$ are defined via \eqref{eq:defi_a} and \eqref{eq:defi_B} respectively. To that end, we consider the space 
    $$
    E=\left\{ z\in Z^\k :z(-\xi)=\overline{z(\xi)},\ z(0)=c,\  \|z\|_{Z^\k} \le 2 \varepsilon\right\},
    $$
endowed with the metric
$$
d(z_1,z_2) = \|z_1-z_2\|_{Z^{\k-\delta}},
$$
so that $(E,d)$ is a complete metric space. 

\textit{Step 1. Boundedness of $\Gamma$.} 
For $\xi<1$, we apply Theorem \ref{thm:multi_mbo_1} to \eqref{eq:defi_z}, yielding
$$
|\Gamma[z](\xi)-c|\lesssim |S_{A,a,B}(\xi)| + \left|\int_0^\xi|\eta|^{\frac12}\mathcal{I}[\tilde{W}](\eta)d\eta\right| \lesssim |\eta|\varepsilon^3
$$
and
$$
|(\Gamma[z])'(\xi)| \lesssim |\xi|^\frac12 |\mathcal{I}[\tilde{W}](\xi)| \lesssim \varepsilon^3
$$
For $\xi>1$, we instead use the expression \eqref{eq:defi_z2} (which holds by the choice of $a$ and $B$). Moreover, by the definition of $A=A(c,z)$,
\begin{align}
  \Gamma[z](\xi)=& \frac{i}{2\pi^2}\int_\xi^\infty \left(|\eta|^\frac12\mathcal{I}[\tilde{W}](\eta) - \frac{\pi (12|A|^2 + 3|c|^2)Ae^{ia\log \eta}}{4\eta}-\frac{e^{2i\eta^2/3+3ia\log\eta}i\pi A^3\sqrt{3}e^{-3ia\log 3}}{\eta}\right)d\eta\label{eq:z_eta_grande}\\ & - \frac{i}{2\pi^2}\int_\xi^\infty \sqrt{3}i\pi A^3e^{-3ia\log 3}\left(-\frac i2+\frac{3a}{4}\right)\frac{e^{2i\eta^2/3+3ia\log \eta} }{\eta^3}d\eta.
  \end{align}
By the computations of the first step in the proof of Proposition \ref{prop:construcao_A},
$$
|\Gamma[z](\xi)|\lesssim \varepsilon^3\int_\xi^\infty |\eta|^{-1-\k}d\eta \lesssim \varepsilon^3 |\xi|^{\k}, \quad |(\Gamma[z])'(\xi)|\lesssim \varepsilon^3|\xi|^{-1-k},
$$
which shows that $\Gamma[z]\in E$ for $z\in E$.

\textit{Step 2. Contraction estimate.} First of all, observe that, by Proposition \ref{prop:construcao_A}, $A, a, B$ are Lipschitz continuous in $z$ (with constant $\varepsilon^2$).
For $\xi<1$, the estimate is direct, we focus on $\xi>1$. Let
$$
\tilde{W_j}=S_{A_j,a_j,B_j} +z_j,\quad j=1,2.
$$
Then
\begin{align}
  &\Gamma[z_1](\xi)- \Gamma[z_2](\xi)\\&= \frac{i}{2\pi^2}\Bigg[\int_\xi^\infty \left(|\eta|^\frac12\mathcal{I}[\tilde{W_1}](\eta) - \frac{\pi (12|A_1|^2 + 3|c|^2)A_1e^{ia_1\log \eta}}{4\eta}-\frac{e^{2i\eta^2/3+3ia_1\log\eta}i\pi A_1^3\sqrt{3}e^{-3ia_1\log 3}}{\eta}\right)d\eta \\&\qquad - \int_\xi^\infty \left(|\eta|^\frac12\mathcal{I}[\tilde{W}_2](\eta) - \frac{\pi (12|A_2|^2 + 3|c|^2)A_2e^{ia_2\log \eta}}{4\eta}-\frac{e^{2i\eta^2/3+3ia_2\log\eta}i\pi A_"^3\sqrt{3}e^{-3ia_2\log 3}}{\eta}\right)d\eta\Bigg]\\ & + \frac{\sqrt{3}}{2\pi}\int_\xi^\infty  A_1^3e^{-3ia_1\log 3}\left(-\frac i2+\frac{3a_1}{4}\right)\frac{e^{2i\eta^2/3+3ia_1\log \eta} }{\eta^3} -  A_2^3e^{-3ia_2\log 3}\left(-\frac i2+\frac{3a_2}{4}\right)\frac{e^{2i\eta^2/3+3ia_"\log \eta} }{\eta^3}d\eta.
  \end{align}
These terms can be written either involving $$z_1-z_2\quad \mbox{or}\quad S_{A_1,a_1,B_1}-S_{A_2,a_2,B_2}.$$
In the first case, the same arguments of Step 2 yield a contraction estimate
$$
|(\mbox{terms in }z_1-z_2)'(\xi)|\lesssim \varepsilon^2|\xi|^{-1-\k+\delta}\|z_1-z_2\|_{Z^{\k-\delta}},\quad  |(\mbox{terms in }z_1-z_2)(\xi)|\lesssim \varepsilon^2|\xi|^{-\k+\delta}\|z_1-z_2\|_{Z^{\k-\delta}}
$$
For the second case, arguing as in the third step of the proof of Proposition \ref{prop:construcao_A},
$$
|(\mbox{terms in }S_{A_1,a_1,B_1}-S_{A_2,a_2,B_2})'(\xi)|\lesssim \varepsilon^2|\xi|^{-1-\k+\delta}|A_1-A_2|,$$
$$ |(\mbox{terms in }S_{A_1,a_1,B_1}-S_{A_2,a_2,B_2})(\xi)|\lesssim \varepsilon^2|\xi|^{-\k+\delta}|A_1-A_2|.
$$
We conclude that
$$
d(\Gamma[z_1],\Gamma[z_2])=\|\Gamma[z_1]-\Gamma[z_2]\|_{Z^{\k-\delta}}\lesssim \varepsilon^2\|z_1-z_2\|_{Z^{\k-\delta}} + \varepsilon^2|A_1-A_2|\lesssim \varepsilon^2\|z_1-z_2\|_{Z^{\k-\delta}},
$$
where we used \eqref{eq:contr_A} with $\k$ replaced by $\k-\delta$. Hence $\Gamma$ is a contraction over $E$ and the proof is finished.

\end{proof}

\section{{Self-similar solutions for NLS}}\label{sec:nls_fixo}

As in the previous cases, the key to the proof lies in suitable multilinear estimates. The proofs of these results can be found in Section \ref{sec:nls}.

\begin{theorem}[Multilinear Estimates I] \label{thm:multi_nls}
    Let $\k_j \in \left(   -\frac12, \frac12\right)$ and $g_j \in Y^{\k_j}$, for $j=1,2,3$  Then the following estimates hold:
    $$\left|\mathcal{T}[h_1,h_2,h_3](\eta)\right| \lesssim \langle\eta\rangle^{-\k_1-\k_2-\k_3} \prod_{j=1}^3 \|h_j\|_{Y^{\k_j}}. $$
\end{theorem}

\begin{theorem}[Multilinear Estimates II]
    For every $A \in \mathbb C$, the following estimates hold:
    $$\left|\mathcal T[S_A,S_A,S_A](\eta)\right| \lesssim |A|^3, \quad \text{for} \ |\eta| \lesssim 1; $$
    $$\left|\mathcal T[S_A,S_A,S_A](\eta)-\pi|A|^2A\  e^{i \frac{|A|^2}{2\pi} \log |\eta|}\right| \lesssim |A|^3\langle\eta\rangle^{-\frac{7}{2}}, \quad \text{for} \ |\eta| \gg 1.$$ 
\end{theorem}
Following the same spirit as for \eqref{dnls}, we fix $\kappa \in \left(0, \tfrac{1}{2}\right)$ and define the mapping $c_+: \mathbb{C} \times Y^{\kappa} \to \mathbb{C}$ as
\begin{equation}
    c_+(A,z) := A- \frac{i}{2\pi^2}
    \int_1^{+\infty} \left(\mathcal{T}[S_A+z](\eta)-\pi|A|^2A\ e^{-i \frac{|A|^2}{2\pi} \log |\eta|}\right)\eta^{-1}\,d\eta
\end{equation}
and analogously for $c_-$.
Observe that, for the proof of Theorem \ref{teorema Self Similar nls}, it is sufficient to perform a fixed-point argument for the operator  
$$
\Gamma_{A} [z](\xi) := c_\pm(A,z) - S_{A}(\xi) + \frac{i}{2\pi^2} \int_{\pm1}^\xi \mathcal{T}[S_A+z](\eta)\,\eta^{-1}\, d\eta,\quad \xi\in \R^\pm,
$$
in the space $Y^\k$. The proof is essentially the same as the argument used in the proof of Theorem \ref{teorema Self Similar_dnls}, the only difference resides in the low-frequency regime.  Therefore, assuming $z \in Y^\k$, we note that, for $|\xi| \lesssim 1$, by Theorem \ref{thm:multi_nls},
\begin{align}
|\Gamma_{A}  [z](\xi)| & \lesssim |c(A,z)| + \int_1^\xi |\mathcal{T}[S_A+z](\eta)|\,|\eta|^{-1}\, d\eta \\
& \lesssim |c(A,z)| + \int_1^\xi |\eta|^{-1}\, d\eta  \lesssim \log |\xi| .
\end{align}
Similarly,
\begin{align}
|(\Gamma_{A}  [z])^\prime(\xi)| & \lesssim |\mathcal{T}[S_A+z](\xi)|\,|\xi|^{-1}  \lesssim |\xi|^{-1}.
\end{align}
 Hence, $\Gamma_{A} [z] \in Y^\k$. The contraction estimates are similar, we leave them to the interested reader.

\section{Multilinear estimates for 4KdV}\label{sec:4kdv}

 This section is dedicated to the proof of Theorem \ref{Multilinear estimates for M}, that is, to the derivation of multilinear estimates for 
 $$
 \mathcal{M}[f_1,f_2,f_3,f_4](\eta)=\iiint e^{-i\Phi} \frac{f_1(\eta_1)f_2(\eta_2)f_3(\eta_3)f_4(\eta_4)}{(\eta_1\eta_2\eta_3\eta_4)^{\frac{1}{3}}}\,\, d\eta_1d\eta_2d\eta_3,\quad \eta_4= \eta- \eta_1-\eta_2-\eta_3,
$$
with $f_j \in Z^{\k_j}$, $\k_j\in \{0,\kappa\}$. If $f_j\notin Z^\k$, we assume that $f_j(\xi)=\chi(\xi)$ or $f_j(\xi)=\chi(-\xi)$. To simplify the exposition, we consider just the first possibility, as the second follows from the same computations. This section is then divided into four subsections, depending on the exact number of functions lying in $Z^\k$.

\begin{remark}
It is worth mentioning that, both in this section and in the following ones, we will consider several regions of integration. We remark that, in the computations below, the integral of a function $f$ over a region $\mathcal{A}$ should be understood as the integral of $f\phi$, where $\phi \in C^\infty(\mathbb{R}^d)$, with $d = 3,4$, is a suitable cut-off function supported in a small neighborhood of $\mathcal{A}$ such that $\phi|_\mathcal{A} \equiv 1$. In particular, if 
$$
|f(\eta_j)| \lesssim \langle \eta_j \rangle^{-\kappa}
\quad \text{and} \quad 
|\partial_{\eta_j} f(\eta_j)| \lesssim \langle \eta_j \rangle^{-\kappa - 1},
$$
then 
$$
|f(\eta_j)\phi| \lesssim \langle \eta_j \rangle^{-\kappa},
\quad
|\partial_{\eta_j}(f(\eta_j)\phi)| \lesssim \langle \eta_j \rangle^{-\kappa - 1}.
$$
\end{remark}

\subsection{Estimates for the constant term}\label{sec:constant_term}
Let $f_j \notin Z^\k$ for $j=1,2,3,4$. Then we need to estimate the term

\begin{equation}\label{M(1)}
 \mathcal{M}[f_1,f_2,f_3,f_4](\eta)= \iiint e^{-i\Phi} \frac{\chi(\eta_1)\chi(\eta_2)\chi(\eta_3)\chi(\eta_4)}{(\eta_1\eta_2\eta_3\eta_4)^{\frac{1}{3}}}\, d\eta_1d\eta_2d\eta_4 =: \mathcal{M}_{\text{constant}}(\eta),
\end{equation}
with $\eta_3= \eta_1-\eta_2-\eta_4$. First, consider $ |\eta| \lesssim 1 $. Let us assume, without loss of generality, that 
\begin{equation}\label{Ordenaçao dos etas}
|\eta_1| \geq |\eta_2| \geq |\eta_3| \geq |\eta_4|.
\end{equation}
 Now, observe that
\begin{align*}
\Phi(\eta_1,\eta_2,\eta_3,\eta_4) &= \eta^3 - \sum_{j=1}^4\eta_j^3
     = \eta^3 -\eta_4^3- \sum_{j=1}^3\eta_j^3
     = \left[(\eta - \eta_4)^3 - \sum_{j=1}^3\eta_j^3 \right]+ 3(\eta^2\eta_4 - \eta\eta_4^2).
\end{align*}
Therefore, we can rewrite \eqref{termo linear} as
\begin{align*}
 \mathcal{M}_{\text{constant}}(\eta) =: \int e^{-3i(\eta^2\eta_4 - \eta\eta_4^2)}\K(\eta- \eta_4)\eta_4^{-1/3} \chi(\eta_4)d\eta_4.
\end{align*}
with
$$
\mathcal{K}(\zeta):=\iint \limits_{\zeta= \eta_1+\eta_2+\eta_3} e^{-i (\zeta^3 - \eta_1^3-\eta_2^3-\eta_3^3)}\frac{\chi(\eta_1)\chi(\eta_2)\chi(\eta_3)}{(\eta_1\eta_2\eta_3)^{\frac{1}{3}}}\, d\eta_1d\eta_2
$$

A fundamental tool for estimating the terms from this point onward is the stationary phase lemma. For the sake of completeness, we state it below.

\begin{theorem}[Stationary phase lemma for smooth functions (see, e.g., \cite{Fedoryuk1989})]\label{fase estacionaria}
Let $\phi : \mathbb{R}^d \to \mathbb{R}$ be smooth, and assume that $\phi$ has a nondegenerate critical point at $x_0$. 
If $\psi : \mathbb{R}^d \to \mathbb{C}$ is smooth and supported in a sufficiently small\footnote{So that it does not contain any other critical points.} neighborhood of $x_0$, then the oscillatory integral
$$
I(\lambda) := \int_{\mathbb{R}^d} e^{i\lambda \phi(x)} \psi(x)\, dx
$$
satisfies, as $\lambda \to \infty$,
$$
I(\lambda)
= (2\pi )^{\frac{d}{2}} e^{\frac{i\pi}{4} \sgn D^2 \phi(x_0) }\Bigl(\det D^2 \phi(x_0)\Bigr)^{-\frac12}\,
e^{i\lambda \phi(x_0)}\,
\psi(x_0)
\,\lambda^{-d/2}
+ O\!\left(\lambda^{-d/2 -1}\right).
$$
where $\sgn D^2\phi(x_0)$ is the difference between the number of positive and negative eigenvalues of the Hessian matrix $D^2\phi(x_0)$.
\end{theorem}

\begin{lemma}\label{Lema função K}
The function $\K(\zeta)$ satisfies
\begin{equation}
|\K(\zeta)| \lesssim \jap{\zeta}^{-2},
\end{equation}
and
\begin{equation}
|\K^\prime(\zeta)| \lesssim 1,
\end{equation}
for all $\zeta \in \R.$
\end{lemma}
\begin{proof} First, let us start with $\K(\zeta)$. For the case $ |\zeta| \lesssim 1 $, it is possible to integrate by parts in the variable $ \eta_1 $ to achieve the required decay, ensuring the integrability of $ \K $. That is,  
$$  
|\K(\zeta)| \lesssim 1.  
$$

Now, consider $|\zeta| \gg 1$. We are interested in studying the oscillatory integral.
\begin{align*}
\K(\zeta)&=\iint e^{-i (\zeta^3 - \eta_1^3-\eta_2^3-\eta_3^3)}\frac{\chi(\eta_1)\chi(\eta_2)\chi(\eta_3)}{(\eta_1\eta_2\eta_3)^{\frac{1}{3}}}\,d\eta_1d\eta_2 \\
&= \zeta \iint e^{-i \zeta^3 P}\frac{\chi(\zeta p_1)\chi(\zeta p_2)\chi(\zeta p_3)}{(p_1p_2p_3)^{\frac{1}{3}}}\,\,dp_1dp_2,
\end{align*}
with $p_j= \frac{\eta_j}{\zeta}$ and 
$$
P(p_1,p_2)=1- p_1^3-p_2^3-p_3^3.
$$
To estimate $\K(\zeta)$, we need to understand the behavior of the phase function $P$. The stationary points are given by
$$
\vec{q}_1=(-1,1), \,\, \vec{q}_2=(1/3,1/3),\,\, \vec{q}_3=(1,-1)\,\text{and}\,\, \vec{q}_4=(1,1).
$$
Furthermore, we have
\begin{equation}
\det (D^2 P)(p_1,p_2)= -6 \left[ (p_3 + p_1)(p_3+p_2) -p_3^2 \right],
\end{equation}
Now, let $ \varphi_j \in C^\infty_c(\R^2) $ be a bump function around a small neighborhood of $ \vec{q}_j $ for $ j \in \{1,2,3,4\} $. Also, consider $ \varphi_5= 1 - \sum_{j=1}^4 \varphi_j $. Thus, we can write
$$
\K(\zeta)= \zeta \sum_{j=1}^5 \K_j(\zeta),
$$
where
\begin{align*}
\K_j(\zeta)&= \iint e^{-i \zeta^3 P}\frac{\chi(\zeta p_1)\chi(\zeta p_2)\chi(\zeta p_3)}{(p_1p_2p_3)^{\frac{1}{3}}} \varphi_j(p_1,p_2)\,dp_1dp_2, \quad \text{for} \,\, j=1,2,3,4,5.
\end{align*}
Observe that, for any $ j=1,2,3, 4, 5 $ we have
$$
\frac{\chi(\zeta p_1)\chi(\zeta p_2)\chi(\zeta p_3)}{(p_1p_2p_3)^{\frac{1}{3}}} \varphi_j(p_1,p_2) \in C^\infty_c(\R^2).
$$
Therefore, by the stationary phase lemma (Theorem \ref{fase estacionaria}),
\begin{align*}
\K_j(\zeta) \simeq_j \zeta^{-3} e^{i \zeta^3 P(\vec{q}_j)} + O(\zeta^{-6}),
\end{align*}
for $j=1,2,3,4$. In the case $j=5$, we are far from any stationary point, so $\K_5$ decays to any order, that is, for any $\gamma>0$.

\begin{equation*}
\K_5(\zeta)\simeq O(\zeta^{-3\gamma}).
\end{equation*}
Thus, we can conclude
\begin{equation}\label{expansao de K}
\K(\zeta) \simeq \zeta^{-2} \sum_{j=1}^4 e^{i \zeta^3 P(\vec{q}_j)} + O(\zeta^{-5}).
\end{equation}
In other words
$$
|\K(\zeta)| \lesssim |\zeta|^{-2},
$$
for $|\zeta| \gg 1$.  Now, for $\K'$, we have that
\begin{equation}
\K'(\zeta)
= -3i \zeta^2 \K(\zeta)
+ 3i \iint e^{-i (\zeta^3 - \eta_1^3-\eta_2^3-\eta_3^3)}\frac{\chi(\eta_1)\chi(\eta_2)}{(\eta_1\eta_2\eta_3)^{1/3}}\left(\eta_2^2\chi(\eta_3) - \frac{\chi(\eta_3)}{3\eta_3} + \chi'(\eta_3)
\right)\,d\eta_1d\eta_2.
\end{equation}
Using arguments similar to those employed before, we deduce that
\begin{equation}
\left| \iint e^{-i (\zeta^3 - \eta_1^3-\eta_2^3-\eta_3^3)}\frac{\chi(\eta_1)\chi(\eta_2)}{(\eta_1\eta_2\eta_3)^{1/3}}\left(\eta_2^2\chi(\eta_3) - \frac{\chi(\eta_3)}{3\eta_3} + \chi'(\eta_3)
\right)\,d\eta_1d\eta_2\right|
\lesssim 1.
\end{equation}
Therefore, by what has been proved above, we conclude that
$$
|\K'(\zeta)| \lesssim 1.
$$

\end{proof}

\begin{lemma}
    For any $\eta\in \R$,
\begin{equation}
    |\mathcal{M}_{\text{constant}}(\eta)|\lesssim \jap{\eta}^{-\frac{17}{6}}.
\end{equation}
\end{lemma}
\begin{proof}
For $|\eta| \lesssim 1$,  by Lemma \ref{Lema função K}, it follows that  
\begin{align*}  
|\mathcal{M}_{\text{constant}}(\eta)| \lesssim \int |\K(\eta-\eta_4)| |\eta_4|^{-1/3} |\chi(\eta_4)|\, d\eta_4 \lesssim \int |\eta_4|^{-1/3}\jap{\eta-\eta_4}^{-2} \, d\eta_4 \lesssim 1.  
\end{align*} 
Now, suppose $ |\eta| \gg 1 $. Writing $\eta_j =\eta p_j$, we have
\begin{align}
\Phi(\eta_1,\eta_2,\eta_3,\eta_4) &= \eta^3 - \sum_{i=1}^4 \eta_j^3 = \eta^3 \left( 1- \sum_{i=1}^4 p_j^3\right) := \eta^3 Q(p_1,p_2,p_3,p_4)
\end{align}
and thus
\begin{equation}
\mathcal{M}_{\text{constant}}(\eta)= \eta^{5/3} \iiint e^{-i\eta^3 Q} \frac{\chi(\eta p_1)\chi(\eta p_2)\chi(\eta p_3)\chi(\eta p_4)}{(p_1p_2p_3 p_4)^{\frac{1}{3}}}\,dp_1 dp_2dp_3.
\end{equation}
Note that the phase $ Q $ has four stationary points:
$$
P_1=  \left(\frac{1}{4},\frac{1}{4},\frac{1}{4}\right), \quad P_2=\left(\frac{1}{2} ,\frac{1}{2} ,\frac{1}{2} \right), \quad P_3= \left(\frac{1}{2} ,-\frac{1}{2} ,\frac{1}{2} \right), \quad P_4= \left(\frac{1}{2} ,\frac{1}{2} ,-\frac{1}{2} \right).
$$
Let $ \psi_j \in C^\infty_c(\mathbb{R}^3) $ be a bump function supported in a sufficiently small neighborhood of $ P_j $ for $ j = 1, 2, 3, 4$. Additionally, define $ \psi_5 = 1 - \sum_{j=1}^4 \psi_j $. With this, we can express
$$
 \mathcal{M}_{\text{constant}}(\eta)= \sum_{j=1}^5 \mathcal{M}_{\text{constant},j}(\eta),
$$
where
\begin{equation}\label{F_j(n)}
\mathcal{M}_{\text{constant},j}(\eta):=\eta^{5/3} \iiint e^{-i\eta^3 Q} \frac{\chi(\eta p_1)\chi(\eta p_2)\chi(\eta p_3)\chi(\eta p_4)}{(p_1p_2p_3 p_4)^{\frac{1}{3}}}\psi_j(p_1,p_2,p_3)\,dp_1 dp_2dp_3.
\end{equation}

\noindent\textbf{Case 1: Stationary case.}
First, note that for every $ j=1, \dots, 4$, we have
$$
\frac{\chi(\eta p_1)\chi(\eta p_2)\chi(\eta p_3)\chi(\eta p_4)}{(p_1p_2p_3p_4)^{\frac{1}{3}}}\, \psi_j(p_1,p_2,p_3)= \frac{\psi_j(p_1,p_2,p_3)}{(p_1p_2p_3p_4)^{\frac{1}{3}}}\, \in C^\infty_c(\R^3).
$$
Since all stationary points are non-degenerate, by the stationary phase lemma (Theorem \ref{fase estacionaria}) we have
$$
\mathcal{M}_{\text{constant},j}(\eta) \simeq   e^{i\eta^3 Q(P_j)} |\eta|^{-17/6} + O(|\eta|^{-15/2}),
$$
and thus
$$
|\mathcal{M}_{\text{constant},j}(\eta)| \lesssim |\eta|^{-17/6},\quad \text{for}\,\,\, j=1,2,3,4.
$$

\noindent\textbf{Case 2: Non-stationary case.} Now we need to focus on estimating the term $ \M_5(\eta) $. Note that
$$
g(p_1,p_2,p_3)=\frac{\chi(\eta p_1)\chi(\eta p_2)\chi(\eta p_3)\chi(\eta p_4)}{(p_1p_2p_3p_4)^{\frac{1}{3}}}\, \psi_j(p_1,p_2,p_3),\quad p_4=1-p_1-p_2-p_3,
$$
satisfies, for any multi-index $\alpha$,
$$
\|D^\alpha g\|_{L^\infty(\R)}\lesssim |\eta|^{|\alpha|},
$$
As we are far from any stationary point, repeated integration by parts yield arbitrary decay (notice that each IBP yields an $\eta^{-3}$ factor, which compensates the growth of $D^\alpha g$). In particular,
$$
|\mathcal{M}_{\text{constant},5}(\eta)| \lesssim |\eta|^{-\frac{17}{6}}.
$$
\end{proof}
\subsection{Estimates for the linear term}
Now, consider $f_1,f_2,f_3 \notin Z^\k$ and $f_4 \in Z^\k$. We focus on the term
\begin{align}\label{termo linear}
 \mathcal{M}[f_1,f_2,f_3,f_4](\eta)&\simeq \int e^{-i\Phi} \frac{\chi(\eta_1)\chi(\eta_2)\chi(\eta_3)}{(\eta_1\eta_2\eta_3\eta_4)^{\frac{1}{3}}} f_4(\eta_4)\, d\eta_1d\eta_2d\eta_4=: \mathcal{M}_{\text{linear}},\quad \eta_3= \eta - \eta_1-\eta_2-\eta_4,
\end{align}
As in Section \ref{sec:constant_term}, we can rewrite this term as
\begin{align*}
 \mathcal{M}_{\text{linear}}(\eta) &= \int e^{-3i(\eta^2\eta_4 - \eta\eta_4^2)}\left[ \iint e^{-i (\zeta^3 - \eta_1^3-\eta_2^3-\eta_3^3)}\frac{\chi(\eta_1)\chi(\eta_2)\chi(\eta_3)}{(\eta_1\eta_2\eta_3)^{\frac{1}{3}}}\, d\eta_1d\eta_2 \right]\eta_4^{-1/3} f_4(\eta_4)\, d\eta_4\\   
 &:= \int e^{-3i(\eta^2\eta_4 - \eta\eta_4^2)}\K(\eta-\eta_4)\eta_4^{-1/3} f_4(\eta_4)\, d\eta_4.
\end{align*}

\begin{lemma}
    For any $\eta\in \R$,
\begin{equation}
    |\mathcal{M}_{\text{linear}}(\eta)|\lesssim \jap{\eta}^{-\frac{4}{3}-k}.
\end{equation}
\end{lemma}
\begin{proof}

\noindent\textbf{Case 1:} $|\eta- \eta_4| \gtrsim |\eta|.$ By the expression \eqref{expansao de K} in the proof of Lemma \ref{Lema função K}, we have 
\begin{equation}
\mathcal{M}_{\text{linear}}(\eta) =  \sum_{j=1}^4 \int e^{-i \Phi^j }(\eta-\eta_4)^{-2} \eta_4^{-1/3} f_4(\eta_4)\, d\eta_4 + \int e^{-3i(\eta^2\eta_4 - \eta\eta_4^2)}O((\eta-\eta_4)^{-5})\eta_4^{-1/3} f_4(\eta_4)\, d\eta_4,
\end{equation}
with $\Phi^j = 3\eta^2\eta_4 - 3\eta\eta_4^2 - P(\vec{q}_j)(\eta-\eta_4)^3$, \textit{i.e.}
$$
\Phi^j = 3\eta^2\eta_4 - 3\eta\eta_4^2, \quad j=1,2,3.
$$
and  $\Phi^4 = 3\eta^2\eta_4 - 3\eta\eta_4^2- \frac{8}{9}(\eta-\eta_4)^3$. Let us start estimating the term
\begin{align*}
I &:=\int e^{-i( 3\eta^2\eta_4 - 3\eta\eta_4^2 )}(\eta-\eta_4)^{-2} \eta_4^{-1/3} f_4(\eta_4)\, d\eta_4 \\
&\simeq  \int e^{-i \eta(\eta-\eta_4)\eta_4} (\eta- \eta_4)^{-2}\eta_4^{-1/3} f_4(\eta_4)\, d\eta_4
\end{align*}
We split our analysis in two cases:

\noindent\textbf{Case 1.1: } $|\eta_4|\gtrsim |\eta|$. Integrating by parts, since $|\eta- \eta_4| \gtrsim |\eta|$, we can estimate by
\begin{align*}
 |I| &\lesssim |\eta|^{-2} \int_{|\eta_4|\gtrsim |\eta|} |\eta_4|^{-4/3 - \k} \, d\eta_4   \lesssim |\eta|^{-7/3 - \k}.
\end{align*}

\noindent \textbf{Case 1.2: } $|\eta_4| \ll |\eta|$. In this region, do not integrate by parts, because this could increase the singularity in the variable $\eta_4$. Since $\eta_4$ is small, we have
$$
|\eta - \eta_4| \sim |\eta|
$$
then
\begin{align*}
|I| &\lesssim   \int |\eta - \eta_4|^{-2}  |\eta_4|^{-1/3}| f_4(\eta_4)| d\eta_4\lesssim \int |\eta - \eta_4|^{-2} |\eta_4|^{-1/3}\langle \eta_4\rangle^{-\k} d\eta_4\\
&\lesssim |\eta|^{-2} \int  |\eta_4|^{-1/3}\langle \eta_4\rangle^{-\k} d\eta_4\lesssim   {|\eta|^{-4/3-\k}},
\end{align*}
since that $\k < \frac{2}{3}$. Then
\begin{equation}
|I| \lesssim|\eta|^{-4/3-\k}.
\end{equation}

In a closely analogous manner, one can derive for $ \k \in  \left(\frac58, \frac23 \right) $ that
$$
\left| \int e^{-i \Phi^4 }(\eta-\eta_4)^{-2} \eta_4^{-1/3} f_4(\eta_4)\, d\eta_4 
 \right|\lesssim |\eta|^{-4/3- \k}
$$
and
$$
\left| \int e^{-i\Phi_2}O((\eta-\eta_4)^{-5})\eta_4^{-1/3} f_4(\eta_4)\, d\eta_4 \right| \lesssim |\eta|^{-4/3- \k}.
$$

 \noindent\textbf{Case 2:} $|\eta- \eta_4| \ll |\eta|:$In this region, we have $ |\eta| \sim |\eta_4| $, and in this case, 
$$
|\partial_{\eta_4} (\eta^2\eta_4 - \eta\eta_4^2)| \sim |\eta|^2.
$$
Moreover, by Lemma~\ref{Lema função K}, we know that $ |\K'(\eta - \eta_4)| \lesssim 1 $. 
Therefore, performing an integration by parts yields
$$
| \mathcal{M}_{\text{linear}}(\eta)| \lesssim \langle \eta \rangle^{-4/3 - \kappa},
$$
which completes the proof.

\end{proof}
\subsection{Estimates for the bilinear term}
Now, let $f_1, f_2 \notin Z^\k$ and $f_3, f_4 \in Z^\k$. Hence, we need to estimate
\begin{align*}
 \mathcal{M}[f_1,f_2,f_3,f_4](\eta)=\iiint e^{-i\Phi}\frac{\chi(\eta_1)\chi(\eta_2)f_3(\eta_3)f_4(\eta_4)}{(\eta_1\eta_2\eta_3\eta_4)^{1/3}}\, d\eta_3d\eta_2d\eta_1=: \mathcal{M}_{\text{bilin}}.
\end{align*}
Note that
\begin{align*}
   \Phi&= \eta^3 - \sum_{j=1}^4 \eta_j^3 \\
   &= \left[(\eta_1+\eta_2)^3- \eta_1^3- \eta_2^3\right] +\left[(\eta_3+\eta_4)^3- \eta_3^3- \eta_4^3\right]  + \left[\eta^3  - (\eta_1+\eta_2)^3- (\eta_3+\eta_4)^3\right]\\
   &=:  \Phi_{12} + \Phi_{34} + \Upsilon,
\end{align*}
where
$$
\Phi_{12} =(\eta_1+\eta_2)^3- \eta_1^3- \eta_2^3,
$$
$$
\Phi_{34} = (\eta_3+\eta_4)^3- \eta_3^3- \eta_4^3,
$$
and 
$$
\Upsilon = \eta^3  - (\eta_1+\eta_2)^3- (\eta_3+\eta_4)^3.
$$
Therefore, if we define $\zeta= \eta_1+ \eta_2= \eta-\eta_3-\eta_4$ and consider the change of variables $\eta_1 \mapsto \zeta$,
\begin{align*}
 \mathcal{M}_{\text{bilin}}(\eta)&=\int e^{-i \Upsilon} F_{12}(\zeta) F_{34} (\eta- \zeta) \, d\zeta
\end{align*}
with
\begin{equation}\label{F12}
F_{12}:=\int e^{-i \Phi_{12}} \frac{\chi(\eta_1)\chi(\eta_2)}{(\eta_1\eta_2)^{1/3}} \, d\eta_2  
\end{equation}

\begin{equation}\label{F13}
 F_{34}:=\int e^{-i \Phi_{34}} \frac{f_3(\eta_3)f_4(\eta_4)}{(\eta_3\eta_4)^{1/3}} \, d\eta_3.   
\end{equation}

\begin{lemma}
The function $F_{12}$ satisfies
$$
|F_{12}(\zeta)| \lesssim |\zeta|^{-1/9}, \quad \text{for } |\zeta| \lesssim 1,
$$  
Moreover, for $|\zeta|\gg 1$,
\begin{equation}\label{asymp_F12}
     F_{12}(\zeta) \simeq  e^{-3i\zeta^3/4 }\zeta^{-7/6} + O\left(\zeta^{-25/6}\right). 
\end{equation}
In particular,
$$
|F_{12}(\zeta)| \lesssim |\zeta|^{-7/6}, \quad \text{for } |\zeta| \gg 1.
$$
\end{lemma}
\begin{proof}
    
We begin with the high-frequency case $|\zeta|\gg 1$. Observe that, for $\zeta$ fixed,
$$
\left|\frac{\partial \Phi_{12}}{\partial \eta_1}\right|\sim |\eta_1^2-\eta_2^2|\sim |\eta||\eta_1-\eta_2|
$$
If  $ \eta_1 \not\simeq \eta_2 $, we are far from a stationary point of the phase $\Phi_{12}$.
 In this case, we can integrate by parts in \eqref{F12} to achieve a decay of arbitrarily high order. If $\eta_1 \simeq \eta_2 \simeq \frac{\zeta}{2}$, define $ p_i = \frac{\eta_i}{\zeta} $ for $ i = 1, 2 $. Then
\begin{align*}
\Phi_{12} &= \zeta \eta_1 \eta_2= \zeta^3 p_1 p_2:= \zeta^3 Q_{12}.
\end{align*}
Applying the change of variables $\eta_2 \mapsto p_2$,
$$
F_{12}(\zeta) = \zeta^{1/3} \int e^{-i \zeta^3 Q_{12}} \frac{\chi(\zeta p_1)\chi(\zeta p_2)}{(p_1p_1)^{1/3}} \, dp_2.
$$
By the stationary phase lemma (Theorem \ref{fase estacionaria}) (as in the proof of Lemma \ref{Lema função K}),
\begin{equation}\label{F12 2}
 F_{12}(\zeta) \simeq  e^{-3i\zeta^3 /4}\zeta^{-7/6} + O\left(\zeta^{-25/6}\right),   
\end{equation}
as claimed.


In the low frequency regime $|\zeta|\lesssim  1$, write
\begin{align*}
F_{12}&=\int e^{-i \Phi_{12}} \frac{\chi(\eta_1)\chi(\eta_2)}{(\eta_1\eta_2)^{1/3}} \, d\eta_2\\
&= \int_{|\eta_2|\leq |\zeta|^{-1/3}} e^{-i \Phi_{12}} \frac{\chi(\eta_1)\chi(\eta_2)}{(\eta_1\eta_2)^{1/3}} \, d\eta_2+ \int_{|\eta_2|> |\zeta|^{-1/3}} e^{-i \Phi_{12}} \frac{\chi(\eta_1)\chi(\eta_2)}{(\eta_1\eta_2)^{1/3}} \, d\eta_2=: F_{12}^1 + F_{12}^2.
\end{align*}
For $F^1_{12}$, a direct estimate yields
\begin{align*}
|F_{12}^1| = \left|  \int_{|\eta_2|\leq |\zeta|^{-1/3}} e^{-i \Phi_{12}} \frac{\chi(\eta_1)\chi(\eta_2)}{(\eta_1\eta_2)^{1/3}} \, d\eta_2 \right|\simeq  \int_{|\eta_2|\leq |\zeta|^{-1/3}}  \frac{1}{|\eta_2|^{2/3}}\, d\eta_2 \lesssim |\zeta|^{-1/9}.
\end{align*}
For $F_{12}^2$, by an integration by parts,
\begin{align}
|F_{12}^2| &\lesssim |\zeta|^{-\frac19}+\int_{|\eta_2|> |\zeta|^{-1/3}} \frac{1}{|\zeta| |\eta_2|^{8/3}} \, d\eta_2\lesssim |\zeta|^{-\frac19}.
\end{align}
Thus, we conclude that  
$$
|F_{12}(\zeta)| \lesssim |\zeta|^{-1/9}, \quad \text{for } |\zeta| \lesssim 1,
$$ 
\end{proof}

By performing similar computations and recalling that $ z \in Z^\k $, we have the following lemma.
\begin{lemma}\label{lem:est_F34}
For $\k>1/2$, let $f_3, f_4 \in Z^\k(\R)$. Then, the function $F_{34}$ satisfies
\begin{equation}\label{eq:F34}
    |F_{34}(\eta -\zeta)| \lesssim \jap{\eta- \zeta}^{-5/3 - \k}
\end{equation}
and
\begin{equation}\label{eq:F'34}
    |F^\prime_{34}(\eta-\zeta)| \lesssim |\eta-\zeta|^{-\frac23},\quad \mbox{ for }|\eta-\zeta|<1.
\end{equation}
\end{lemma}
\begin{proof} We begin with \eqref{eq:F34}. Note that, in the case $|\eta-\zeta|>1$, unlike $F_{12}$, we cannot directly apply stationary phase arguments or perform multiple integrations by parts to control the term $F_{34}$, since the functions $f_j$ do not possess enough regularity for such manipulations. Therefore, we need to proceed more carefully with our estimates. 

First, denote $\rho := \eta - \zeta$. Without loss of generality, assume that $|\eta_3| \geq |\eta_4|$, so that $|\eta_3| \gtrsim |\rho|$. We focus on the case $|\rho| \gg 1$, and divide the analysis into the following cases:

\medskip
\noindent \textbf{Case 1:} $|\eta_4| \ll |\rho|^{-2}$. Since $|\eta_4|$ is small and $\rho = \eta_3 + \eta_4$, we have $|\eta_3| \sim |\rho|$. Estimating directly and using the bound $|f_j(\eta_j)| \lesssim \langle \eta_j \rangle^{-\kappa}$, we obtain
\begin{align}
|F_{34}(\rho)| 
  &\lesssim |\rho|^{-1/3 - \kappa} 
  \int_{|\eta_4| \ll |\rho|^{-2}} |\eta_4|^{-1/3}\jap{\eta_4}^{-\k}\, d\eta_4 \\
  &\lesssim |\rho|^{-1/3 - \kappa} 
  \int_{|\eta_4| \ll |\rho|^{-2}} |\eta_4|^{-1/3}\, d\eta_4 \\
   & \lesssim |\rho|^{-5/3 - \kappa}.
\end{align}

\medskip
\noindent \textbf{Case 2:} $|\rho|^{-2}\lesssim |\eta_4| \le 1 $. 
In this region, we may integrate by parts in the variable $\eta_3$, and since
$$
|\eta_3^2 - \eta_4^2| \gtrsim |\rho|^2,
$$
we deduce that
\begin{align}
|F_{34}(\rho)| 
&\lesssim |\rho|^{-2} 
   \int |\eta_3|^{-4/3 - \kappa} |\eta_4|^{-1/3 - \kappa}\, d\eta_4
   + |\rho|^{-2} 
   \int |\eta_3|^{-1/3 - \kappa} |\eta_4|^{-4/3 - \kappa}\, d\eta_4 \\
&\lesssim |\rho|^{-5/3 - \kappa}.
\end{align}
\textbf{Case 3:} $|\eta_4|>1$, $|\eta_3-\eta_4|\ll |\rho|^{-\frac12}$. Once again, we integrate directly
$$
\int \frac{d\eta_4}{\jap{\eta_3}^{\k+\frac13}\jap{\eta_4}^{\k + \frac13}} \lesssim |\rho|^{-\frac12}\cdot\frac{1}{\jap{\rho}^{2\k+\frac23}}\lesssim |\rho|^{-\frac76 -2\k}. 
$$
\textbf{Case 4:} $|\eta_4|>1$, $|\eta_3-\eta_4|\gtrsim |\rho|^{-\frac12}$. Integrating by parts, the integral can be bounded as
\begin{align*}
    \left|\int e^{-i\Phi_{34}}\partial_{\eta_4}\left(\frac{f_3(\eta_3)f_4(\eta_4)}{(\eta_3^2-\eta_4^2)|\eta_3|^{\frac13}|\eta_4|^{\frac13}}\right)d\eta_4\right| &\lesssim \int \frac{\eta_3+\eta_4}{(\eta_3^2-\eta_4^2)^2}\frac{f_3(\eta_3)f_4(\eta_4)}{|\eta_3|^{\frac13}|\eta_4|^{\frac13}} d\eta_4 + \int \frac{1}{\eta_3^2-\eta_4^2}\left|\partial_{\eta_4}\left(\frac{f_3(\eta_3)f_4(\eta_4)}{|\eta_3|^{\frac13}|\eta_4|^\frac13}\right)\right|d\eta_4\\&\lesssim |\rho|^{-\frac53-2\k}\int \frac{1}{|\eta_3-\eta_4|^2}d\eta_4 + |\rho|^{-\frac43-\k}\int \frac{d\eta_4}{|\eta_3-\eta_4|\jap{\eta_4}^{\k+\frac43}}\\&\quad +  |\rho|^{-\frac73-\k}\int \frac{d\eta_4}{|\eta_3-\eta_4|\jap{\eta_4}^{\k+\frac13}} \\&\lesssim |\rho|^{-\frac76-2\k} + |\rho|^{-\frac53 -\k} \lesssim |\rho|^{-\frac53 - \k}
\end{align*}
since $\k>1/2$.

\medskip
We now move to the proof of \eqref{eq:F'34}, where we consider only $|\rho|<1$. Assume that $|\eta_3|\gtrsim |\eta_4|$. Then
\begin{equation}\label{eq:F'34_2}
    F_{34}'(\rho) 
= -3i \int e^{-i\Phi_{34}} (\rho^2 - \eta_3^2) 
   \frac{f_3(\eta_3) f_4(\eta_4)}{(\eta_3 \eta_4)^{1/3}}\, d\eta_4
   + \int e^{-i\Phi_{34}} 
     \frac{f_3'(\eta_3) f_4(\eta_4)}{(\eta_3 \eta_4)^{1/3}}\, d\eta_4
   - \frac{1}{3} \int e^{-i\Phi_{34}} 
     \frac{f_3(\eta_3) f_4(\eta_4)}{\eta_3^{4/3} \eta_4^{1/3}}\, d\eta_4.
\end{equation}
For the first term, if $|\eta_3|\lesssim |\rho|^{-\frac12}$,
$$
\left|\int e^{-i\Phi_{34}} (\rho^2 - \eta_3^2) 
   \frac{f_3(\eta_3) f_4(\eta_4)}{(\eta_3 \eta_4)^{1/3}}\, d\eta_4\right|\lesssim \int_{|\eta_3|<|\rho|^{-\frac12}} \frac{|\eta_3|^2}{|\eta_3|^{\frac13}|\eta_4|^{\frac13}\jap{\eta_3}^\k\jap{\eta_4}^\k
   }d\eta_4\lesssim |\rho|^{-\frac76+\k}.
$$
On the other hand, if $|\eta_3|\gg |\rho|^{-\frac12}$, we integrate by parts to find
$$
\left|\int e^{-i\Phi_{34}}\partial_{\eta_3}\left(\frac{(\rho^2-\eta_4^2)f_3(\eta_3)f_4(\eta_4)}{(\eta_3^2-\eta_4^2)|\eta_3|^{\frac13}|\eta_4|^{\frac13}}\right)d\eta_3\right|\lesssim |\rho|^{-1}\int |\eta_3|^{-\frac23 - 2\k} d\eta_3 \lesssim |\rho|^{-\frac76+\k}.
$$
The second term in \eqref{eq:F'34_2} is bounded directly, as there is enough decay and no singularity near 0. The same occurs for the last term when $|\eta_4|\gtrsim 1$.

It remains to estimate the last term in the region $|\eta_4|\ll 1$. If $|\eta_4|\ll |\rho|$, then we bound the integral by
$$
\int_{|\eta_4|\ll|\rho|} \frac{d\eta_4}{|\eta_3|^{\frac43}|\eta_4|^\frac13} \lesssim |\rho|^{-\frac43}\int_{|\eta_4|\ll|\rho|} \frac{d\eta_4}{|\eta_4|^\frac13}\lesssim |\rho|^{-\frac23}.
$$
If $|\eta_4|\gtrsim |\rho|$, we estimate as
$$
\int_{|\rho|\lesssim |\eta_4|\le|\eta_3|} \frac{d\eta_4}{|\eta_3|^{\frac43}|\eta_4|^\frac13} \lesssim \int_{|\rho|\lesssim |\eta_4|\le|\eta_3|} \frac{d\eta_4}{|\eta_4|^\frac53} \lesssim |\rho|^{-\frac23}.
$$


\end{proof}

We are now in position to bound 
\begin{align*}
 \mathcal{M}_{\text{bilin}}(\eta)&=\int e^{-i \Upsilon(\eta,\zeta)} F_{12}(\zeta) F_{34} (\eta- \zeta) \, d\zeta
\end{align*}
with $\Upsilon= \eta^3 - \zeta^3 - (\eta-\zeta)^3 =-3 \eta \zeta (\eta- \zeta)$. Once again, it will be necessary to divide the analysis into several regions.

\noindent \textbf{Case 1:} $|\eta|\lesssim 1$. In this scenario, we have directly
$$
| \mathcal{M}_{\text{bilin}}(\eta)|  \lesssim \int_{|\zeta|\lesssim 1}|\zeta|^{-\frac16}d\zeta + \int_{|\zeta|\gg1} |\zeta|^{-\frac76} d\zeta\lesssim 1.
$$
\noindent \textbf{Case 2:} $|\zeta|\gg 1$, $|\eta|\gg 1$ and $|\eta-\zeta|\gg 1$. 
In this case,
\begin{align*}
| \mathcal{M}_{\text{bilin}}(\eta)| \lesssim \int \jap{\zeta}^{-7/6} \jap{\eta- \zeta}^{-5/3-\k}\, d\zeta \lesssim |\eta|^{-5/3-\k}.
\end{align*}

\noindent \textbf{Case 3:} $|\eta| \gg 1$ and  $|\zeta|\lesssim 1$.
In this region, $|\eta-\zeta| \simeq |\eta| >1$ and
\begin{align*}
| \mathcal{M}_{\text{bilin}}(\eta)| \lesssim \int |\zeta|^{-1/6} |\eta- \zeta|^{-5/3-\k}\, d\zeta\lesssim |\eta|^{-5/3-\k} \int |\zeta|^{-1/6} \, d\zeta\lesssim |\eta|^{-5/3-\k}.
\end{align*}

\noindent \textbf{Case 4:} $|\zeta|\gg 1$, and $|\eta-\zeta|\lesssim 1$, which implies that $\zeta\simeq \eta$.
In this case, it is necessary to use the expansion \eqref{asymp_F12}:
\begin{align*}
 \mathcal{M}_{\text{bilin}}(\eta)&=\int e^{-i \Upsilon(\eta,\zeta)} F_{12}(\zeta) F_{34} (\eta- \zeta) \, d\zeta\\
& \simeq  \int e^{-i(\Upsilon(\eta,\zeta) +3\zeta^3/4)}|\zeta|^{-7/6} F_{34} (\eta- \zeta) \, d\zeta + \int e^{-i \Upsilon(\eta,\zeta)} O(|\zeta|^{-25/6}) F_{34} (\eta- \zeta) \, d\zeta
\end{align*}
The second term is bounded directly. For the first, since
$$
\partial_\zeta \left(\Upsilon(\eta,\zeta)+ \frac{3\zeta^3}{4}\right) = \eta(\eta -2 \zeta) - \frac{9}{4} \zeta^2 \sim \eta^2,
$$
by integrating by parts, by Lemma \ref{lem:est_F34}, 
\begin{align*}
I&=  \int e^{-i(\Upsilon(\eta,\zeta)- \zeta^3)}\partial_\zeta\left(\frac{\zeta^{-7/6} F_{34} (\eta- \zeta)}{\partial_\zeta \left(\Upsilon(\eta,\zeta)+ \frac{3\zeta^3}{4}\right)} \right)\, d\zeta\\
&=  \eta^{-2}  \int e^{-i(\Upsilon(\eta,\zeta)- \zeta^3)}|\zeta|^{-13/6} F_{34} (\eta- \zeta) \, d\zeta +  \eta^{-2}  \int e^{-i(\Upsilon(\eta,\zeta)- \zeta^3)}|\zeta|^{-7/6} \partial_\zeta F_{34} (\eta- \zeta) \, d\zeta \\ & \lesssim |\eta|^{-4/3-\k},
\end{align*}
Therefore
$$
| \mathcal{M}_{\text{bilin}}(\eta)| \lesssim |\eta|^{-4/3-\k},
$$
which concludes the proof of the bilinear estimate.

\subsection{Estimates for the higher-order terms}
Finally, we turn to the cubic or quartic terms. More precisely, consider
\begin{equation}\label{eq_Mhigh}
     \mathcal{M}[f_1,f_2,f_3,f_4](\eta) = \iiint e^{-i\Phi} \frac{f_1(\eta_1)f_2(\eta_2)f_3(\eta_3)f_4(\eta_4)}{(\eta_1\eta_2\eta_3\eta_4)^{1/3}} \, d\eta_3d\eta_2d\eta_1 := \mathcal{M}_{\text{high}}(\eta)
\end{equation}
where we may have $f_j = \chi$ for some (and only one) $j \in \{1,2,3,4\}$, and the remaining factors belong to the class $Z^\k$. By symmetry, we may assume, without loss of generality, that $\eta_1$ is chosen so that $|\eta-\eta_1|\gtrsim |\eta-\eta_j|$, for all $j$. This implies $|\eta - \eta_1| \gtrsim |\eta|$, and we order the remaining frequencies, $|\eta_2|\ge |\eta_3|\ge |\eta_4|$. In this way, we can decompose

\begin{equation}\label{Mtrilinear}
   \mathcal{M}_{\mathrm{high}}(\eta)
= \int e^{3i\eta\eta_1(\eta-\eta_1)}\, 
\mathcal{E}(\eta-\eta_1)\, 
\eta_1^{-1/3}\,
f_1(\eta_1)\, d\eta_1, 
\end{equation}

where
$$
\mathcal{E}(\gamma)=\mathcal{E}(f_2,f_3,f_4)(\gamma) :=  \iint_{\gamma=\eta_2+\eta_3+\eta_4} e^{-i(\gamma^3-\eta_2^3-\eta_3^3-\eta_4^3)} \frac{f_2(\eta_2)f_3(\eta_3)f_4(\eta_4)}{(\eta_2\eta_3\eta_4)^{1/3}} \, d\eta_2d\eta_3.
 $$
\begin{lemma}\label{Lema Estimativas E}
Let $f_2, f_3 \in Z^\k(\R)$  with $\k \in (\frac{5}{8}, \frac23)$ and $f_4 \in Z^{\k_4}(\R)$ with $\k_4 \in \{0, \k\}$. Then there exists $\theta \geq  2^+ +\k_4$ such that
$$
|\E(\gamma)| \lesssim |{\gamma}|^{-\theta}, \quad {\text{for} \quad |\gamma|\gg 1.}
$$ 
\end{lemma}
\begin{proof} As done previously, we need to divide the analysis into several regions. For brevity, we will estimate only the terms that yield the worst decay (typically, those corresponding to the lowest frequency and the highest singularity). 
Let us assume that  $|\eta_2| \geq |\eta_3| \geq |\eta_4|$. In particular, we may suppose that $|\eta_2| \sim |\gamma|$.

\medskip

\noindent \textbf{Case 1:} $|\eta_4|\geq 1$. In this situation, we further divide the analysis into two subregions:

\medskip
\noindent \textbf{Case 1.1:} $\big||\eta_2|-|\eta_4|\big|\leq 1$.  
In this regime, integration by parts is not available. However, since the frequency ordering ensures
$$
|\eta_j|\sim |\gamma|, \qquad j=2,3,4,
$$
we may estimate directly:
\begin{align*}
|\mathcal{E}(\gamma)|
&\lesssim \iint |\eta_2|^{-1/3-\kappa}|\eta_3|^{-1/3-\kappa}|\eta_4|^{-1/3-\kappa_4}\, d\eta_3 d\eta_4 \\
&\lesssim |\gamma|^{-1-2\kappa-\kappa_4}\\
&\lesssim |\gamma|^{-2^+-\kappa_4}.
\end{align*}

\medskip
\noindent \textbf{Case 1.2:} $\big||\eta_2|-|\eta_4|\big|>1$.  
Here we may integrate by parts in $\eta_2$, treating $\eta_4$ as a dependent variable. This yields
\begin{align}
\mathcal{E}(\gamma)
=&6 \iint e^{-i(\gamma^3-\eta_2^3-\eta_3^3-\eta_4^3)}
\frac{1}{(\eta_2+\eta_4)(\eta_2-\eta_4)^2}
\frac{f_2(\eta_2)f_3(\eta_3)f_4(\eta_4)}{(\eta_2\eta_3\eta_4)^{1/3}}\, d\eta_2 d\eta_3 \label{eq:I1}\\
&- \iint e^{-i(\gamma^3-\eta_2^3-\eta_3^3-\eta_4^3)}
\frac{1}{\eta_2^2-\eta_4^2}
\partial_{\eta_2}\!\left(\frac{f_2(\eta_2)f_3(\eta_3)f_4(\eta_4)}{(\eta_2\eta_3\eta_4)^{1/3}}\right)\! d\eta_2 d\eta_3 \label{eq:I2}\\
=&: I_1 + I_2. \nonumber
\end{align}

We begin with the estimate for $I_2$. In the worst case, the derivative hits the smallest frequency, so that
\begin{align}
|I_2|
&\lesssim \iint \frac{1}{|\eta_2^2-\eta_4^2|}
|\eta_2|^{-1/3-\kappa}|\eta_3|^{-1/3-\kappa}|\eta_4|^{-4/3-\kappa_4}\, d\eta_3 d\eta_4 \\
&\lesssim |\gamma|^{-4/3^- -\kappa} \iint
\frac{1}{\big||\eta_2|-|\eta_4|\big|^{1^+}}
|\eta_3|^{-1/3-\kappa}|\eta_4|^{-4/3-\kappa_4}\, d\eta_3 d\eta_4.
\end{align}

We now distinguish between several subcases:

\smallskip
\noindent $\bullet$ {If $|\eta_4|\gtrsim|\gamma|$},
\begin{align}
|I_2|
&\lesssim |\gamma|^{-8/3^--\kappa-\kappa_4}
\iint \frac{1}{\big||\eta_2|-|\eta_4|\big|^{1^+}} |\eta_3|^{-1/3-\kappa}\, d\eta_3 d\eta_4 \\
&\lesssim |\gamma|^{-2^--2\kappa-\kappa_4}\\
&\lesssim |\gamma|^{-2^+-\kappa_4}.
\end{align}

\noindent $\bullet$ {If $|\eta_4|\ll|\gamma|$ and $|\eta_3|\gtrsim|\gamma|$},  since  
$\big||\eta_2|-|\eta_4|\big|\sim|\eta_2|\gtrsim|\gamma|$,
$$
|I_2|\lesssim |\gamma|^{-5/3-2\kappa}\lesssim |\gamma|^{-2^+-\kappa_4}.
$$

\noindent $\bullet$ {If $|\eta_4|\ll|\gamma|$ and $|\eta_3|\ll|\gamma|$}, integration in $\eta_3$ gives
$$
|I_2|\lesssim |\gamma|^{-5/3-2\kappa}\lesssim |\gamma|^{-2^+-\kappa_4}.
$$

\medskip
We now turn to $I_1$, where the worst-case scenario is when $|\eta_2+\eta_4|\gtrsim |\eta_2-\eta_4|$. From \eqref{eq:I1},
\begin{align}
|I_1|
&\lesssim \iint \frac{1}{|\eta_2+\eta_4|\,|\eta_2-\eta_4|^{2}}
|\eta_2|^{-1/3-\kappa}|\eta_3|^{-1/3-\kappa}|\eta_4|^{-1/3-\kappa_4}\, d\eta_3 d\eta_4 \\
&\lesssim |\gamma|^{-4/3-\kappa}
\iint \frac{1}{|\eta_2-\eta_4|^{2}}
|\eta_3|^{-1/3-\kappa}|\eta_4|^{-1/3-\kappa_4}\, d\eta_3 d\eta_4.
\end{align}

The worst contribution occurs when $|\eta_4|\gtrsim|\gamma|$; otherwise  
$|\eta_2-\eta_4|^2\sim|\gamma|^2$ and we gain sufficient decay. Thus,
$$
|I_1|\lesssim |\gamma|^{-5/3-\kappa-\kappa_4}\lesssim |\gamma|^{-2^+-\kappa_4}.
$$

\bigskip

\noindent \textbf{Case 2:} $|\eta_4|\leq 1$.  
We split the analysis further:

\medskip
\noindent \textbf{Case 2.1:} $|\eta_4|> |\gamma|^{-\k}$.  
Here $|\eta_2^2-\eta_4^2|\sim|\gamma|^2$, so integrating by parts in $\eta_2$ and estimating the worst term (derivative hitting $\eta^{-1/3}$) gives
\begin{align*}
|\mathcal{E}(\gamma)|
&\lesssim |\gamma|^{-2}\iint
|\eta_2|^{-1/3-\kappa}|\eta_3|^{-1/3-\kappa}
|\eta_4|^{-4/3}\langle\eta_4\rangle^{-\kappa_4}\, d\eta_3 d\eta_4 \\
&\lesssim |\gamma|^{-7/3-\kappa}
\iint |\eta_3|^{-1/3-\kappa}|\eta_4|^{-4/3}\, d\eta_3 d\eta_4 \\
&\lesssim |\gamma|^{-5/3-5\kappa/3}\\
&\lesssim |\gamma|^{-2^+-\kappa_4}.
\end{align*}

\medskip
\noindent \textbf{Case 2.2:} $|\eta_4|\leq |\gamma|^{-\k}$.  
We now distinguish four subcases.

\medskip
\noindent \textbf{Case 2.2.1:} $|\eta_3|>1$ and $\big||\eta_2|-|\eta_3|\big|>|\gamma|^{-1/2 + \k/3}$.  
In this region, we integrate by parts in $\eta_2$ (with $\eta_3$ dependent):
\begin{align}
\mathcal{E}(\gamma)
=&6 \iint e^{-i(\gamma^3-\eta_2^3-\eta_3^3-\eta_4^3)}
\frac{1}{(\eta_2+\eta_3)(\eta_2-\eta_3)^2}
\frac{f_2(\eta_2)f_3(\eta_3)f_4(\eta_4)}{(\eta_2\eta_3\eta_4)^{1/3}}\, d\eta_2 d\eta_3 \\
&- \iint e^{-i(\gamma^3-\eta_2^3-\eta_3^3-\eta_4^3)}
\frac{1}{\eta_2^2-\eta_3^2}
\partial_{\eta_2}\!\left(\frac{f_2(\eta_2)f_3(\eta_3)f_4(\eta_4)}{(\eta_2\eta_3\eta_4)^{1/3}}\right)\! d\eta_2 d\eta_3\\
=&: J_1 + J_2.
\end{align}

For $J_2$, estimating the worst contribution:
\begin{align}
|J_2|
&\lesssim \iint
\frac{1}{|\eta_2+\eta_3|\,|\eta_2-\eta_3|}
|\eta_2|^{-1/3-\kappa}|\eta_3|^{-4/3-\kappa}|\eta_4|^{-1/3}\, d\eta_3 d\eta_4 \\
&\lesssim |\gamma|^{-4/3-\kappa}\iint
\frac{1}{|\eta_2-\eta_3|}
|\eta_3|^{-4/3-\kappa}|\eta_4|^{-1/3}\, d\eta_3 d\eta_4.
\end{align}

If $|\eta_3|\ll|\gamma|$, then
$$
|J_2|\lesssim |\gamma|^{-7/3-\kappa}\lesssim |\gamma|^{-2^+-\kappa_4},
$$

if $|\eta_3|\gtrsim|\gamma|$,
$$
|J_2|\lesssim |\gamma|^{-8/3-2\kappa}|\log|\gamma||\lesssim |\gamma|^{-2^+-\kappa_4}.
$$

For $J_1$,
\begin{align}
|J_1|
&\lesssim \iint \frac{1}{|\eta_2+\eta_3|\,|\eta_2-\eta_3|^2}
|\eta_2|^{-1/3-\kappa}|\eta_3|^{-1/3-\kappa}|\eta_4|^{-1/3}\, d\eta_3 d\eta_4 \\
&\lesssim |\gamma|^{-4/3-\kappa}\iint \frac{1}{|\eta_2-\eta_3|^2}
|\eta_3|^{-1/3-\kappa}|\eta_4|^{-1/3}\, d\eta_3 d\eta_4 \\
\end{align}
If $|\eta_3|\ll|\gamma|$,
$$
|J_1|\lesssim |\gamma|^{-8/3-2\kappa}\lesssim |\gamma|^{-2^+-\kappa_4}.
$$
If $|\eta_3|\gtrsim|\gamma|,$
$$
\quad |J_1|\lesssim |\gamma|^{-7/6-7\kappa/3}\lesssim |\gamma|^{-2^+-\kappa_4}.
$$

\medskip
\noindent \textbf{Case 2.2.2:} $|\eta_3|>1$ and $\big||\eta_2|-|\eta_3|\big|\leq |\gamma|^{-1/2 + \k/3}$.  
Here, integration by parts is not possible, but
$$
|\gamma|\lesssim|\eta_2|\sim|\eta_3|.
$$
Thus, a direct estimate gives
$$
|\mathcal{E}(\gamma)|\lesssim |\gamma|^{-7/6-7\kappa/3}\lesssim |\gamma|^{-2^+-\kappa_4}.
$$

\medskip
\noindent \textbf{Case 2.2.3:} $|\eta_3|<|\gamma|^{-2}<1$.  
Then
\begin{align}
|\mathcal{E}(\gamma)|
&\lesssim \iint |\eta_2|^{-1/3-\kappa}|\eta_3|^{-1/3}|\eta_4|^{-1/3}\, d\eta_3 d\eta_4\\
&\lesssim |\gamma|^{-1/3-\kappa}
\iint |\eta_3|^{-1/3} |\eta_4|^{-1/3}\, d\eta_3 d\eta_4 \\
&\lesssim|\gamma|^{-5/3-5\kappa/3}\\
&\lesssim |\gamma|^{-2^+-\kappa_4}.
\end{align}

\medskip
\noindent \textbf{Case 2.2.4:} $|\gamma|^{-2}\lesssim|\eta_3|\leq 1$.  
Here we may integrate by parts in $\eta_2$ with $\eta_3$ dependent, since  
$|\eta_2^2-\eta_3^2|\sim|\gamma|^2$. Estimating the worst term yields
\begin{align}
|\mathcal{E}(\gamma)|
&\lesssim |\gamma|^{-2}\iint
|\eta_2|^{-1/3-\kappa}|\eta_3|^{-4/3}|\eta_4|^{-1/3}\, d\eta_2 d\eta_4\\
&\lesssim |\gamma|^{-7/3-\kappa}
\iint |\eta_3|^{-4/3}|\eta_4|^{-1/3}\, d\eta_3 d\eta_4\\
&\lesssim |\gamma|^{-5/3-5\kappa/3}\\
&\lesssim |\gamma|^{-2^+-\kappa_4},
\end{align}
which completes the proof.

\end{proof}

Now, if  $| \eta - \eta_1 |\gtrsim |\eta|\gtrsim 1$, estimating \eqref{Mtrilinear} directly, we obtain by Lemma \ref{Lema Estimativas E},
\begin{align}
| \mathcal{M}_{\mathrm{high}}(\eta)|
&\lesssim \int \langle \eta - \eta_1 \rangle^{-2^+ - \kappa_4}\, |\eta_1|^{-1/3}\, |f_1(\eta_1)| \, d\eta_1 \\
&\lesssim |\eta|^{-4/3 - \kappa}
\int \langle \eta - \eta_1 \rangle^{-(2/3)^+ + \kappa - \kappa_4}\, |\eta_1|^{-1/3}\, |f_1(\eta_1)| \, d\eta_1 \\
&\lesssim |\eta|^{-4/3 - \kappa}
\int \langle \eta - \eta_1 \rangle^{-(2/3)^+ + \kappa - \kappa_4}\, |\eta_1|^{-1/3}\jap{\eta_1}^{-\k_1}\, \, d\eta_1 .
\end{align}
Since $\k_1+\k_4\ge \k$, the integral is uniformly bounded and thus
\begin{align}
|\mathcal{M}_{\mathrm{high}}(\eta)|
\lesssim |\eta|^{-4/3 - \kappa}.
\end{align}
On the other hand, if $|\eta-\eta_1|\gtrsim 1 \gg |\eta|$, then, by \eqref{Mtrilinear}
\begin{align}
| \mathcal{M}_{\mathrm{high}}(\eta)|
&\lesssim \int \langle \eta - \eta_1 \rangle^{-2^+ - \kappa_4}\, |\eta_1|^{-1/3}\, |f_1(\eta_1)| \, d\eta_1 \\
&\lesssim 
\int \langle \eta_1 \rangle^{-2^+ - \kappa_4}\, |\eta_1|^{-1/3}\jap{\eta_1}^{-\k_1}\, \, d\eta_1 \\
&< \infty.
\end{align}
Finally, if $|\eta-\eta_1|\lesssim 1$, then $|\eta_j|\lesssim 1$ for all $j \in \{1,2,3,4\}$ and we bound \eqref{eq_Mhigh} directly:
\begin{align}
| \mathcal{M}_{\mathrm{high}}(\eta)|
\lesssim \iiint |\eta_1\eta_2\eta_3\eta_4|^{-\frac13} d\eta_1d\eta_2d\eta_3
< \infty.
\end{align}

\qed

\section{Multilinear estimates for mBO - Proof of Theorem \ref{thm:multi_mbo_1}}\label{sec:est_mbo1}

We split the multilinear operator $\mathcal{I}$ as in \eqref{eq:splitI} and estimate each term separately.
For $\mathcal{I}_{\text{h}\times\text{h}}$, notice that there are four cases to consider, depending on the signs of $\eta_1,\eta_2$ and $\eta_3$ (and their permutations):
    $$
    (+,+,+),\quad (+,-,+), \quad (-.+,-),\quad (-,-,-).
    $$
Since the third and fourth cases can be reduced to the second and first cases by complex conjugation of the operator, we prove the multilinear estimate for the first two configurations. As such, we consider the operators
$$
\mathcal{I}_{\text{h}\times\text{h}}^{(+,+,+)}[g_1,g_2,g_3](\eta):=\iint\limits_{\eta_1,\eta_2,\eta_3>0}\frac{1}{|\eta_1\eta_2\eta_3|^\frac{1}{2}}e^{i (\eta^2-\eta_1^2-\eta_2^2-\eta_3^2)} g_1(\eta_1){g_2(\eta_2)}g_3(\eta_3)\sum_{j=1}
^3\prod_{k\neq j}(1-\phi(\eta_j)) d\eta_1d\eta_2.
$$
and
$$
\mathcal{I}_{\text{h}\times\text{h}}^{(+,-,+)}[g_1,g_2,g_3](\eta):=\iint\limits_{\eta_1,\eta_3>0, \eta_2<0}\frac{1}{|\eta_1\eta_2\eta_3|^\frac{1}{2}}e^{i (\eta^2-\eta_1^2+\eta_2^2-\eta_3^2)} g_1(\eta_1){g_2(\eta_2)}g_3(\eta_3)\sum_{j=1}
^3\prod_{k\neq j}(1-\phi(\eta_j)) d\eta_1d\eta_2.
$$

\subsection{Bounds for $\mathcal{I}_{\text{h}\times\text{h}}^{(+,+,+)}$}\label{sec:Ihh}

\begin{lemma}
    Let $\k_j \in \left( -\frac{1}{4},\frac{1}{4}\right)$ and $g_j \in Z^{\k_j}$, for $j=1,2,3$. Then
    $$\left|\mathcal I_{\text{h}\times \text{h}}^{(+,+,+)}[g_1,g_2,g_3](\eta)\right| \lesssim \jap{\eta}^{-\frac32 - \k_1-\k_2-\k_3}\prod_{j=1}^3\|g_j\|_{Z^{\kappa_j}}, $$
\end{lemma}
\begin{proof}
\textit{Step 1. $|\eta_3|>1$.} We observe that $|\eta_j|^{-\frac{1}{2}} \simeq \langle \eta_j\rangle^{-\frac{1}{2}}$. In particular, given $g_j \in Z^{\k_j}$, it is not hard to check that $f_j:=|\eta_j|^{-\frac{1}{2}}g_j \in Z^{\k_j+\frac{1}{2}}$. In particular, we are interested in estimating\footnote{The cut-off functions in $\eta_1$, $\eta_2$ or $\eta_3$ can be included in the functions $f_1,f_2$ and $f_3$.} 
    $$\iint\limits_{|\eta_j| \ge1}e^{i \Phi} f_1(\eta_1){f_2(\eta_2)}f_3(\eta_3) d\eta_1d\eta_2,\quad f_j\in Z^{\k+\frac12}.$$
    
In this case, the phase function reads
$$
\Phi=\eta^2-\eta_1^2-\eta_2^2-\eta_3^2
$$
    We suppose, without loss of generality, that $|\eta_1|\ge |\eta_2|\ge |\eta_3|$, which implies that $|\eta_1|\gtrsim |\eta|$. We start by integrating by parts in $\eta_1$ using the relation
    $$
    \partial_{\eta_1}\left((\eta_1-\eta_3)e^{i\Phi}\right)=2(1-i(\eta_1-\eta_3)^2)e^{i\Phi}.
    $$
    The resulting integral is given by
\begin{align*}
    &\iint e^{i\Phi}\frac{(\eta_1-\eta_3)}{2(1-i(\eta_1-\eta_3)^2)}f_2(\eta_2)\left(f_1'(\eta_1)f_3(\eta_3)+f_1(\eta_1)f_3'(\eta_3)\right)d\eta_1d\eta_2 \\+ & \iint e^{i\Phi}\frac{(\eta_1-\eta_3)^2}{2(1-i(\eta_1-\eta_3)^2)^2}f_1(\eta_1)f_2(\eta_2)f_3(\eta_3)d\eta_1d\eta_2 =: I_1 + I_2+ I_3.
\end{align*}

For $I_1$, we can integrate by parts in $\eta_2$, since
$$
    \partial_{\eta_2}\left((\eta_2-\eta_3)e^{i\Phi}\right)=2(1-i(\eta_2-\eta_3)^2)e^{i\Phi}.
$$
This leads us to estimate
\begin{align*}
    &\iint e^{i\Phi}\frac{(\eta_1-\eta_3)}{2(1-i(\eta_1-\eta_3)^2)}\frac{(\eta_2-\eta_3)}{2(1-i(\eta_2-\eta_3)^2)}f_1'(\eta_1)(f_2(\eta_2)f_3'(\eta_3)+f_2'(\eta_2)f_3(\eta_3))d\eta_1d\eta_2 \\  + &\iint e^{i\Phi}\partial_{\eta_2}\left(\frac{(\eta_1-\eta_3)}{2(1-i(\eta_1-\eta_3)^2)}\frac{(\eta_2-\eta_3)}{2(1-i(\eta_2-\eta_3)^2)}\right)f_1'(\eta_1)f_2(\eta_2)f_3(\eta_3)d\eta_1d\eta_3.
    \end{align*}
We estimate
$$
\left|\frac{(\eta_1-\eta_3)}{2(1-i(\eta_1-\eta_3)^2)}\frac{(\eta_1-\eta_2)}{2(1-i(\eta_1-\eta_2)^2)}\right|\lesssim \frac{1}{\jap{\eta_1-\eta_3}\jap{\eta_1-\eta_2}}
$$
and
$$
\left|\partial_{\eta_2}\left(\frac{(\eta_1-\eta_3)}{2(1-i(\eta_1-\eta_3)^2)}\frac{(\eta_2-\eta_3)}{2(1-i(\eta_2-\eta_3)^2)}\right)\right|\lesssim \frac{1}{\jap{\eta_1-\eta_3}\jap{\eta_2-\eta_3}^2}.
$$
Therefore
\begin{align*}
    |I_1|&\lesssim \iint \frac{1}{\jap{\eta_1-\eta_3}\jap{\eta_2-\eta_3}\jap{\eta_1}^{\k_1+\frac32}\jap{\eta_2}^{\k_2+\frac12}\jap{\eta_3}^{\k_3+\frac32}}d\eta_1d\eta_2 \\&+ \iint \frac{1}{\jap{\eta_1-\eta_3}\jap{\eta_2-\eta_3}^2\jap{\eta_1}^{\k_1+\frac32}\jap{\eta_2}^{\k_2+\frac12}\jap{\eta_3}^{\k_3+\frac12}}d\eta_1d\eta_2\\& \lesssim \jap{\eta}^{-\k_1-\frac52} + \jap{\eta}^{-\k_1-\k_3-2} \lesssim \jap{\eta}^{-\k_1-\k_2-\k_3-\frac32}.
\end{align*}

We now focus on the $I_2$ term. Integrating now in $\eta_1,\eta_3$, we may integrate by parts once again in $\eta_1$, using the relation 
    $$
    \partial_{\eta_1}\left((\eta_1-\eta_2)e^{i\Phi}\right)=(1-2i(\eta_1-\eta_2)^2)e^{i\Phi}.
    $$
    The resulting terms are
    \begin{align*}
    &\iint e^{i\Phi}\frac{(\eta_1-\eta_3)}{2(1-i(\eta_1-\eta_3)^2)}\frac{(\eta_1-\eta_2)}{2(1-i(\eta_1-\eta_2)^2)}(f_1'(\eta_1)f_2(\eta_2)+f_1(\eta_1)f_2'(\eta_2))f_3'(\eta_3)d\eta_1d\eta_3 \\  + &\iint e^{i\Phi}\partial_{\eta_1}\left(\frac{(\eta_1-\eta_3)}{2(1-i(\eta_1-\eta_3)^2)}\frac{(\eta_1-\eta_2)}{2(1-i(\eta_1-\eta_2)^2)}\right)f_1(\eta_1)f_2(\eta_2)f_3'(\eta_3)d\eta_1d\eta_3.
    \end{align*}
Since,
$$
\left|\partial_{\eta_1}\left(\frac{(\eta_1-\eta_3)}{2(1-i(\eta_1-\eta_3)^2)}\frac{(\eta_1-\eta_2)}{2(1-i(\eta_1-\eta_2)^2)}\right)\right|\lesssim \frac{1}{\jap{\eta_1-\eta_3}\jap{\eta_1-\eta_2}^2},
$$
we find
\begin{align*}
    |I_2|&\lesssim \iint \frac{1}{\jap{\eta_1-\eta_3}\jap{\eta_1-\eta_2}\jap{\eta_1}^{\k_1+\frac12}\jap{\eta_2}^{\k_2+\frac32}\jap{\eta_3}^{\k_3+\frac32}}d\eta_1d\eta_2 \\&+  \iint \frac{1}{\jap{\eta_1-\eta_3}\jap{\eta_1-\eta_2}^2\jap{\eta_1}^{\k_1+\frac12}\jap{\eta_2}^{\k_2+\frac12}\jap{\eta_3}^{\k_3+\frac32}}d\eta_1d\eta_2\\&\lesssim \jap{\eta}^{-\k_1-\frac52} + \jap{\eta}^{-\k_1-\k_2-2}\lesssim \jap{\eta}^{-\k_1-\k_2-\k_3-\frac32}.
\end{align*}
For the $I_3$ term, we integrate by parts again in $\eta_1$, which yields
\begin{align*}
     I_3=&-\iint e^{i\Phi}\frac{(\eta_1-\eta_3)^3}{2(1-i(\eta_1-\eta_3)^2)^3}f_2(\eta_2)(f_1'(\eta_1)f_3(\eta_3)+f_1(\eta_1)f_3'(\eta_3))d\eta_1d\eta_2 \\&-  \iint e^{i\Phi}(\eta-\eta_1)\partial_{\eta_1}\left(\frac{(\eta_1-\eta_3)^2}{2(1-i(\eta_1-\eta_3)^2)^3}\right)f_1(\eta_1)f_2(\eta_2)f_3(\eta_3)d\eta_1d\eta_2 
\end{align*}
which we bound directly as
\begin{align*}
    |I_3|\lesssim & \ \iint \frac{1}{\jap{\eta_1-\eta_3}^3\jap{\eta_1}^{\k_1+\frac12}\jap{\eta_2}^{\k_2+\frac12}\jap{\eta_3}^{\k_3+\frac32}}d\eta_1d\eta_2 \\&+ \iint \frac{1}{\jap{\eta_1-\eta_3}^4\jap{\eta_1}^{\k_1+\frac12}\jap{\eta_2}^{\k_2+\frac12}\jap{\eta_3}^{\k_3+\frac12}}d\eta_1d\eta_2 \\\lesssim &\ 
    \jap{\eta}^{-\k_1-\frac72} +\jap{\eta}^{-\k_1-\k_2-\k_3-\frac32}\lesssim \jap{\eta}^{-\k_1-\k_2-\k_3-\frac32}.
    \end{align*}

\textit{Step 2. $|\eta_3|<1$.} In this case, we have to be careful when integrating by parts, as the weight $|\eta_3|^{-\frac12}$ may introduce strong singularities. Write the integral as
$$
\iint\limits_{|\eta_j| \ge1}e^{i \Phi} f_1(\eta_1){f_2(\eta_2)}\frac{g_3(\eta_3)}{|\eta_3|^{\frac12}} d\eta_1d\eta_3,\quad f_j\in Z^{\k+\frac12}.
$$
Since
$$
    \partial_{\eta_1}\left((\eta_1-\eta_2)e^{i\Phi}\right)=2(1-i(\eta_1-\eta_2)^2)e^{i\Phi},
$$
integrating by parts leads to
\begin{align*}
    &\iint e^{i\Phi}\frac{(\eta_1-\eta_2)}{2(1-i(\eta_1-\eta_2)^2)}\left(f_1'(\eta_1)f_2(\eta_3)+f_1(\eta_1)f_2'(\eta_3)\right)\frac{g_3(\eta_3)}{|\eta_3|^\frac12}d\eta_1d\eta_3 \\+ & \iint e^{i\Phi}\frac{(\eta_1-\eta_2)^2}{2(1-i(\eta_1-\eta_2)^2)^2}f_1(\eta_1)f_2(\eta_2)\frac{g_3(\eta_3)}{|\eta_3|^\frac12}d\eta_1d\eta_3=: I_1 + I_2 + I_3.
\end{align*}
If $|\eta_3|\lesssim \jap{\eta_1}^{-1}$, a direct estimate yields
\begin{align*}
    |I_1+I_2+I_3|&\lesssim \iint \frac{1}{\jap{\eta_1-\eta_2}\jap{\eta_1}^{\k_1+\frac12}\jap{\eta_2}^{\k_2+\frac32}|\eta_3|^{\frac12}}d\eta_1d\eta_3 + \iint \frac{1}{\jap{\eta_1-\eta_2}^2\jap{\eta_1}^{\k_1+\frac12}\jap{\eta_2}^{\k_2+\frac12}|\eta_3|^\frac12}d\eta_1d\eta_3\\&\lesssim \jap{\eta}^{-\k_1-2}.
\end{align*}
If $|\eta_3|\gtrsim \jap{\eta_1}^{-1}$, we may integrate by parts as done in Step 1, with the worst-case scenario occuring when the derivative hits the singular weight $|\eta_3|^{-\frac12}$. Let us exemplify with one such term:
\begin{align*}
    &\left|\iint e^{i\Phi}\frac{(\eta_1-\eta_3)}{2(1-i(\eta_1-\eta_3)^2)}\frac{(\eta_1-\eta_2)}{2(1-i(\eta_1-\eta_2)^2)}(f_1'(\eta_1)f_2(\eta_2)+f_1(\eta_1)f_2'(\eta_2))\frac{g_3(\eta_3)}{|\eta_3|^{\frac32}}d\eta_1d\eta_3\right| \\&\lesssim \iint \frac{1}{\jap{\eta_1-\eta_2}\jap{\eta_1-\eta_3}\jap{\eta_1}^{\k_1+\frac12}\jap{\eta_2}^{\k_2+\frac32}|\eta_3|^{\frac32}}d\eta_1d\eta_3\lesssim \jap{\eta}^{-\k_1-2}.
\end{align*}
A similar estimate holds for the remaining terms, we leave it to the interested reader.
\end{proof}
\subsection{Bounds for $ \mathcal{I}_{\text{h}\times\text{h}}^{(+,-,+)}$}
We now focus on the operator
\begin{equation}\label{eq:I+-+}
    \mathcal{I}_{\text{h}\times\text{h}}^{(+,-,+)}[g_1,g_2,g_3](\eta):=\iint\limits_{\eta_1,\eta_3>0, \eta_2<0}\frac{1}{|\eta_1\eta_2\eta_3|^\frac{1}{2}}e^{i (\eta^2-\eta_1^2+\eta_2^2-\eta_3^2)} g_1(\eta_1){g_2(\eta_2)}g_3(\eta_3)\sum_{j=1}
^3\prod_{k\neq j}(1-\phi(\eta_j)) d\eta_1d\eta_2
\end{equation}
and tackle first the case $|\eta_j|>1$, $j=1,2,3$.
Writing once again $f_j=|\eta_j|^{-\frac{1}{2}}g_j \in Z^{\k_j+\frac{1}{2}}$. we must bound
    $$\iint\limits_{|\eta_j| \gg 1}e^{i (\eta^2-\eta_1^2+\eta_2^2-\eta_3^2)} f_1(\eta_1){f_2(\eta_2)}f_3(\eta_3) d\eta_1d\eta_2,\quad f_j\in Z^{\k+\frac12}.$$

\begin{remark}
    In the above integral, we remove the restrictions $\eta_1,\eta_3>0$, $\eta_2<0,$ to be able to reuse the estimates from this section in the \eqref{3NLS} framework.
\end{remark}
    
    We start by writing the phase function as 
    $$\Phi = \eta^2-\eta_1^2+\eta_2^2-\eta_3^2 = \frac{(\eta+\eta_2)^2-(\eta_3-\eta_1)^2}{2}.$$
    In particular, defining
    $$\begin{cases}
        r = \eta+\eta_2\\
        s = \eta_3-\eta_1
    \end{cases} \Leftrightarrow\begin{cases}
        \eta_1 = \eta-\frac{r+s}{2}\\
        \eta_2 = r-\eta\\
        \eta_3 = \eta-\frac{r-s}{2}
    \end{cases}$$
    we perform a change of variables to obtain:
    $$\iint\limits_{|\eta_j| \gg 1} e^{i\frac{r^2-s^2}{2}} f_1\left(\eta-\frac{r+s}{2}\right){f_2\left( r-\eta\right)}f_3\left(\eta-\frac{r-s}{2}\right)drds.$$
        
    Now, we consider the identities   
    $$e^{i\frac{r^2-s^2}{2}} = \frac{\partial_r \left( r \ e^{i\frac{r^2-s^2}{2}} \right)}{1+ir^2} \quad \text{and}\quad e^{i\frac{r^2-s^2}{2}} = \frac{\partial_s \left( s \ e^{i\frac{r^2-s^2}{2}} \right)}{1-is^2}$$
    and perform integration by parts, on both $r$ and $s$. However, we need to be careful as $f_j$ cannot be differentiated more than once. 

    Integrating by parts in $r$,
    $$-\frac{1}{2}\iint\limits_{|\eta_j|\gg1} r \ e^{i\frac{r^2-s^2}{2}} \partial_r \left[\frac{1}{1+ir^2}f_1\left(\eta-\frac{r+s}{2}\right){f_2\left( r-\eta\right)}f_3\left(\eta-\frac{r-s}{2}\right)\right]drds,$$
    which we write as the sum of
    \begin{equation}
            \frac{1}{4}\iint\limits_{|\eta_j|\gg1}  \frac{re^{i\frac{r^2-s^2}{2}}}{1+ir^2}{f_2\left( r-\eta\right)}\left[f_1'\left(\eta-\frac{r+s}{2}\right)f_3\left(\eta-\frac{r-s}{2}\right) +f_1\left(\eta-\frac{r+s}{2}\right)f_3'\left(\eta-\frac{r-s}{2}\right)\right]drds
        \label{eq: IBP f'h'}
    \end{equation}
    and
    \begin{equation}
            \frac{1}{2}\iint\limits_{|\eta_j|\gg1}  e^{i\frac{r^2-s^2}{2}}\frac{r}{1+ir^2} \left[\frac{2ir}{1+ir^2} {f_2\left( r-\eta\right)}- {f_2'\left( r-\eta\right)}\right]f_1\left(\eta-\frac{r+s}{2}\right)f_3\left(\eta-\frac{r-s}{2}\right)drds
        \label{eq: IBP g'}
    \end{equation}

    We focus on \eqref{eq: IBP f'h'}. Since both $f_1$ and $f_3$ depend on $s$, we cannot perform another integration by parts directly on \eqref{eq: IBP f'h'}. Therefore, we consider another change of variables:
    $$\begin{cases}
        u = \frac{r+s}{2}\\
        v = \frac{r-s}{2}
    \end{cases} \Leftrightarrow \begin{cases}
        r = u+v\\
        s = u-v
    \end{cases} \Leftrightarrow
    \begin{cases}
        \eta_1 = \eta-u\\
        \eta_2 = u+v-\eta\\
        \eta_3 = \eta-v
    \end{cases}.$$
    With this in mind, we can write \eqref{eq: IBP f'h'} as
    \begin{align*}
        &\frac{1}{2}\iint\limits_{|\eta_j| \gg 1}  e^{i2uv} \frac{u+v}{1+i(u+v)^2}{f_2\left( u+v-\eta\right)}\left[f_1'\left(\eta-u\right)f_3\left(\eta-v\right)+f_1\left(\eta-u\right)f_3'\left(\eta-v\right)\right]dudv
    \end{align*}

    It becomes clear that now we can integrate by parts again in a suitable variable. Indeed, since both integrals have the same structure, we will restrict our analysis to the first one. To avoid singularities, we will again use the identities
    $$e^{i2uv} = \frac{\partial_u \left(u e^{i2uv} \right)}{1+i2uv}, \quad e^{i2uv} = \frac{\partial_v \left(v e^{i2uv} \right)}{1+i2uv}.$$
    Integrating by parts on $v$, 
    \begin{align*}
        & -\frac{1}{2}\iint\limits_{|\eta_j| \gg 1}  \frac{ve^{i2uv}}{1+i2uv} \frac{u+v}{1+i(u+v)^2}f_1'\left(\eta-u\right)\left({f_2'\left( u+v-\eta\right)}f_3\left(\eta-v\right)-f_2\left( u+v-\eta\right)f_3'\left(\eta-v\right)\right)dudv\\
        & -\frac{1}{2}\iint\limits_{|\eta_j| \gg 1} e^{i2uv}\frac{v}{1+i2uv} \frac{1-i(u+v)^2}{(1+i(u+v)^2)^2}f_1'(\eta-u){f_2\left( u+v-\eta\right)}f_3\left(\eta-v\right)dudv\\
        & +\frac{1}{2}\iint\limits_{|\eta_j| \gg 1} e^{i2uv}\frac{i2uv}{(1+i2uv)^2} \frac{u+v}{1+i(u+v)^2}f_1'(\eta-u){f_2\left( u+v-\eta\right)}f_3\left(\eta-v\right)dudv
    \end{align*}
    or, in the original variables,
    \begin{align}
        & -\frac{1}{2}\iint\limits_{|\eta_j| \gg 1} e^{i\Phi} \frac{\eta-\eta_3}{1+i\Phi} \frac{\eta+\eta_2}{1+i(\eta+\eta_2)^2}f_1'(\eta_1)\left({f_2'(\eta_2)}f_3(\eta_3)-{f_2(\eta_2)}f_3'(\eta_3)\right)d\eta_1 d\eta_2 \label{eq: IBP 1 1}\\
        & -\frac{1}{2}\iint\limits_{|\eta_j| \gg 1} e^{i\Phi}\frac{\eta-\eta_3}{1+i\Phi} \frac{1-i(\eta+\eta_2)^2}{(1+i(\eta+\eta_2)^2)^2}f_1'(\eta_1){f_2(\eta_2)}f_3(\eta_3)d\eta_1d\eta_2\label{eq: IBP 1 2}\\
        & +\frac{1}{2}\iint\limits_{|\eta_j| \gg 1} e^{i\Phi} \frac{i\Phi}{(1+i\Phi)^2} \frac{\eta+\eta_2}{1+i(\eta+\eta_2)^2}f_1'(\eta_1){f_2(\eta_2)}f_3(\eta_3)d\eta_1d\eta_2 \label{eq: IBP 1 3}
    \end{align}

\begin{lemma}
    Assume $f_j\in Z^{\k_j+\frac{1}{2}}$ for some $\k_j\in \left(-\frac{1}{4},\frac{1}{4}\right)$. Then
    $$
    |\eqref{eq: IBP f'h'}|\lesssim \left|\eqref{eq: IBP 1 1}\right| + \left|\eqref{eq: IBP 1 2}\right| + \left|\eqref{eq: IBP 1 3}\right|  \lesssim \langle \eta \rangle^{-\frac{3}{2}-\k_1-\k_2-\k_3}.
    $$
\end{lemma}

\begin{proof}     

 Observe that
    $$\left|\frac{\eta-\eta_3}{1+i\Phi} \right|\lesssim \frac{|\eta-\eta_3|}{\langle (\eta-\eta_1)(\eta-\eta_3)\rangle}, \quad \left|\frac{\eta+\eta_2}{1+i(\eta+\eta_2)^2} \right|\lesssim \langle \eta+\eta_2\rangle^{-1}.$$
As such,
\begin{equation}
        \left|\eqref{eq: IBP 1 1}\right| \lesssim  \iint_{|\eta_j|\gg 1}\frac{|\eta-\eta_3|}{\langle (\eta-\eta_1)(\eta-\eta_3)\rangle \langle \eta+\eta_2\rangle}\jap{\eta_1}^{-\k-\frac32}\left(\jap{\eta_2}^{-\k-\frac32}\jap{\eta_3}^{-\k-\frac12} + \jap{\eta_2}^{-\k-\frac12}\jap{\eta_3}^{-\k-\frac32}\right)d\eta_1d\eta_2=:I.
\end{equation}

As the terms in \eqref{eq: IBP 1 2} and \eqref{eq: IBP 1 3} are handled similarly to \eqref{eq: IBP 1 1} (taking absolute values inside the integral), we prove only the estimate for \eqref{eq: IBP 1 1}.
   
    \noindent\textit{Step 1. $\eta>1$} We divide the proof into several regions.
          
    \noindent $\bullet$ \textbf{Case 1:} $\eta_1 \simeq -\eta_2 \simeq \eta_3 \simeq \eta$.

    Here it is convenient to use the variables $u$ and $v$. Indeed, we bound the integral by
    $$I\lesssim \langle\eta \rangle^{-\k_1-\k_2-\k_3-\frac{7}{2}}\iint\limits_{|u|,|v|\ll\eta} \frac{|v|}{\langle u v \rangle} \langle u+v \rangle^{-1} dudv\lesssim \langle\eta \rangle^{-\k_1-\k_2-\k_3-\frac{7}{2}}\log^2\jap{\eta}.$$

    
    \noindent $\bullet$ \textbf{Case 2:} $\eta_1 \simeq \eta$ and $\eta_3 \not \simeq \eta$

    In this region, we have that $\eta_3 \simeq-\eta_2$. In particular, we  can again consider the variables $u,v$ and write 
    \begin{align*}
       I\lesssim \langle \eta\rangle^{-\k_1-\frac{3}{2}}\iint \frac{|v|}{\langle uv\rangle} \langle v\rangle^{-1}  \langle \eta-v\rangle^{-\k_2-\k_3-2}dudv 
        \lesssim \langle \eta \rangle^{-\frac{3}{2}-\k_1-\k_2-\k_3}.
    \end{align*}

    \noindent $\bullet$ \textbf{Case 3:} $\eta_3 \simeq \eta$ and $\eta_1 \not \simeq \eta$

    This region can be treated as case 2. Indeed, we have that $\eta_1 \simeq-\eta_2$ and
    \begin{align*}
        I\lesssim &\ \langle \eta \rangle^{-\k_3-\frac{1}{2}}\iint\ \frac{|\eta-\eta_3|}{\langle (\eta-\eta_1)(\eta-\eta_3)\rangle}\langle \eta-\eta_1\rangle^{-1} \langle \eta_1\rangle^{-\k_1-\k_2-3}d\eta_1d\eta_3\\
        +&\  \langle \eta \rangle^{-\k_3-\frac{3}{2}}\iint\frac{|\eta-\eta_3|}{\langle (\eta-\eta_1)(\eta-\eta_3)\rangle}\langle \eta-\eta_1\rangle^{-1} \langle \eta_1\rangle^{-\k_1-\k_2-2}d\eta_1d\eta_3
        \\\lesssim &\  \langle \eta\rangle^{-\k_3-\frac{1}{2}}\iint \frac{|v|}{\langle uv\rangle}  \langle u\rangle^{-1}  \langle \eta-u\rangle^{-\k_1-\k_2-3}dudv+\langle \eta\rangle^{-\k_3-\frac{3}{2}}\iint \frac{|v|}{\langle uv\rangle}  \langle u\rangle^{-1}  \langle \eta-u\rangle^{-\k_1-\k_2-2}dudv \\
        \lesssim & \langle \eta \rangle^{-\frac{3}{2}-\k_1-\k_2-\k_3}.
    \end{align*}
    
    \noindent $\bullet$ \textbf{Case 4:} $\eta_1 \not \simeq \eta$ and $|\eta_1| \lesssim |\eta|$

    In this region, we have that $\langle (\eta-\eta_1)(\eta-\eta_3)\rangle \simeq \langle \eta (\eta-\eta_3) \rangle$. We will consider some additional subregions.
    
    \noindent $\circ$ \textbf{Case 4.1:} $|\eta_3| \lesssim |\eta|$ and $\eta_3 \not \simeq \eta$
    In this region, we have that 
    $$\frac{|\eta-\eta_3|}{\langle \eta (\eta-\eta_3)\rangle} \lesssim \frac{\eta}{\langle\eta\rangle^2} \lesssim \langle \eta \rangle^{-1}.$$
    In particular, we estimate the integral as
    $$I\lesssim \langle \eta \rangle^{-1} \iint \langle \eta+\eta_2\rangle^{-1} \langle \eta_1 \rangle^{-\k_1-\frac{3}{2}}(\langle \eta_2 \rangle^{-\k_2-\frac{3}{2}}\langle \eta_3\rangle^{-\k_3-\frac{1}{2}}+\langle \eta_2 \rangle^{-\k_2-\frac{1}{2}}\langle \eta_3\rangle^{-\k_3-\frac{3}{2}})d\eta_1d\eta_2\lesssim\langle \eta \rangle^{-\frac{3}{2}-\k_1-\k_2-\k_3}.$$

    \noindent $\circ$ \textbf{Case 4.2:} $|\eta_3| \gg |\eta|$
    
    In this region, we have that 
    $$\frac{|\eta-\eta_3|}{\langle \eta (\eta-\eta_3)\rangle} \lesssim \frac{|\eta_3|}{\langle\eta\rangle\langle \eta_3 \rangle}\lesssim\langle \eta\rangle^{-1}.$$
    Moreover, since $|\eta_1| \lesssim |\eta|$ and $\eta_1 \not \simeq \eta$, we have that $\eta_3 \simeq-\eta_2$ and $|\eta_2|,|\eta_3| \gg |\eta|$. In particular, $\langle \eta+\eta_2\rangle^{-1} \simeq \langle\eta_3\rangle^{-1}$ and we may bound the integral by
    $$ \langle \eta\rangle^{-1}\iint\limits_{\substack{|\eta_1|\lesssim |\eta| \\ |\eta_3| \gg |\eta|}} \langle \eta_1 \rangle^{-\k_1-\frac{3}{2}}\langle \eta_3 \rangle^{-\k_2-\k_3-3}d \eta_1d\eta_3.\lesssim  \langle \eta\rangle^{-\k_2-\k_3-3}\int\limits_{|\eta_1|\lesssim |\eta| } \langle \eta_1 \rangle^{-\k_1-\frac{3}{2}}d \eta_1 \lesssim \langle \eta \rangle^{-\frac{3}{2}-\k_1-\k_2-\k_3}.$$

    \noindent $\bullet$ \textbf{Case 5:} $|\eta_1| \gg |\eta|$

    We can adapt the proofs for case 4. 

    \noindent $\circ$ \textbf{Case 5.1:} $|\eta_3| \lesssim |\eta|$ and $\eta_3 \not \simeq \eta$

    In this case, we have 
    $$\frac{|\eta-\eta_3|}{\langle (\eta-\eta_1)(\eta-\eta_3)\rangle} \simeq\frac{|\eta|}{\langle \eta_1 \rangle \langle \eta \rangle} \lesssim \frac{|\eta|}{\langle \eta \rangle^2} \lesssim \langle \eta\rangle^{-1}$$
    and the proof follows as in case 4.1.

    \noindent $\circ$ \textbf{Case 5.2:} $|\eta_3| \gg |\eta|$
    
    In this region, 
    $$\frac{|\eta-\eta_3|}{\langle (\eta-\eta_1)(\eta-\eta_3)\rangle} \simeq\frac{|\eta_3|}{\langle \eta_1\rangle\langle \eta_3\rangle}\lesssim \langle \eta_1 \rangle^{-1}.$$
    We need to consider two cases: $\eta_1 \simeq \eta_3$ and $\eta_1 \not \simeq \eta_3$. If $\eta_1 \simeq \eta_3$, then $\eta_2 \simeq \eta-2\eta_1 \simeq -\eta_1$ and we estimate the integral as
    $$\iint\limits_{\substack{|\eta_1| \gg |\eta|\\\eta_1 \simeq -\eta_2}} \langle \eta_1\rangle^{-\k_1-\k_2-\k_3-\frac{11}{2}}d\eta_1d\eta_2 \lesssim \langle \eta\rangle^{-\frac{7}{2}-\k_1-\k_2-\k_3}.$$

    In the case $\eta_1 \not \simeq \eta_3$,  if $|\eta| \ll|\eta_1| \ll |\eta_3|$, then $\eta_2 \simeq-\eta_3$ and $\langle \eta+\eta_2\rangle^{-1}\simeq \langle \eta_3\rangle^{-1}$ which yields
    $$\iint\limits_{|\eta_1|,|\eta_3| \gg |\eta|} \langle \eta_1\rangle^{-\k_1-\frac{5}{2}} \langle \eta_3\rangle^{-\k_2-\k_3-3}d\eta_1d\eta_3 \lesssim \langle \eta\rangle^{-\frac{7}{2}-\k_1-\k_2-\k_3}.$$
     Otherwise, if $|\eta|\ll|\eta_3|\ll|\eta_1|$, then $\eta_2 \simeq -\eta_1$ and using that $\langle \eta_1 \rangle^{-1} \lesssim \langle \eta_3 \rangle^{-1}$, we obtain 
    $$\iint\limits_{|\eta_1|,|\eta_3| \gg |\eta|} \langle \eta_1\rangle^{-\k_1-\k_2-3} \langle \eta_3\rangle^{-\k_3-\frac{5}{2}}d\eta_1d\eta_3 \lesssim \langle \eta\rangle^{-\frac{7}{2}-\k_1-\k_2-\k_3}.$$

    \noindent \textit{Step 2.} $|\eta| \lesssim 1$. We bound the integral by
    $$\iint \frac{|\eta_3|}{\langle \eta_1 \eta_3\rangle} \langle \eta_1\rangle^{-\k_1-\frac{3}{2}} \langle \eta_2\rangle^{-1}(\langle \eta_2\rangle^{-\k_2-\frac{3}{2}}\langle \eta_3 \rangle^{-\k_3-\frac{1}{2}}+\langle \eta_2 \rangle^{-\k_2-\frac{1}{2}} \langle \eta_3 \rangle^{-\k_3-\frac{3}{2}})d\eta_1 d\eta_2.$$
    If $|\eta_1| \lesssim |\eta_3|$, then $$\frac{|\eta_3|}{\langle \eta_1 \eta_3\rangle} \lesssim \frac{|\eta_3|}{\langle \eta_3\rangle^{2}} \lesssim \langle \eta_3 \rangle^{-1}$$
    and,  if $|\eta_3| \lesssim |\eta_1|$, then $$\frac{|\eta_3|}{\langle \eta_1 \eta_3\rangle} \lesssim \frac{|\eta_1|}{\langle \eta_1\rangle^{2}} \lesssim \langle \eta_1 \rangle^{-1}\lesssim \langle \eta_3\rangle^{-1}.$$
    In either case, the integral is bounded by
    $$\iint \langle \eta_1\rangle^{-\k_1-\frac{3}{2}} \langle \eta_2\rangle^{-1}(\langle \eta_2\rangle^{-\k_2-\frac{3}{2}}\langle \eta_3 \rangle^{-\k_3-\frac{3}{2}}+\langle \eta_2 \rangle^{-\k_2-\frac{1}{2}} \langle \eta_3 \rangle^{-\k_3-\frac{5}{2}})d\eta_1 d\eta_2\lesssim 1.$$
\end{proof}

    We now turn our attention to the integral \eqref{eq: IBP g'}:
    $$-\frac{1}{2}\iint\limits_{|\eta_j| \gg 1}  e^{i\frac{r^2-s^2}{2}}\frac{r}{1+ir^2} \left[-\frac{2ir}{1+ir^2} {f_2\left( r-\eta\right)}+{f_2'\left( r-\eta\right)}\right]f_1\left(\eta-\frac{r+s}{2}\right)f_3\left(\eta-\frac{r-s}{2}\right)drds.$$
    As discussed before, we may integrate by parts on $s$:
    $$\frac{1}{2}\iint\limits_{|\eta_j| \gg 1}  e^{i\frac{r^2-s^2}{2}}\frac{r}{1+ir^2} \left[-\frac{2ir}{1+ir^2} {f_2\left( r-\eta\right)}+{f_2'\left( r-\eta\right)}\right]s \partial_s \left[\frac{1}{1-is^2}f_1\left(\eta-\frac{r+s}{2}\right)f_3\left(\eta-\frac{r-s}{2}\right)\right]drds,$$
    which can be written as
    \begin{align*}
        &-\frac{1}{4}\iint\limits_{|\eta_j| \gg 1} e^{i\frac{r^2-s^2}{2}} \frac{2ir^2}{(1+ir^2)^2}\frac{4is^2}{(1-is^2)^2} \ f_1\ {f_2}\ f_3\ drds\\
         &+\frac{1}{4}\iint\limits_{|\eta_j| \gg 1} e^{i\frac{r^2-s^2}{2}} \frac{2ir^2}{(1+ir^2)^2}\frac{s}{1-is^2} \ {f_2} \left[f_1'\ f_3-f_1\ f_3'\right]drds\\
         & +\frac{1}{4}\iint\limits_{|\eta_j| \gg 1} e^{i\frac{r^2-s^2}{2}} \frac{r}{1+ir^2}\frac{4is^2}{(1-is^2)^2} \ f_1\ {f_2'}\ f_3\ drds \\
         &+\frac{1}{4}\iint\limits_{|\eta_j| \gg 1} e^{i\frac{r^2-s^2}{2}} \frac{r}{1+ir^2}\frac{s}{1-is^2} \ {f_2'} \left[f_1'\ f_3-f_1\ f_3'\right]drds 
    \end{align*}
    or, in the original variables, 
    \begin{align}
         &-\frac{1}{2}\iint\limits_{|\eta_j| \gg 1} e^{i\Phi} \frac{2i(\eta+\eta_2)^2}{(1+i(\eta+\eta_2)^2)^2}\frac{4i(\eta_3-\eta_1)^2}{(1-i(\eta_3-\eta_1)^2)^2} \ f_1\ {f_2}\ f_3\ d\eta_1d\eta_2 \label{eq: ibp on s 1}\\
         &+\frac{1}{2}\iint\limits_{|\eta_j| \gg 1} e^{i\Phi} \frac{2i(\eta+\eta_2)^2}{(1+i(\eta+\eta_2)^2)^2}\frac{\eta_3-\eta_1}{1-i(\eta_3-\eta_1)^2} \ {f_2} \left[f_1'\ f_3-f_1\ f_3'\right]d\eta_1d\eta_2 \label{eq: ibp on s 2}\\
         & +\frac{1}{2}\iint\limits_{|\eta_j| \gg 1} e^{i\Phi} \frac{\eta+\eta_2}{1+i(\eta+\eta_2)^2}\frac{4i(\eta_3-\eta_1)^2}{(1-i(\eta_3-\eta_1)^2)^2} \ f_1\ {f_2'}\ f_3\ d\eta_1d\eta_2 \label{eq: ibp on s 3}\\
         &+\frac{1}{2}\iint\limits_{|\eta_j| \gg 1} e^{i\Phi} \frac{\eta+\eta_2}{1+i(\eta+\eta_2)^2}\frac{\eta_3-\eta_1}{1-i(\eta_3-\eta_1)^2} \ {f_2'} \left[f_1'\ f_3-f_1\ f_3'\right]d\eta_1d\eta_2 \label{eq: ibp on s 4}.
    \end{align}


    Notice that we can bound every integral by 
    $$\iint\limits \langle \eta+\eta_2\rangle^{-2+\k}\langle \eta_3-\eta_1\rangle^{-2+j} \langle \eta_1\rangle^{-\k_1-\frac{1}{2}-j} \langle\eta_2\rangle^{-\k_2-\frac{1}{2}-\k}\langle \eta_3\rangle^{-\k_3-\frac{1}{2}}d\eta_1d\eta_3$$
    where $\k=0$ for \eqref{eq: ibp on s 1} and  \eqref{eq: ibp on s 2} and $\k=1$ otherwise, and $j=0$ for \eqref{eq: ibp on s 1} and \eqref{eq: ibp on s 3} and $j=1$ otherwise. 

    \begin{lemma}[Estimates for \eqref{eq: IBP g'}]
        Let $\k_j\in \left(-\frac12, \frac12\right)$ and $\k,j \in \{0,1\}$. Then
        $$\iint\limits \langle \eta+\eta_2\rangle^{-2+\k}\langle \eta_3-\eta_1\rangle^{-2+j} \langle \eta_1\rangle^{-\k_1-\frac{1}{2}-j} \langle\eta_2\rangle^{-\k_2-\frac{1}{2}-\k}\langle \eta_3\rangle^{-\k_3-\frac{1}{2}}d\eta_1d\eta_2 \lesssim \langle \eta \rangle^{-\frac{3}{2}-\k_1-\k_2-\k_3}.$$
    \end{lemma}

    \begin{proof} We present the proof only for $\k, j =0$, 
    $$
    I= \iint\limits \langle \eta+\eta_2\rangle^{-2}\langle \eta_3-\eta_1\rangle^{-2} \langle \eta_1\rangle^{-\k_1-\frac{1}{2}} \langle\eta_2\rangle^{-\k_2-\frac{1}{2}}\langle \eta_3\rangle^{-\k_3-\frac{1}{2}}d\eta_1d\eta_2
    $$
    as the other cases are completely analogous.
  We start with $\eta\gg1$ and  consider several regions.
    
    \noindent $\bullet$ \textbf{Case 1: } $\eta_1 \simeq -\eta_2 \simeq \eta_3 \simeq \eta$. We have
    \begin{align*}
I&\lesssim\iint\limits_{\eta_1\simeq-\eta_2\simeq\eta_3 \simeq \eta} \langle \eta+\eta_2\rangle^{-2}\langle \eta_3-\eta_1\rangle^{-2} \langle\eta \rangle^{-\frac{3}{2}-\k_1-\k_2-\k_3}d\eta_1d\eta_2 \\
        & = \frac{1}{2} \langle \eta\rangle^{-\frac{3}{2}-\k_1-\k_2-\k_3}\int\limits_{\tilde{\eta_1}\simeq \eta} \langle \eta-\tilde \eta_1\rangle^{-2} d\tilde \eta_1\int\limits_{\tilde{\eta_2}\simeq -\eta} \langle \eta+\tilde \eta_2\rangle^{-2} d\tilde \eta_2\lesssim \langle \eta\rangle^{-\frac{3}{2}-\k_1-\k_2-\k_3}
    \end{align*}
    where we performed the change of variables $(\tilde \eta_1, \tilde \eta_2) = (2\eta_1+\eta_2,\eta_2)$.

    
    \noindent $\bullet$ \textbf{Case 2: } $\eta_1 \simeq \eta_3\not\simeq \eta $.

    In this region, considering the change of variables
    $$\begin{cases}
        \alpha = \eta_3+\eta_1\\
        \beta = \eta_3-\eta_1
    \end{cases} \Leftrightarrow\begin{cases}
        \eta_1 = \frac{\alpha-\beta}{2}\\
        \eta_3 = \frac{\alpha+\beta}{2}
    \end{cases}$$
    we write the integral as 
    $$I\sim \frac{1}{2}\iint\limits_{\beta\simeq0} \langle 2\eta-\alpha\rangle^{-2}\langle \beta\rangle^{-2} \left\langle \frac{\alpha+\beta}{2}\right\rangle^{-\k_1-\k_3-1} \langle\eta-\alpha\rangle^{-\k_2-\frac{1}{2}}d\alpha d\beta.$$
The integral in $\beta$ can be estimated as $$\int \langle \beta\rangle^{-2} \left\langle \frac{\alpha+\beta}{2}\right\rangle^{-\k_1-\k_3-1}d\beta \lesssim \langle \alpha \rangle^{-1-\k_1-\k_3}$$
   Since $\alpha\not\simeq 2\eta$, we consider two cases: if $|\alpha| \lesssim |2\eta|$, we bound the integral by
    \begin{align*}
       I \lesssim & \langle \eta\rangle^{-2} \int\limits_{|\alpha| \lesssim |2\eta|} \langle \alpha \rangle^{-1-\k_1-\k_3}   \langle\eta-\alpha\rangle^{-\k_2-\frac{1}{2}}d\alpha \\
        \lesssim & \langle \eta\rangle^{-\frac{5}{2}-\min\{, 1/2+\k_1+\k_3, \k_2\}}.
    \end{align*}
If $|\alpha| \gg |2\eta|$, then we estimate the integral as
    $$ I\lesssim\int\limits_{|\alpha| \gg |2\eta|} \langle \alpha \rangle^{-\k_1-\k_3-\k_2-\frac{7}{2}}   d\alpha \lesssim \langle \eta \rangle^{-\frac{5}{2}-\k_1-\k_2-\k_3}.$$

    \noindent $\bullet$ \textbf{Case 3: } $\eta_1 \not \simeq \eta_3$

    Using the same change of variables, we obtain
    $$I\lesssim \iint\limits_{|\beta|\gtrsim 1} \langle 2\eta-\alpha\rangle^{-2}\langle \beta\rangle^{-2} \left\langle \frac{\alpha-\beta}{2}\right\rangle^{-\k_1-\frac{1}{2}} \langle\eta-\alpha\rangle^{-\k_2-\frac{1}{2}}\left\langle \frac{\alpha+\beta}{2}\right\rangle^{-\k_3-\frac{1}{2}}d\alpha d\beta.$$
    Integrating in $\beta$,
    $$I\lesssim \int\limits_{|\beta|\gtrsim 1}\langle \beta\rangle^{-2} \left\langle \frac{\alpha-\beta}{2}\right\rangle^{-\k_1-\frac{1}{2}}\left\langle \frac{\alpha+\beta}{2}\right\rangle^{-\k_3-\frac{1}{2}}d\alpha \lesssim \langle \alpha\rangle^{-\frac{5}{2}-\min\{\k_3,\k_1\}}$$
   If $|\alpha| \lesssim |2\eta|$ (and $\alpha \not \simeq 2\eta$), we bound the integral by
    $$\langle \eta\rangle^{-2} \int\limits_{|\alpha| \lesssim |2\eta|} \langle \alpha\rangle^{-\frac{5}{2}-\min\{\k_3,\k_1\}+\varepsilon}   \langle\eta-\alpha\rangle^{-\k_2-\frac12}d\alpha \lesssim \langle \eta\rangle^{- \k_2-\frac52}.$$
    If $|\alpha| \gg |2\eta|$, then we estimate the integral as
    $$ I\lesssim \int\limits_{|\alpha| \gg |2\eta|} \langle \alpha\rangle^{-\frac{9}{2}-\min\{\k_3,\k_1\}-\k_2}   d\alpha \lesssim \langle \eta \rangle^{-\frac{7}{2}-\min\{\k_3,\k_1\}-\k_2}.$$

    For $|\eta| \lesssim 1$, we can consider again the change of variables $\alpha,\ \beta$ to write
    $$I \lesssim \iint \langle 2\eta-\alpha\rangle^{-2}\langle \beta\rangle^{-2} \left\langle \frac{\alpha+\beta}{2}\right\rangle^{-\k_1-\frac{1}{2}} \langle\eta-\alpha\rangle^{-\k_2-\frac{1}{2}}\left\langle \frac{\alpha-\beta}{2}\right\rangle^{-\k_3-\frac{1}{2}}d\alpha d\beta.$$
    Particularly, since $\eta$ is small, we can bound the integral by
    $$\iint \langle \beta\rangle^{-2} \left\langle \frac{\alpha+\beta}{2}\right\rangle^{-\k_1-\frac{1}{2}} \langle \alpha\rangle^{-\frac{5}{2}-\k_2}\left\langle \frac{\alpha-\beta}{2}\right\rangle^{-\k_3-\frac{1}{2}}d\alpha d\beta \lesssim \int \langle\alpha\rangle^{-\frac72-\min\{\k_1,\k_3\}-\k_2}d\alpha\lesssim 1.$$
    \end{proof}

\begin{lemma}
    Let $\k_j \in \left( -\frac{1}{4},\frac{1}{4}\right)$ and $g_j \in Z^{\k_j}$, for $j=1,2,3$. Then
    $$\left|\mathcal I_{\text{h}\times \text{h}}^{(+,-,+)}[g_1,g_2,g_3](\eta)\right| \lesssim \jap{\eta}^{-\frac32 - \k_1-\k_2-\k_3}\prod_{j=1}^3\|g_j\|_{Z^{\kappa_j}}, $$
\end{lemma}
\begin{proof}

The multilinear estimate for $\mathcal{I}_{\text{h}\times\text{h}}^{(+,-,+)}$ when all frequencies are large is a consequence of the above discussion.   
It remains to derive the multilinear bounds for \eqref{eq:I+-+} when $|\eta_j|<1$ for some $j=1, 2, 3.$ This must be done carefully as the integration by parts may introduce uncontrollable singularities near $\eta_j=0$.

\medskip

\noindent\textit{Step 1. $|\eta_1|\lesssim 1 \ll |\eta_2|,|\eta_3|$.}

We first consider the proof for $\eta \gg 1$. In this region, we leave $\eta_3$ as the dependent variable and integrate by parts \eqref{eq:I+-+} in $\eta_2$ to obtain
$$\iint\limits_{|\eta_1| \lesssim 1 \ll |\eta_2|,|\eta_3|} \frac{e^{i\Phi}}{2i(\eta-\eta_1)} \frac{g_1(\eta_1)}{|\eta_1|^{\frac{1}{2}}}\left[\left(\frac{g_2(\eta_2)}{|\eta_2|^{\frac{1}{2}}} \right)'\frac{g_3(\eta_3)}{|\eta_3|^{\frac{1}{2}}}-\frac{g_2(\eta_2)}{|\eta_2|^{\frac{1}{2}}}\left(\frac{g_3(\eta_3)}{|\eta_3|^{\frac{1}{2}}} \right)' \right]d\eta_1d\eta_2 =: I_1 + I_2.$$
If $|\eta_1| \leq \langle \eta \rangle^{-2}$, since $|\eta_2|, |\eta_3|\gtrsim |\eta|$,  we estimate $I_1+I_2$ directly by
\begin{align*}
    &\ \langle \eta \rangle^{-1}\iint\limits_{|\eta_1| \leq \langle \eta \rangle^{-2}} |\eta_1|^{-\frac{1}{2}} \left(\langle \eta_2 \rangle^{-\k_2-\frac{3}{2}}\langle \eta_3 \rangle^{-\k_3-\frac{1}{2}} +\langle \eta_2 \rangle^{-\k_2-\frac{1}{2}}\langle \eta_3 \rangle^{-\k_3-\frac{1}{2}}\right)d\eta_2 d\eta_1\\
    \lesssim &\ \langle\eta\rangle^{-\frac{3}{2}-\min\{\k_2,\k_3\}} \int\limits_{|\eta_1| \leq \langle \eta \rangle^{-2}} |\eta_1|^{-\frac{1}{2}} d\eta_1 
    \lesssim \langle\eta\rangle^{-\frac{5}{2}-\min\{\k_2,\k_3\}} 
\end{align*}
where we used the fact that $|\eta-\eta_1| \simeq |\eta|$.

If $|\eta_1| \geq \langle \eta \rangle^{-2}$, we may integrate by parts $I_1$ on $\eta_1$ to find
\begin{align*}
    -\iint\limits_{\langle \eta \rangle^{-2}\leq |\eta_1|\lesssim 1} & \frac{e^{i\Phi}}{4(\eta-\eta_1)(\eta_3-\eta_1)}\\
    & \left(\frac{g_2(\eta_2)}{|\eta_2|^{\frac{1}{2}}}\right)'\left[\frac{1}{\eta-\eta_1}\frac{g_1(\eta_1)}{|\eta_1|^{\frac{1}{2}}}\frac{g_3(\eta_3)}{|\eta_3|^{\frac{1}{2}}}+\left(\frac{g_1(\eta_1)}{|\eta_1|^{\frac{1}{2}}} \right)'\frac{g_3(\eta_3)}{|\eta_3|^{\frac{1}{2}}}-\frac{g_1(\eta_1)}{|\eta_1|^{\frac{1}{2}}}\left(\frac{g_3(\eta_3)}{|\eta_3|^{\frac{1}{2}}} \right)' \right]d\eta_1d\eta_2
\end{align*}
which we bound as 
\begin{align*}
    &\ \langle \eta\rangle^{-1}\iint\limits_{\langle \eta \rangle^{-2}\leq |\eta_1|\lesssim 1}  \langle \eta_1\rangle^{-\k_1-\frac{1}{2}}\langle\eta_2\rangle^{-\k_2-\frac{3}{2}}\langle \eta_3 \rangle^{-\k_3-\frac{3}{2}}\left(\langle \eta \rangle^{-1}+\langle \eta_1 \rangle^{-1}+\langle \eta_3 \rangle^{-1} \right)d\eta_1d\eta_2\\\lesssim &\ \langle \eta\rangle^{-\frac{5}{2}-\min\{\k_2,\k_3\}} \int\limits_{\langle \eta \rangle^{-2}\leq |\eta_1|\lesssim 1}  \langle \eta_1\rangle^{-\k_1-\frac{1}{2}}d\eta_1\lesssim \langle \eta\rangle^{-\frac{5}{2}-\min\{\k_2,\k_3\}}.
\end{align*}
For $I_2$, integrating by parts in $\eta_1$, with $\eta_2$ dependent,
\begin{align*}
    I_2 \sim \iint\limits_{\langle \eta \rangle^{-2}\leq |\eta_1|\lesssim 1} & \frac{e^{i\Phi}}{4(\eta-\eta_1)(\eta-\eta_3)}\\
    & \left(\frac{g_3(\eta_3)}{|\eta_3|^{\frac{1}{2}}}\right)'\left[\frac{1}{\eta-\eta_1}\frac{g_1(\eta_1)}{|\eta_1|^{\frac{1}{2}}}\frac{g_2(\eta_2)}{|\eta_2|^{\frac{1}{2}}}+\left(\frac{g_1(\eta_1)}{|\eta_1|^{\frac{1}{2}}} \right)'\frac{g_2(\eta_2)}{|\eta_2|^{\frac{1}{2}}}-\frac{g_1(\eta_1)}{|\eta_1|^{\frac{1}{2}}}\left(\frac{g_2(\eta_2)}{|\eta_2|^{\frac{1}{2}}} \right)' \right]d\eta_1d\eta_3,
\end{align*}
which can be bounded  by
$$\langle \eta\rangle^{-1}\iint\limits_{\langle \eta \rangle^{-2}\leq |\eta_1|\lesssim 1}  \langle \eta_1\rangle^{-\k_1-\frac{1}{2}}\langle\eta_2\rangle^{-\k_2-\frac{3}{2}}\langle \eta_3 \rangle^{-\k_3-\frac{3}{2}}\left(\langle \eta \rangle^{-1}+\langle \eta_1 \rangle^{-1}+\langle \eta_2 \rangle^{-1} \right)d\eta_1d\eta_3\lesssim \jap{\eta}^{-\frac32-\k_1-\k_2-\k_3}.$$
Now we focus on the case $\eta\lesssim 1$. If $|\eta-\eta_1|\gtrsim \jap{\eta_2}^{-2/3}$, we bound $I_1+I_2$ as
$$
\iint_{|\eta-\eta_1|\gtrsim \jap{\eta_2}^{-\frac23}} \frac{1}{|\eta-\eta_1||\eta_1|^{\frac12}}\jap{\eta_2}^{-2-\k_2-\k_3}d\eta_1d\eta_2 \lesssim \int \jap{\eta_2}^{-\frac43-\k_2-\k_3}d\eta_2\lesssim 1.
$$
If  $|\eta-\eta_1|\ll\jap{\eta_2}^{-2/3}$, we estimate directly:
$$
\iint_{|\eta-\eta_1|\ll \jap{\eta_2}^{-\frac23}} \frac{1}{|\eta_1|^{\frac12}}\jap{\eta_2}^{-1-\k_2-\k_3}d\eta_1d\eta_2 \lesssim \int \jap{\eta_2}^{-\frac43-\k_2-\k_3}d\eta_2\lesssim 1.
$$



\medskip
\noindent\textit{Step 2. $|\eta_2|\lesssim 1 \ll |\eta_1|,|\eta_3|$.}

We start by considering the case $\eta \gg 1$ and split the analysis in two regions:

\noindent $\bullet$ \textbf{Case 1: }$|\eta_1-\eta_3| \gtrsim |\eta|$. In this region, we may integrate by parts \eqref{eq:I+-+}  in $\eta_1$  to find
\begin{equation}\label{eq:hlh}
    \iint\limits_{|\eta_2|\lesssim 1 \ll |\eta_1|,|\eta_3|} \frac{e^{i\Phi}}{2i(\eta_3-\eta_1)} \frac{g_2(\eta_2)}{\eta_2^\frac{1}{2}} \left[\left(\frac{g_1(\eta_1)}{\eta_1^\frac{1}{2}} \right)'\frac{g_3(\eta_3)}{\eta_3^\frac{1}{2}}-\frac{g_1(\eta_1)}{\eta_1^\frac{1}{2}}\left(\frac{g_3(\eta_3)}{\eta_3^\frac{1}{2}}  \right)'   \right]d\eta_1d\eta_2.
\end{equation}
 If $|\eta_2| \leq \langle \eta \rangle^{-2}$, we estimate the integral crudely by
\begin{align*}
    &\ \langle \eta \rangle^{-1} \iint\limits_{|\eta_2|\leq \langle \eta \rangle^{-2}} |\eta_2|^{-\frac{1}{2}} ( \langle \eta_1\rangle^{-\k_1-\frac{3}{2}}\langle \eta_3\rangle^{-\k_3-\frac{1}{2}}+\langle \eta_1\rangle^{-\k_1-\frac{1}{2}}\langle \eta_3\rangle^{-\k_3-\frac{3}{2}})d\eta_1d\eta_2\\\lesssim &\ \langle \eta\rangle^{-\frac{3}{2}-\min\{\k_1,\k_3\}}\int\limits_{|\eta_2| \leq \langle \eta \rangle^{-2}}|\eta_2|^{-\frac{1}{2}}d\eta_2 \lesssim \langle \eta \rangle^{-\frac{5}{2}-\min\{\k_1,\k_3\}}.
\end{align*}
If $|\eta_2| \geq \langle \eta \rangle^{-2}$, we may integrate by parts in $\eta_2$. We will only show the argument for the first term, since the second is exactly the same. Integrating by parts in $\eta_2$,
\begin{align*}
    -\iint\limits_{\langle \eta \rangle^{-2}\leq |\eta_2|\lesssim 1} & \frac{e^{i\Phi}}{4(\eta-\eta_1)(\eta_3-\eta_1)}\\
    & \left(\frac{g_1(\eta_1)}{\eta_1^{\frac{1}{2}}}\right)'\left[\frac{1}{\eta_3-\eta_1}\frac{g_2(\eta_2)}{\eta_2^{\frac{1}{2}}}\frac{g_3(\eta_3)}{\eta_3^{\frac{1}{2}}}+\left(\frac{g_2(\eta_2)}{\eta_2^{\frac{1}{2}}} \right)'\frac{g_3(\eta_3)}{\eta_3^{\frac{1}{2}}}-\frac{g_2(\eta_2)}{\eta_2^{\frac{1}{2}}}\left(\frac{g_3(\eta_3)}{\eta_3^{\frac{1}{2}}} \right)' \right]d\eta_1d\eta_2.
\end{align*}
Since $|\eta-\eta_1| \simeq |\eta_2+\eta_3| \simeq |\eta_3|$, we estimate this integral as
\begin{align*}
    &\ \langle \eta \rangle^{-1}\iint\limits_{\langle \eta \rangle^{-2}\leq |\eta_2|\lesssim 1} \langle \eta_1 \rangle^{-\k_1-\frac{3}{2}}\langle \eta_2 \rangle^{-\k_2-\frac{1}{2}}\langle \eta_3 \rangle^{-\k_3-\frac{3}{2}}(\langle \eta \rangle^{-1}+\langle \eta_2\rangle^{-1}+\langle \eta_3\rangle^{-1})d\eta_1d\eta_3\\\lesssim &\ \langle \eta \rangle^{-\frac{5}{2}-\min\{\k_1,\k_3\}} \int \limits_{\langle \eta\rangle^{-2}\leq |\eta_2| \lesssim 1} \langle \eta_2 \rangle^{-\k_2-\frac{1}{2}}d\eta_2 \lesssim \langle \eta \rangle^{-\frac{5}{2}-\min\{\k_1,\k_3\}}.
\end{align*}
For $\eta \lesssim 1$,  we necessarily have $\eta_1 \simeq-\eta_3$. Thus, we may integrate by parts on $\eta_1$ and use the fact that $|\eta_3-\eta_1| \simeq |\eta_1|$. The resulting integral is then controlled by
$$\iint\limits_{|\eta_2|\lesssim 1} |\eta_2|^{-\frac{1}{2}} \langle \eta_1\rangle^{-\k_1-\k_3-3}d\eta_1d\eta_2\lesssim 1.$$

\noindent $\bullet$ \textbf{Case 2: }$|\eta| \gg |\eta_1-\eta_3|$.
Observe that, in this region, $\eta_1\simeq \eta_3\simeq \frac{\eta}{2}$. If $|\eta_1-\eta_3|\lesssim |\eta|^{-1}$, we estimate directly \eqref{eq:I+-+}  by
$$
\iint_{|\eta_1-\eta_3|\lesssim |\eta|^{-1}} \frac{1}{|\eta_2|^{\frac12}}\jap{\eta_1}^{-\k_1-\frac12}\jap{\eta_3}^{-\k_3-\frac12}d\eta_1 d\eta_2 \lesssim \jap{\eta}^{-\k_1-\k_3-2}
$$
If $|\eta_1-\eta_3|\gg |\eta|^{-1}$, we integrate by parts in $\eta_1$ (as in Region A) to find \eqref{eq:hlh}. The resulting integral is bounded by
$$
\iint_{|\eta_1-\eta_3|\gg |\eta|^{-1}} \frac{1}{|\eta_1-\eta_3||\eta_2|^{\frac12}}\jap{\eta_1}^{-\k_1-\k_3-2}d\eta_1d\eta_2 \lesssim \jap{\eta}^{-\k_1-\k_3-2}\log \eta.
$$

\end{proof}

\subsection{Bounds for $\mathcal{I}_{\text{h}\times \text{l}\times\text{l}}$} We now consider the case where $|\eta_1|>1$ and $|\eta_2|, |\eta_3|<1$, that is,
\begin{equation}\label{eq:Ihll}
    \mathcal{I}_{\text{h}\times \text{l}\times\text{l}}[g_1,g_2,g_3](\eta)=\iint e^{i\Phi}\frac{g_1(\eta_1)g_2(\eta_2)g_3(\eta_3)}{|\eta_1\eta_2\eta_3|^\frac12}\phi(\eta_2)\phi(\eta_3)(1-\phi(\eta_1))d\eta_1d\eta_2.
\end{equation}
For $\eta_2,\eta_3>0$, consider the change of variables
$$
(\eta_2,\eta_3)= (y_2^2,y_3^2),\quad y_2,y_3>0.
$$
The resulting integral can be written as
\begin{align}
        &\ e^{i\eta^2}\iint_{(\R^+)^2} e^{-i(\eta-y_2^2-y_2^3)^2}\frac{e^{-iy_2^4}g_2(y_2^2)e^{iy_3^4}g_3(y_3^2)g_1(\eta-y_2^2-y_3^2)}{|\eta-y_2^2-y_3^2|^\frac12}\phi(y_2^2)\phi(y_3^2)(1-\phi(\eta-y_2^2-y_3^2))dy_2dy_3\\=&\ \iint_{(\R^+)^2}e^{2i\eta(y_2^2+y_3^2)}\frac{e^{-i(y_2^2+y_3^2)^2}e^{-iy_2^4}g_2(y_2^2)e^{iy_3^4}g_3(y_3^2)g_1(\eta-y_1^2-y_2^2)}{|\eta-y_2^2-y_3^2|^\frac12}\phi(y_2^2)\phi(y_3^2)(1-\phi(\eta-y_2^2-y_3^2))dy_2dy_3\\&\qquad =: \iint_{(\R^+)^2} e^{2i\eta(y_2^2+y_3^2)}G_{+,+}(y_2,y_3;\eta)dy_2dy_3 \label{lxlxh}
\end{align}
where $G_{+,+}(\cdot, \eta)$ is compactly supported in the unit ball and satisfies
\begin{equation}\label{eq:prop_G}
    \|\partial_{y_2}G_{+,+}(\eta)\|_{L^\infty_{y_2,y_3}} + \|\partial_{y_3}G_{+,+}(\eta)\|_{L^\infty_{y_2,y_3}} + \|\partial_{y_2}\partial_{y_3}G_{+,+}(\eta)\|_{L^\infty_{y_2,y_3}}\lesssim \eta^{-\frac12},\quad G_{+,+}(0,0;\eta)=\frac{g_2(0)g_3(0)g_1(\eta)}{\eta^{\frac12}}.
\end{equation}

Similarly, the contribution of the remaining quadrants in the $(\eta_2,\eta_3)$ plane can be written as
$$
\eta_2<0,\eta_3>0:\qquad\iint_{(\R^+)^2} e^{2i\eta(y_2^2-y_3^2)}G_{+,-}(y_2,y_3;\eta)dy_2dy_3, 
$$
$$
\eta_2>0,\eta_3<0:\qquad\iint_{(\R^+)^2} e^{2i\eta(-y_2^2+y_3^2)}G_{-,+}(y_2,y_3;\eta)dy_2dy_3,
$$
$$
\eta_2<0,\eta_3<0:\qquad\iint_{(\R^+)^2} e^{2i\eta(-y_2^2-y_3^2)}G_{-,-}(y_2,y_3;\eta)dy_2dy_3,
$$
where $G_{-,+}$, $G_{+,-}$ and $G_{-,-}$ also satisfy \eqref{eq:prop_G}. To proceed, we need the following technical lemma.

\begin{lemma}\label{lema integral y}
Let $F : \R \times \R \to \R$ be such that $\supp\ F  \subset [0,1] \times [0,1]$ and
\begin{equation}
\|\partial_{y_2}F\|_{L^\infty_{y_2,y_3}} 
+ \|\partial_{y_3}F\|_{L^\infty_{y_2,y_3}} 
+ \|\partial_{y_2}\partial_{y_3}F\|_{L^\infty_{y_2,y_3}} 
\leq C.
\end{equation}
Then, for every $\eta>0$, $\lambda_2,\lambda_3=\pm 1$, we have
\begin{equation}
\left|
\iint_{(\R^+)^2} e^{2i\eta(\lambda_2y_2^2+\lambda_3y_3^2)} F(y_2,y_3)\,dy_2dy_3
- F(0,0)\, \frac{\pi}{8\eta}\, e^{\frac{i\pi}{4}(\lambda_2+\lambda_3)}
\right|
\lesssim C
\left(\frac{\log \eta}{\eta}\right)^2.
\end{equation}
\end{lemma}

\begin{proof}
We prove only the case $\lambda_2=\lambda_3=1$, as the others are similar. First, since  
$$
\iint_{(\R^+)^2} e^{2i\eta(y_2^2+y_3^2)}\,dy_2dy_3 
=  \frac{\pi}{8\eta}\, e^{\frac{i\pi}{2} },
$$
we have
\begin{align*}
   &\iint_{(\R^+)^2} e^{2i\eta(y_2^2+y_3^2)}F(y_2,y_3)\,dy_2dy_3
   - F(0,0)\, \frac{\pi}{8\eta}\, e^{\frac{i\pi}{2} }\\
   &= \iint_{(\R^+)^2} e^{2i\eta(y_2^2+y_3^2)}[F(y_2,y_3)- F(0,0)]\,dy_2dy_3.
\end{align*}

Since
\begin{align*}
F(y_2, y_3) - F(0,0)
&=\int_{0}^{y_2}\!\!\int_{0}^{y_3}
\partial_{y_2}\partial_{y_3}F(\nu_2, \nu_3)\, d\nu_3\, d\nu_2 
+ \int_{0}^{y_2} \partial_{y_2}F(\nu_2, 0)\, d\nu_2
+ \int_{0}^{y_3}\partial_{y_3}F(0, \nu_3)\, d\nu_3,
\end{align*}
we obtain
$$
\iint_{(\R^+)^2} e^{2i\eta(y_2^2+y_3^2)}[F(y_2,y_3)- F(0,0)]\,dy_2dy_3
= I_1 + I_2 + I_3,
$$
where
$$
I_1=\iint_{(\R^+)^2} e^{2i\eta(y_2^2+y_3^2)}
\int_{0}^{y_2}\!\!\int_{0}^{y_3}
\partial_{y_2}\partial_{y_3}F(\nu_2, \nu_3)\, d\nu_3\, d\nu_2\,dy_2dy_3,
$$
$$
I_2= \iint_{(\R^+)^2} e^{2i\eta(y_2^2+y_3^2)}
\int_{0}^{y_2} \partial_{y_2}F(\nu_2, 0)\, d\nu_2\,dy_2dy_3,
$$
and
$$
I_3= \iint_{(\R^+)^2} e^{2i\eta(y_2^2+y_3^2)} 
\int_{0}^{y_3} \partial_{y_3}F(0, \nu_3)\, d\nu_3\,dy_2dy_3.
$$

Let us start by estimating $I_1$.  
After a change of variables and using Fubini’s theorem, we obtain
\begin{align*}
|I_1|
&= \left|\iint_{(\R^+)^2} 
\left( \int_{\nu_2}^{+\infty}\!\!\int_{\nu_3}^{+\infty}
e^{2i\eta(y_2^2+y_3^2)} \, dy_2dy_3 \right)
\partial_{y_2}\partial_{y_3}F(\nu_2, \nu_3)\, d\nu_2\, d\nu_3 \right| \\
&\leq  
\left( \int_0^{+\infty}\left|\ \!\!\int_{\nu_2}^{+\infty}e^{2i\eta y_2^2} dy_2\right|d\nu_2 \right) 
\left( \int_0^{+\infty}\ \!\!\left| \int_{\nu_3}^{+\infty}e^{2i\eta y_3^2} dy_3\right|d\nu_3 \right)
\|\partial_{y_2}\partial_{y_3}F\|_{L^\infty_{y_2,y_3}} \\
&\leq C \left( \int_0^{+\infty}\left|\ \!\!\int_{\nu_2}^{+\infty}e^{2i\eta y_2^2} dy_2\right|d\nu_2 \right)
\left( \int_0^{+\infty}\!\ \left| \int_{\nu_3}^{+\infty}e^{2i\eta y_3^2} dy_3\right|d\nu_3 \right).
\end{align*}
By Lemma~3 in \cite{CCV20}, it follows that
$$
\left|\int_0^{+\infty}\left|\!\int_{\nu}^{+\infty}e^{2i\eta y^2} \, dy\,\right|d\nu\right|
\lesssim \frac{\log \eta}{\eta}.
$$
Hence,
$$
|I_1| \lesssim C \left( \frac{\log \eta}{\eta}\right)^2.
$$

A similar argument applies to $I_2$ and $I_3$. Indeed,
\begin{align*}
I_2 
&= \iint_{(\R^+)^2} e^{2i\eta(y_2^2+y_3^2)}
\int_{0}^{y_2} \partial_{y_2}F(\nu_1, 0)\, d\nu_1\,dy_2dy_3  \\
&=\left(\int_0^{+\infty}  e^{2i\eta y_3^2} \, dy_3\right)
\left( \int_0^{+\infty}  e^{2i\eta y_2^2}
\int_0^{y_2}\partial_{y_2}F(\nu_1, 0)\, d\nu_1\, dy_2 \right)\\
&= \left(\int_0^{+\infty}  e^{2i\eta y_3^2} \, dy_3\right)
\left( \int_0^{+\infty}  \left[\int_{\nu_1}^{+\infty} e^{2i\eta y_2^2}\, dy_2\right] 
\partial_{y_2}F(\nu_1, 0)\, d\nu_1\,\right)
\end{align*}
Using the same reasoning as before, we conclude that
$$
|I_2| \lesssim C\left( \frac{\log \eta}{\eta}\right)^2,
$$
and similarly,
$$
|I_3| \lesssim C \left( \frac{\log \eta}{\eta}\right)^2.
$$
This completes the proof.
\end{proof}
Racelling that, for $\eta \gg 1$, we have
$$
G_{\pm,\pm}(0,0;\eta) = \frac{g_2(0)\,{g_3(0)}\,g_1(\eta)}{\eta^{1/2}},
$$
applying Lemma \ref{lema integral y} to \eqref{eq:Ihll} yields
\begin{equation}
\left|
\mathcal I_{\text{l}\times \text{l}\times \text{h}}[g_1,g_2,g_3](\eta)
- \frac{\pi}{4\eta^{3/2}}\,g_2(0)\,{g_3(0)}\,g_1(\eta)
\right|
\lesssim 
\langle \eta \rangle^{-5/2}\,(\log\eta)^2.
\end{equation}
as stated in Theorem \ref{thm:multi_mbo_1}
\subsection{Bounds for $\mathcal{I}\label{sec_Illl}_{\text{l}\times\text{l}\times\text{l}}$.} Finally, we consider the case where all frequencies are small,
$$
\mathcal{I}_{\text{l}\times\text{l}\times\text{l}}[g_1,g_2,g_3](\eta)=\iint e^{i\Phi}\frac{g_1(\eta_1)g_2(\eta_2)g_3(\eta_3)}{|\eta_1\eta_2\eta_3|^\frac12}\phi(\eta_1)\phi(\eta_2)\phi(\eta_3)d\eta_1d\eta_2.
$$
In this region, as the domain of integration is bounded, a rough estimate is sufficient:
$$
|\mathcal{I}_{\text{l}\times\text{l}\times\text{l}}[g_1,g_2,g_3](\eta)|\lesssim \iint_{[-1,1]^2} \frac{d\eta_1d\eta_2}{|\eta_1\eta_2\eta_3|^{\frac12}}\lesssim 1.
$$
Putting together the conclusions of Sections \ref{sec:Ihh}-\ref{sec_Illl}, the proof of Theorem \ref{thm:multi_mbo_1} is finished.

\section{Multilinear estimates for mBO - Proof of Theorem \ref{thm:multi_mbo_2}}\label{sec:est_mbo2}

We devote this section to finding suitable estimates for
$$\iint e^{i\Phi}\frac{S_{A,a,B}(\eta_1)S_{A,a,B}(\eta_2)S_{A,a,B}(\eta_3)}{(\eta_1\eta_2\eta_3)^\frac{1}{2}}d\eta_1 d\eta_2.$$
where $S_{A,a,B}$ is defined in \eqref{eq:defi_Saab} as 
$$
 S_{A,a,B}(\eta)=\left(Ae^{ia\log \eta} + B \frac{e^{2i\eta^2/3+3ia\log\eta}}{\eta^2}\right)\chi(\eta),\quad \eta>0,\qquad S_{A,a,B}(-\eta)=\overline{S_{A,a,B}(\eta).}
$$
We decompose the ansatz as
$$
S_{A,a,B}=S_{A,a,0}   + S_{0,a,B}
$$
and focus first on the cubic terms in $S_{A,a,0}$.

\begin{lemma}\label{lem:IA}
    For $\eta>1$,
    $$
    \left|\mathcal{I}[S_{A,a,0}] - \frac{3|A|^2A\pi e^{ia\log \eta}}{\eta^\frac32} - e^{2i\eta^2/3+3ia\log\eta}\frac{i\pi A^3\sqrt{3}e^{-3ia\log 3}}{\eta^\frac32}\right|\lesssim \jap{\eta}^{-\frac72}.
    $$
 
\end{lemma}
\begin{proof}
 Let us consider the change of variables
$$\begin{cases}
    \eta_1 = \eta p_1\\
    \eta_2 = \eta p_2\\
    \eta_3 = \eta p_3
\end{cases}.$$
Then, we may write the integral as
$$I=\eta^\frac{1}{2}\iint e^{i\eta^2\Psi}\frac{S_{A,a,0}(\eta p_1)S_{A,a,0}(\eta p_2)S_{A,a,0}(\eta p_3)}{(p_1p_2p_3)^\frac{1}{2}}dp_1 dp_2,$$
where $\Psi = 1-|p_1|p_1-|p_2|p_2-|p_3|p_3.$ This resonance function has four stationary points, namely
$$
(1,1,-1), (1,-1,1), (-1,1,1), \left(\frac{1}{3}, \frac13, \frac13\right)
$$
For $j=1,\dots, 4$, let $\phi_j \in C^\infty_c(\mathbb R^2)$ be a cut-off function around a small neighborhood of each stationary point. Write 
$$
I_j=\eta^\frac{1}{2}\iint e^{i\eta^2\Psi}\frac{S_{A,a,0}(\eta p_1)S_{A,a,0}(\eta p_2)S_{A,a,0}(\eta p_3)}{(p_1p_2p_3)^\frac{1}{2}}\phi_j(p_1,p_2)dp_1 dp_2,\quad j=1,\dots, 4,
$$
and
$$
I_5:=I-\sum_{j=1}^4 I_j.
$$
We now estimate each term separately. For $j=1$, we rewrite 
$$
I_1=\eta^\frac{1}{2}|A|^2Ae^{ia\log \eta}\iint e^{i\eta^2\Psi}\frac{e^{ia\log\left|\frac{p_1p_2}{p_3}\right|}}{(p_1p_2p_3)^\frac{1}{2}}\phi_j(p_1,p_2)dp_1 dp_2.
$$
Since we are near a stationary point of $\Psi$ with $\det D^2\Psi=-4$, using the stationary phase lemma (Theorem \ref{fase estacionaria}), this term may be written as
$$I_1=|A|^2A\frac{\pi e^{ia\log \eta}}{\eta^\frac{3}{2}}+O(\eta^{-\frac{7}{2}}).$$
The contribution of $I_j$ for $j=2,3$ is exactly the same. For $j=4$, 
$$
I_4 = \eta^{\frac12}A^3e^{i3a\log\eta}\iint e^{i\eta^2\Psi}\frac{e^{ia\log|p_1p_2p_3|}}{|p_1p_2p_3|^\frac12}\phi_4(p_1,p_2)dp_1dp_2. 
$$
Since $\det D^2\Psi=4/3$ and $\Psi=2/3$ at the last stationary point, the application of the stationary phase lemma (Theorem \ref{fase estacionaria}) yields
$$
I_4= A^3e^{2i\eta^2/3+3ia\log\eta}\frac{i\pi \sqrt{3}e^{-3ia\log 3}}{\eta^\frac32} + O(\eta^{-\frac72}).
$$
For $j=5$, that is, away from any stationary point, suppose that $|p_1|\ge |p_2|\ge |p_3|$. Then we must have 
$$|\partial_{p_1}(\eta^2\Psi)|\sim \eta^2||p_1|-|p_3||\gtrsim \eta^2,$$
while
$$
\left|\partial_{p_1}^k\left(\frac{S_{A,a,0}(\eta p_1)S_{A,a,0}(\eta p_2)S_{A,a,0}(\eta p_3)}{(p_1p_2p_3)^\frac{1}{2}}\left(1-\sum_{j=1}^4\phi_j(p_1,p_2)\right)\right)\right|\lesssim |p_1|^{-\frac12 -k}\mathbbm{1}_{|\eta p_1|>1}.
$$
In particular, integrating by parts $k$ times in $p_1$, the resulting integral can be estimated as
$$
\eta^{\frac12}\iint_{|p_2|<|p_1|, |\eta p_1|>1} \frac{1}{\eta^{2k}}|p_1|^{-\frac12 - k}dp_1dp_2 \lesssim \eta^{-k}.
$$
In particular, taking $k$ large, the contribution of the non-stationary region decays as fast as one desires.

\end{proof}

\begin{lemma}\label{lem:IAB}
    For $\eta>1$,
    $$
    \left|\mathcal{I}[S_{A,a,B}] - \frac{3|A|^2A\pi e^{ia\log \eta}}{\eta^\frac32} - e^{2i\eta^2/3+3ia\log\eta}\frac{i\pi A^3\sqrt{3}e^{-3ia\log 3}}{\eta^\frac32}\right|\lesssim \jap{\eta}^{-\frac72}.
    $$

\end{lemma}
\begin{proof}
    By Lemma \ref{lem:IA}, it suffices to show that the terms with at least one $B$ decay at least as $\jap{\eta}^{-\frac72}$. This follows again from a stationary phase analysis, observing that the terms with $B$ possess an extra $\jap{\eta}^{-2}$ weight (and thus the decay will be two orders faster than in the proof of Lemma \ref{lem:IA}). We omit the details.
\end{proof}

\begin{proof}[Proof of Theorem \ref{thm:multi_mbo_2}]
The first part of Theorem \ref{thm:multi_mbo_2} follows from Lemma \ref{lem:IAB}. To prove the contraction estimate \eqref{eq:IA_contr}, one must use the multilinear structure of $\mathcal{I}$ to write
\begin{align*}
    \mathcal{I}[S_{A_1,a_1,B_1}] - \mathcal{I}[S_{A_2,a_2,B_2}] &= \mathcal{I}[S_{A_1,a_1,B_1}, S_{A_1,a_1,B_1}, S_{A_1,a_1,B_1} - S_{A_2,a_2,B_2}] \\&+ \mathcal{I}[S_{A_1,a_1,B_1}, S_{A_2,a_2,B_2}, S_{A_1,a_1,B_1} - S_{A_2,a_2,B_2}]\\&+ \mathcal{I}[S_{A_2,a_2,B_2}, S_{A_2,a_2,B_2}, S_{A_1,a_1,B_1} - S_{A_2,a_2,B_2}].
\end{align*}
Applying the same arguments in the proofs of Lemmas \ref{lem:IA} and \ref{lem:IAB}, one may easily obtain the desired estimate.
\end{proof}

\section{Multilinear estimates for the NLS equation}\label{sec:nls}
\addtocontents{toc}{\protect\setcounter{tocdepth}{1}}
We now estimate the trilinear operator
$$
    \T(\eta):=\mathcal{T}[h_1,h_2,h_3] (\eta)
    = \iint e^{i \Phi}\, h_1(\eta_1) h_2(\eta_2) h_3(\eta_3)\, d\eta_1 d\eta_2,
$$
with the phase function 
$$
\Phi = \eta^2 - \eta_1^2 + \eta_2^2 - \eta_3^2.
$$
\begin{theorem}[Multilinear Estimates I]
    Let $\k_j  \in \left(-\frac12, \frac12\right)$ and $g_j \in Y^{\k_j}$, for $j=1,2,3$  Then the following estimates hold:
    $$\left|\mathcal{T}[h_1,h_2,h_3](\eta)\right| \lesssim \langle\eta\rangle^{-\k_1-\k_2-\k_3} \prod_{j=1}^3 \|h_j\|_{Y^{\k_j}}. $$
\end{theorem}
\begin{proof} To proceed, we divide the proof into several regions:\\

\noindent \textbf{- High $\times$ High $\times$ High: $  1 \ll |\eta_1|, |\eta_2|, |\eta_3|$}.  Analogously to the \eqref{dnls} case, since 
$$
\|f\|_{Y^\kappa} \sim \|f\|_{Z^\kappa} 
$$
for $|\eta|\gg 1$. In this case, we obtain
$$
\left|\int e^{i \Phi}\, h_1(\eta_1) h_2(\eta_2) h_3(\eta_3)\, d\eta_1 d\eta_2 \right| \lesssim \langle \eta \rangle^{-\k_1-\k_2-\kappa_3}.
$$

\noindent \textbf{- High $\times$ Low $\times$ High: $|\eta_2| \lesssim 1 \ll |\eta_1|, |\eta_3|$}
We divide the analysis into two cases.

\noindent \textbf{Case 1: $|\eta_1- \eta_3| \gtrsim  |\eta|$}
Integrating by parts in the variable $\eta_1$, we obtain
$$
\T(\eta)= \iint \frac{e^{i\Phi}}{i(\eta_1-\eta_3)^2} h_1(\eta_1) h_2(\eta_2) h_3(\eta_3)\, d\eta_1 d\eta_2 + \iint \frac{e^{i\Phi}}{2i(\eta_1-\eta_3)} \partial_{\eta_1}[h_1(\eta_1) h_2(\eta_2) h_3(\eta_3)]\, d\eta_1 d\eta_2.
$$
It suffices to estimate the worst term, i.e.
\begin{align*}
\iint \frac{e^{i\Phi}}{2i(\eta_1-\eta_3)} h_1^\prime(\eta_1) h_2(\eta_2) h_3(\eta_3)\, d\eta_1 d\eta_2.
\end{align*}
Since $|\eta_2|\lesssim 1$, without loss of generality we may assume $|\eta_3|\sim |\eta|$, so that
\begin{align*}
\left|\iint \frac{e^{i\Phi}}{2i(\eta_1-\eta_3)} h_1^\prime(\eta_1) h_2(\eta_2) h_3(\eta_3)\, d\eta_1 d\eta_2\right| 
&\lesssim \langle \eta \rangle^{-1} \iint \langle \eta_1\rangle^{-\kappa_1 -1} |\log|\eta_2|| \, \langle \eta_3\rangle^{-\kappa_3}\, d\eta_1 d\eta_2\\
&\lesssim \langle \eta \rangle^{-1-\kappa_3} \iint \langle \eta_1\rangle^{-\kappa_1 -1} |\log|\eta_2||\, d\eta_1 d\eta_2 \\
&\lesssim \langle \eta \rangle^{-1-\kappa_3}\\
&\lesssim \langle \eta \rangle^{-\k_1-\k_2-\kappa_3}.
\end{align*}

\noindent \textbf{Case 2: $|\eta_1- \eta_3| \ll  |\eta|$}
In this case we have $|\eta_1|\sim |\eta_3| \sim |\eta|/2$. We split into two subcases:

\medskip
\noindent\textbf{Case 2.1:} $|\eta_1-\eta_3| \geq |\eta|^{-1}$.  
Integrating by parts in $\eta_1$ and, once again, estimating the worst–case term, we obtain
\begin{align*}
\left|\iint \frac{e^{i\Phi}}{2i(\eta_1-\eta_3)} h_1^\prime(\eta_1) h_2(\eta_2) h_3(\eta_3)\, d\eta_1 d\eta_2\right|
&\lesssim \langle \eta \rangle^{-1-\kappa_1-\kappa_3} \iint |\eta_1-\eta_3|^{-1} |\log|\eta_2||\, d\eta_1 d\eta_2\\
&\lesssim \langle \eta \rangle^{-1-\kappa_1-\kappa_3}\int |\eta_1-\eta_3|^{-1}\, d\eta_1 \\
&\lesssim \langle \eta \rangle^{1-\kappa_1-\kappa_3} |\log|\eta||\\
&\lesssim \langle \eta \rangle^{-1^--\k_2-\kappa_3}\\
&\lesssim \langle \eta \rangle^{-\k_1-\k_2-\kappa_3}.
\end{align*}

\noindent\textbf{Case 2.2:} $|\eta_1-\eta_3| \leq |\eta|^{-1}$. In this region we cannot integrate by parts. Thus, by estimating directly,
\begin{align*}
\left|\iint e^{i\Phi} h_1(\eta_1) h_2(\eta_2) h_3(\eta_3)\, d\eta_1 d\eta_2\right|
&\lesssim \langle \eta \rangle^{-\kappa_1-\kappa_2} \iint |\log|\eta_2||\, d\eta_1 d\eta_2 \\
&\lesssim \langle \eta \rangle^{-1-\kappa_1-\kappa_2}\\
&\lesssim \langle \eta \rangle^{-\k_1-\k_2-\kappa_3}.
\end{align*}

\noindent \textbf{- High $\times$ High $\times$ Low: $|\eta_3|\lesssim1 \ll |\eta_1|,|\eta_2|$}
This case is analogous to the previous one, taking $\eta_2$ as the dependent variable.

\noindent \textbf{- Low $\times$ Low $\times$ High: $|\eta_1|,|\eta_2|\lesssim1 \ll |\eta_3|$}
Here we necessarily have $1 \ll |\eta_3|\sim |\eta|$. We split into two subcases.\\

\noindent\textbf{Case 1:} $|\eta_1|\gtrsim |\eta|^{-1}$. Once again, integrating by parts in $\eta_1$ and estimating the worst–case term, we obtain the bound
\begin{align*}
\left|\iint \frac{e^{i\Phi}}{2i(\eta_1-\eta_3)} h_1^\prime(\eta_1) h_2(\eta_2) h_3(\eta_3)\, d\eta_1 d\eta_2\right|
&\lesssim \langle \eta \rangle^{-1-\kappa_3} \iint |\eta_1|^{-1}\log|\eta_2|\, d\eta_1 d\eta_2 \\
&\lesssim \langle \eta \rangle^{-1-\kappa_3} |\log\!{|\eta|}| \\
&\lesssim \langle \eta \rangle^{-\k_1-\k_2-\kappa_3}.
\end{align*}

\noindent\textbf{Case 2:} $|\eta_1|\ll |\eta|^{-1}$. In this region we cannot integrate by parts. Thus, by estimating directly
\begin{align*}
\left|\iint e^{i\Phi} h_1(\eta_1) h_2(\eta_2) h_3(\eta_3)\, d\eta_1 d\eta_2\right|
&\lesssim \langle \eta \rangle^{-\kappa_3} \iint |\log|\eta_1|||\log|\eta_2||\, d\eta_1 d\eta_2 \\
&\lesssim \langle \eta \rangle^{-1-\kappa_3} |\log\!{|\eta|}| \\
&\lesssim \langle \eta \rangle^{-\k_1-\k_2-\kappa_3}.
\end{align*}

\noindent \textbf{- High $\times$ Low $\times$ Low: $|\eta_2|,|\eta_3|\lesssim1 \ll |\eta_1|$}
This case is analogous to the previous one.

\noindent \textbf{- Low $\times$ High $\times$ Low: $|\eta_1|,|\eta_3|\lesssim1 \ll |\eta_2|$}
Here we have $1 \ll |\eta_2|\sim |\eta|$ and $|\eta_2+\eta_3|\sim |\eta|\gg 1$.\\

\noindent\textbf{Case 1:} $|\eta_3|\geq |\eta|^{-1}$. Integrating by parts in $\eta_{2}$, and since $\eta_{3}$ is the small variable, the worst term is given by 

\begin{align*}
\left|\iint \frac{e^{i\Phi}}{2i(\eta_2+\eta_3)} h_1(\eta_1) h_2(\eta_2) h_3^\prime(\eta_3)\, d\eta_1 d\eta_2\right|
&\lesssim \langle \eta \rangle^{-1-\kappa_2} \iint |\log|\eta_1||\,|\eta_3|^{-1}\, d\eta_3 d\eta_1 \\
&\lesssim \langle \eta \rangle^{-1-\kappa_2} |\log\!{|\eta|}|\\
&\lesssim \langle \eta \rangle^{-\k_1-\k_2-\kappa_3}.
\end{align*}

\noindent\textbf{Case 2:} $|\eta_3|\ll |\eta|^{-1}$.  
Estimating directly,
\begin{align*}
\left|\int e^{i \Phi} h_1(\eta_1) h_2(\eta_2) h_3(\eta_3)\, d\eta_3 d\eta_2\right|
&\lesssim \langle \eta \rangle^{-\kappa_2} \int |\log|\eta_1|||\log|\eta_3||\, d\eta_1 d\eta_3 \\
&\lesssim \langle \eta \rangle^{-1-\kappa_2} |\log\!{|\eta|}| \\
&\lesssim \langle \eta \rangle^{-\k_1-\k_2-\kappa_3}.
\end{align*}

\noindent \textbf{- Low $\times$ Low $\times$ Low: $|\eta_1|,|\eta_2|,|\eta_3|\lesssim1$}
In this case, $|\eta|\lesssim |\eta_1|+|\eta_2|+|\eta_3|\lesssim 1$, and therefore
$$
\left|\iint e^{i\Phi} h_1(\eta_1) h_2(\eta_2) h_3(\eta_3)\, d\eta_1 d\eta_2\right| \lesssim 1.
$$
\end{proof}

\noindent \textbf{Acknowledgments:} The third author, T. S. R. Santos, would like to thank FAPESP and CAPES–PDSE for its support. The third author also gratefully acknowledges the Instituto Superior Técnico, Universidade de Lisboa, for its hospitality during the period in which part of this work was carried out.

\bigskip

\noindent \textbf{Data Availability Statement:} This manuscript has no associated data.

\bibliographystyle{alpha}
\bibliography{Biblio}

\end{document}